\documentclass[11pt,a4paper]{article}
\usepackage{graphicx} % Required for inserting images
\usepackage{amsfonts,amsmath,amssymb,amsthm,mathtools}

\usepackage[T1]{fontenc}
\usepackage[english]{babel}
\usepackage[utf8]{inputenc}

\usepackage{geometry}
\usepackage[hyphens]{url}
\usepackage[linktoc=all,hidelinks,colorlinks,unicode=true]{hyperref} % Must be loaded after url
\hypersetup{linkcolor={blue!70!black},citecolor={green!70!black},urlcolor={purple!70!black}}

\usepackage{thm-restate}

\usepackage{tikz}
\usetikzlibrary{calc}
\usetikzlibrary{positioning}
\usetikzlibrary{nfold}
\makeatletter
\tikzset{
  side by side/.style args={#1:#2}{
    postaction={path only,draw=#1,offset=+.5\pgflinewidth},
    postaction={path only,draw=#2,offset=+-.5\pgflinewidth}},
  side by side'/.style={path only,side by side={#1}},
  offset/.code=
    \tikz@addoption{%
      \pgfgetpath\tikz@temp
      \pgfsetpath\pgfutil@empty
      \pgfoffsetpath\tikz@temp{#1}}}
\makeatother

\tikzset{
    dashone/.style={dash pattern=on 7pt off 7pt},
  }
\usepackage[dvipsnames]{xcolor} % Colours with names -- see https://www.overleaf.com/learn/latex/Using_colours_in_LaTeX for list

\usepackage{subcaption}

\usepackage[textsize=scriptsize]{todonotes}
\newtheorem{theorem}{Theorem}[section]
\newtheorem{lemma}[theorem]{Lemma}
\newtheorem{claim}[theorem]{Claim}
\newtheorem*{claim*}{Claim}
\newtheorem{conjecture}{Conjecture}[section]
\newtheorem{corollary}[theorem]{Corollary}

\theoremstyle{definition}

\newtheorem{problem}[conjecture]{Problem}
\newenvironment{poc}{\begin{proof}[Proof of
    Claim]}{\end{proof}}

\newcommand{\eps}{\varepsilon}

\newcommand{\norm}[1]{\left\lVert#1\right\rVert}
\DeclareMathOperator{\cp}{cp}
\DeclareMathOperator{\Prob}{Prob}
\DeclareMathOperator{\perm}{Perm}
\DeclareMathOperator{\mix}{mix}

\newcommand{\colora}{Goldenrod}
\newcommand{\colorb}{SkyBlue}

\newcommand{\colord}{orange}

\definecolor{color1}{rgb}{.8,.1,.1}
\definecolor{color2}{rgb}{.1,.8,.1}
\definecolor{color3}{rgb}{.1,.1,.8}

\newcommand{\defin}[1]{\emph{\textcolor{ForestGreen}{#1}}}

\title{Dirac's theorem and the switch geometry of perfect matchings}
\author{Ross J. Kang\footnotemark[1] \and Clément Legrand-Duchesne\footnotemark[2]}
\date{\today}

\begin{document}
\maketitle
\renewcommand{\thefootnote}{\fnsymbol{footnote}} % Make affiliation marks symbols

\footnotetext[1]{Korteweg--de Vries Institute for Mathematics, University of Amsterdam, Netherlands (\textsf{\href{mailto:r.kang@uva.nl}{r.kang@uva.nl}}).}
\footnotetext[2]{Theoretical Computer Science Department, Faculty of Mathematics and Computer Science, Jagiellonian University, Kraków, Poland
(\textsf{\href{mailto:clement.legrand-duchesne@uj.edu.pl}{clement.legrand-duchesne@uj.edu.pl}}).}
\renewcommand{\thefootnote}{\arabic{footnote}} % Return to normal footnote symbols

\begin{abstract}
Let $G$ be a graph on an even number $n$ of vertices and let ${\cal M}_G$ be the
collection of perfect matchings in $G$. Dirac's theorem says that if the minimum
degree $\delta(G)$ of $G$ is at least $n/2$, then ${\cal M}_G$ is guaranteed to
be non-empty, while this is not necessarily the case if $\delta(G) \le n/2-1$.

Given an integer $k\ge 2$, let $\mathcal H_k(G)$ be the reconfiguration graph
formed on ${\cal M}_G$ by connecting two distinct $M_1,M_2\in {\cal M}_G$ by an
edge in $\mathcal H_k(G)$ if $M_1$ can be obtained from $M_2$ by switching at most $k$ edges.
Besides non-emptiness, as per Dirac's theorem, what  other natural properties of $\mathcal H_k(G)$ are guaranteed based on the minimum degree $\delta(G)$ of $G$?

We show that if $\delta(G) \ge \lfloor2n/3\rfloor+1$, then $\mathcal H_2(G)$ must be connected and an
expander, while for each $\delta\le \lfloor(2n-2)/3\rfloor$ there are $n$-vertex graphs
$G$ with minimum degree $\delta$ such that $\mathcal H_2(G)$ is disconnected. We also show that, if $\delta(G)
\ge n/2+2$, then $\mathcal H_3(G)$ must be connected and an expander, while for each
$\delta\le n/2-C_k$ there are $n$-vertex graphs $G$ with minimum degree $\delta$ such that $\mathcal H_k(G)$ is
disconnected, for some $C_k$ depending on $k\ge 3$.  Furthermore, for every $\eps >0$, there exists a $c>1$ such that for every
$k\ge 2$ and every large enough $n$, there are $n$-vertex graphs $G$ with $\delta(G) \ge \frac{n}2-\eps kn$ such that $\mathcal H_k(G)$
has at least $c^n$ components.

With respect to guaranteeing that $\mathcal H_k(G)$ has positive minimum degree (or, equivalently, no isolated vertices) we show that if $\delta(G) \ge
n/2+1$, then $\mathcal H_2(G)$ must have positive minimum degree. For $k\ge 3$, we show how this threshold for $\delta(G)$ is related to the notorious Caccetta-Häggkvist conjecture.

Additionally, we show analogous results for $G$ a balanced bipartite graph on $2n$
vertices.
Moreover, most of our results extend naturally to matchings of a prescribed size (instead of perfect matchings), as well as to Ore-type conditions (instead of minimum degree conditions).
\end{abstract}

%Keywords: Dirac's theorem, perfect matching, combinatorial reconfiguration, phase transition, Caccetta-Häggkvist conjecture
%MSC2020 codes: 05C70, 05D15, 05C35, 68R10, 68W20

%\tableofcontents

\section{Introduction}

Let $G$ be a graph on an even number $n$ of vertices.  We are interested in the
macroscopic changes in the space of perfect matchings of $G$
according to a natural metric, as we vary the minimum degree $\delta(G)$ of
$G$.

It is already a fundamental result in graph theory that the space of perfect
matchings undergoes some sort of phase transition at $\delta(G)$ around $n/2$.
Dirac's theorem~\cite{Dir52} states that if $\delta(G) \ge n/2$, then $G$ must
always contain a perfect matching, while if $\delta(G) < n/2$, then there are examples
of $n$-vertex $G$ with no perfect matching.  Something striking suddenly occurs at this juncture. Letting $\Phi(G)$ denote the {\em number} of perfect matchings of
$G$, Cuckler and Kahn~\cite{CuKa09a,CuKa09b} (see also the
results of Sárközy, Selkow, and Szemerédi~\cite{SSS03} and
Br\`egman~\cite{Bre73}) showed that provided $\delta(G) \ge n/2$,
\begin{equation}\label{eqn:delta} \Phi(G) \ge
\left(\frac{\delta(G)}{e+o_{n\to\infty}(1)}\right)^{n/2}.
\end{equation}
Given that there are all of a sudden so many perfect matchings, right at the non-emptiness
threshold of $\delta(G) = n/2$, a  crude ``volume'' heuristic suggests they might ``cluster''.

To test this hypothesis, we consider the collection ${\cal M}_G$ of
perfect matchings in $G$ as a state space. We define transitions between states
according to their proximity in some metric. A natural metric on ${\cal M}_G$ is
the Hamming distance, that is, the size of the symmetric difference. Due to the
definition of perfect matchings, this distance parameter can only take values in
even integers at least $4$. Two perfect matchings $M_1, M_2 \in {\cal M}_G$ have
Hamming distance at most $2k$, $k\ge 2$, if and only if there is some {\em
  $k$-switch} between $M_1$ and $M_2$ ---that is, the addition of $j$ edges and
the removal of $j$ other edges that takes $M_1$ to $M_2$ for some $j \le k$. Let
$\mathcal H_k(G)$ be the state transition graph  on %(also called the reconfiguration graph) on
${\cal M}_G$ defined by joining two distinct $M_1,M_2 \in {\cal M}_G$ if $M_1$
and $M_2$ have Hamming distance at most $2k$. We call $\mathcal H_k(G)$ the {\em
  $k$-switch (reconfiguration) graph} on ${\cal M}_G$.

Concretely, we cast the ``clustering'' question above as follows. What is the
smallest condition on the minimum degree $\delta(G)$ of $G$ in terms of $n$ that guarantees
the connectedness of the $k$-switch graph $\mathcal H_k(G)$? Or a component in $\mathcal H_k(G)$ of size a constant proportion of $\Phi(G)$? Or positive minimum degree in $\mathcal H_k(G)$?  Observe that if the minimum degree of
$G$ is $n-1$, then $G$ must be the complete graph $K_n$ and it is then easy to
see that $\mathcal H_k(G)$ is connected.  Observe also that if $G$ is the disjoint union
of $n/(2k+2)$ cycles of length $2k+2$ (for $n$ a multiple of $2k+2$), then
$\mathcal H_k(G)$ consists of $2^{n/(2k+2)}$ isolated vertices.  Thus, by monotonicity, the above questions
are all well-defined and non-trivial.

Could we obtain some ``phase diagram'' for $\mathcal H_k(G)$ (see
Figure~\ref{fig:phasediagramcartoon}) similar to those observed in 
various random discrete systems (see for example~\cite{Krzakala2007Gibbs,AcCo08})?
In particular, does something indeed special happen
around the non-emptiness threshold at $\delta(G) = n/2$?  Our main objective in
this work is to sketch the first outlines for such a diagram. Our work follows in the
line of other recent studies in ``emergent
combinatorics''~\cite{feghali2016Brooks,BBP21,bousquet2025Colorings,BHK+,BKO25,GaKa25,CCCHK25+}
where analogous phenomena are explored in other combinatorial contexts.

\begin{figure}
  \centering
  \begin{tikzpicture}[scale=.9]
    \foreach \i in {1,...,6}{
      \node (P\i) at (-2*\i,0) {\includegraphics[width=1.5cm]{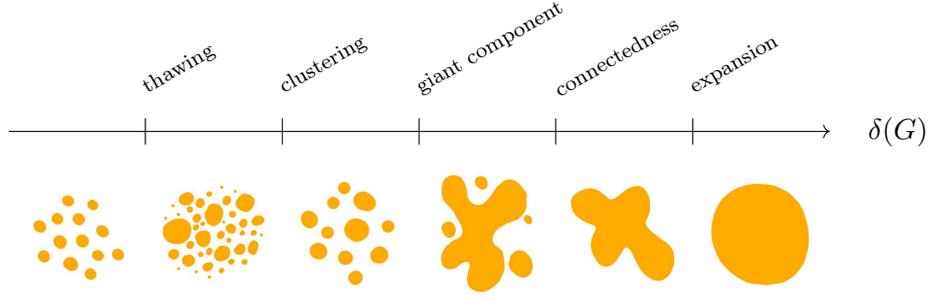}};
    }
    \foreach \i in {1,...,5}{
      \node (t\i) at (-2*\i-1,1.5) {\small $|$};
    }
    \node[above=-1pt of t1, rotate=30,anchor = south west] {\scriptsize expansion};
    \node[above=-1pt of t2, rotate=30,anchor = south west] {\scriptsize connectedness};
    \node[above=-1pt of t3, rotate=30,anchor = south west] {\scriptsize giant component};
    \node[above=-1pt of t4, rotate=30,anchor = south west] {\scriptsize clustering};
    \node[above=-1pt of t5, rotate=30,anchor = south west] {\scriptsize thawing};

    \draw (-13,1.5) edge[->] (-1,1.5);
    \node at (0,1.5) {$\delta(G)$};
  \end{tikzpicture}
  \caption{An artist's depiction of a possible phase diagram for the guaranteed
structure of $\mathcal H_k(G)$ as we vary the condition on $\delta(G)$ as a function of
$n=|G|$. \label{fig:phasediagramcartoon}}
\end{figure}

Note that $\mathcal H_j(G) \subseteq \mathcal H_k(G)$ if $j\le k$. Keeping this monotonicity in mind, we will soon see how $k=3$ is mostly representative for any larger
$k$. We discuss the finer subtleties later in the article, in addition to the
results we have obtained for mixing time, and for matchings of some prescribed
size $\gamma n/2$, where $0 \le \gamma \le 1$. Here are some of our main results.

\begin{theorem}\label{thm:k=2} Let $G$ be a graph on an
  even number $n$ of vertices.  If $G$ has minimum degree at most $n/2-2$, then
  the $2$-switch graph $\mathcal H_2(G)$ may have isolated vertices. If $G$ has
  minimum degree at least $n/2+1$, then $\mathcal H_2(G)$ has positive minimum
  degree.\\
  If $G$ has minimum degree at most $\lfloor (2n-2)/3\rfloor$, then $\mathcal H_2(G)$ may be disconnected. If $G$ has minimum degree at least $\lfloor 2n/3\rfloor+1$, then $\mathcal H_2(G)$ is connected.
\end{theorem}

\begin{theorem}\label{thm:k=3} Let $G$ be a graph on an even number
$n$ of vertices. If $G$ has minimum degree at most $n/2-1$, then the $3$-switch graph $\mathcal H_3(G)$ may be disconnected. If $G$ has minimum degree at least $n/2+2$, then $\mathcal H_3(G)$ is connected.
\end{theorem}

\begin{theorem}\label{thm:k>=2} Let $G$ be a graph on an even number
  $n$ of vertices. For every $\eps >1$, there is some $c >1$ such that the
  following holds for any $k \ge 2$ and $n$. If $G$ has minimum degree at most
  $n/2- k - \eps kn$, then the $k$-switch graph $\mathcal H_k(G)$ may have
  $c^n$ components such that, within each of these components, all matchings share a
  linear number of {\em frozen} edges.
\end{theorem}

\begin{table}
  \centering
  \begin{tabular}{|c|c|c|c|}
    \hline
    {\small Threshold} & $k=2$ & $k=3$ & $k \ge 4$\\
    \hline
    {\small Thawing} & $\left[\frac{n}2- 2 - o(n), \frac{n}2+1\right]$ & $\left[\frac{n}2-3-o(n), \frac{n}2+2\right]$ & $\left[\frac{n}2-k-o(n), \frac{n}2+1\right]$ \\
    {\small Clustering} & $\left[\frac{n}2-2-o(n), \left\lfloor
                    \frac{2n}3 \right\rfloor+1\right]$ & $\left[\frac{n}2-3-o(n), \frac{n}2+2\right]$ & $\left[\frac{n}2-k-o(kn), \frac{n}2+1\right]$\\
    {\small Giant comp.} & $\left[\frac{n}2-2-o(1), \left\lfloor
                    \frac{2n}3 \right\rfloor+1\right]$ & $\left[\frac{n}2-3-o(1), \frac{n}2+2\right]$ & $\left[\frac{n}2-k-o(k), \frac{n}2+1\right]$ \\
    {\small No isolated vts.}& $\left[\frac{n}2-1,\frac{n}2+1\right]$ & $\le \lceil0.3465 n\rceil +1$& {\small Conjecture:} $\frac{n}k +O(1)$%\todo[backgroundcolor=blue]{Ross:put something here?} 
    \\
    {\small Connectedness} & $\left[\left\lfloor \frac{2n+1}3 \right\rfloor,\left\lfloor
                    \frac{2n}3 \right\rfloor+1\right]$&$\left[\frac{n}2,
                                                        \frac{n}2+2\right]$ & $\left[ \frac{n}2 -k,\frac{n}2+1\right]$ \\
    {\small Expansion} & $\left[\left\lfloor \frac{2n+1}3 \right\rfloor,\left\lfloor
                    \frac{2n}3 \right\rfloor+1\right]$&$\left[\frac{n}2,
                                                        \frac{n}2+2\right]$ & $\left[ \frac{n}2-k,\frac{n}2+1\right]$ \\
    \hline
\end{tabular}
    \caption{(general case) Bounds on thresholds for the reconfiguration geometry
      of perfect matchings in terms of minimum degree (see
      Subsection~\ref{ssec:thresholds} for how to interpret these thresholds and how to interpret these bounds). 
      \label{tab:thresholds}}
\end{table}

Informally, we can recast these results as follows (see also
Table~\ref{tab:thresholds}). If we only allow $2$-switches, then the reconfiguration graph
$\mathcal H_2(G)$ on the space ${\cal M}_G$ of perfect matchings in $G$ exhibits the following phase transitions in
its component structure. Until a little before the non-emptiness threshold
$\delta(G)\ge n/2$, the graph $\mathcal H_2(G)$ can have isolated vertices and can be
``shattered'' into exponentially many components, each having a linear number in common of frozen
edges, that is, edges belonging to all matchings of the
component. Almost right after the non-emptiness threshold, $\mathcal H_2(G)$ must have
positive minimum degree (so no isolated vertices). Markedly past the
non-emptiness threshold, at around $\delta(G)=2n/3$, there is a threshold in the
connectivity of $\mathcal H_2(G)$: just below $2n/3$ it can be disconnected, while just
above $2n/3$ the space must be connected. Note that between $n/2$ and $2n/3$ it
remains open as to whether some ``large'' component in $\mathcal H_2(G)$ is guaranteed.
We posit that this occurs closer to $n/2$ and pose this explicitly as an open problem (see Conjecture~\ref{conj:giant} in Section~\ref{sec:open}).

If rather than only $2$-switches we allow $k$-switches for some fixed $k\ge 3$,
then the picture narrows somewhat. Until a little before the non-emptiness threshold, the
switch graph $\mathcal H_k(G)$ can be shattered into exponentially many components with a linear
number of frozen edges. Almost right at the non-emptiness threshold, $\mathcal H_k(G)$
goes from potentially having no linear size component to necessarily being connected.

One of the motivations for proving that some reconfiguration graph is connected is
to ensure that Markov chain Monte Carlo algorithms that sample from the corresponding space converge toward the desired probability distributions. However,
connectivity alone is not enough to ensure that these algorithms are
efficient: the rate of convergence of these Markov chains 
relates to the structural expansion properties of the
reconfiguration graph. We additionally prove that $\mathcal H_k(G)$ becomes a weak expander right at
the connectivity threshold 
and that one can sample in
polynomial time approximately uniform random perfect matchings by iteratively
applying random switches of bounded size (see
Subsection~\ref{ssec:sampling} for more precise details).

In the discussion above, we left incomplete the narrative with respect to positive minimum degree (or equivalently, the possible presence of isolated vertices) in $\mathcal H_k(G)$.
Perhaps contrary to intuition, for $k\ge 3$ this threshold is more subtle. It is related to a longstanding open problem.

\begin{theorem}\label{thm:caccettahaggkvistlink}
    For all $k\ge 2$, $d$ and even $n$, if there is a graph $G$ on $n$ vertices with $\delta(G) = d$ such that $\mathcal H_k(G)$ contains an isolated vertex,
    then there exists an oriented $n$-vertex graph with minimum outdegree $d-1$ and no directed cycle of length at most $k$.
\end{theorem}

The optimal choice of $d$ as a function of $n$ in the conclusion of Theorem~\ref{thm:caccettahaggkvistlink} is the object of the notorious Caccetta-Häggkvist conjecture~\cite{caccetta1978Minimal}. The special case $k=3$ has drawn special attention, and the state of the art there, due to Hladký, Král', and Norin~\cite{hladky2017Counting} using flag algebra, along with Theorem~\ref{thm:caccettahaggkvistlink}, implies the following.

\begin{corollary}\label{cor:caccettahaggkvist}
Let $G$ be a graph on an even number $n$ of vertices.
If $G$ has minimum degree strictly larger than $\lceil0.3465 n\rceil$, then $\mathcal H_3(G)$ contains no isolated vertices.
\end{corollary}

By Theorem~\ref{thm:caccettahaggkvistlink}, the hypothetical truth of the Caccetta-Häggkvist conjecture would then suggest $\lceil n/k\rceil +1$ as an upper bound on the threshold minimum degree of $G$ for positive minimum degree in the $k$-switch graph $\mathcal H_k(G)$. We indeed suspect that $\lceil n/k\rceil +1$ is the correct threshold in this case and highlight this explicitly as an open problem (see Conjecture~\ref{conj:isolated} in Section~\ref{sec:open}).

The special case of $G$ being a balanced bipartite graph on $2n$ vertices, since the number of perfect matchings in this case encodes the permanent, is of particular importance.
(We discuss further context for this below.)
Similarly to~\eqref{eqn:delta}, Cuckler and Kahn~\cite{CuKa09a,CuKa09b} also showed in this
case that provided $\delta(G) \ge n/2$,
\begin{align}\label{eqn:bipartite,delta} \Phi(G) \ge
\left(\frac{\delta(G)}{e+o_{n\to\infty}(1)}\right)^n.
\end{align}
We have established an almost completely analogous set of results for $\mathcal H_k(G)$ in this
special case for $G$, often with slightly sharper bounds (see Table~\ref{tab:thresholds_bip}).
\begin{table}
  \centering
  \begin{tabular}{|c|c|c|c|}
    \hline
    {\small Threshold} & $k=2$ & $k=3$ & $k \ge 4$\\
    \hline
    {\small Thawing} & $\left[\frac{n- 1}2- o(n), \left\lfloor\frac{n+3}2\right\rfloor\right]$ & $\left[\frac{n-2}2- o(n), \left\lfloor\frac{n}2\right\rfloor+1\right]$ & $\left[\frac{n-k+1}2- o(kn), \left\lfloor\frac{n+1}2\right\rfloor\right]$ \\
    {\small Clustering} & $\left[\frac{n-1}2-o(n), \left\lfloor
                    \frac{2n}3 \right\rfloor+1\right]$ & $\left[\frac{n-2}2- o(n), \left\lfloor\frac{n}2\right\rfloor+1\right]$ & $\left[\frac{n-k+1}2-o(n), \left\lfloor\frac{n+1}2\right\rfloor\right]$\\
    {\small Giant comp.} & $\left[\frac{n-1}2-o(1), \left\lfloor
                    \frac{2n}3 \right\rfloor+1\right]$ & $\left[\frac{n-2}2-o(1), \left\lfloor\frac{n}2\right\rfloor+1\right]$ & $\left[\frac{n-k+1}2-o(k),\left\lfloor\frac{n+1}2\right\rfloor\right]$ \\
    {\small No isolated vts.}& $\left\lfloor\frac{n+3}2\right\rfloor$ & $\left[\left\lfloor\frac{n+5}3\right\rfloor,\lceil0.3465 n\rceil +1\right]$ & {\small Conjecture:} $\lceil\frac{n}k\rceil +1$
    \\
    {\small Connectedness} & $\left\lfloor \frac{2n}3 \right\rfloor+1$&$\left\lfloor\frac{n}2\right\rfloor+1$ & $\left[\left\lfloor \frac{n-k+1}2 \right\rfloor,\left\lfloor\frac{n+1}2\right\rfloor\right]$ \\
    {\small Expansion} & $\left\lfloor \frac{2n}3 \right\rfloor+1$&$\left\lfloor\frac{n}2\right\rfloor+1$ & $\left[\left\lfloor \frac{n-k+1}2 \right\rfloor,\left\lfloor\frac{n}2\right\rfloor+1\right]$ \\
    \hline
\end{tabular}
    \caption{(bipartite case) Bounds on thresholds for the reconfiguration geometry
      of perfect matchings in terms of minimum degree, when $G$ is a balanced bipartite graph on $2n$ vertices (see
      Subsection~\ref{ssec:thresholds} for how to interpret these thresholds).
      \label{tab:thresholds_bip}}
\end{table}
An interesting point here is that we can actually establish {\em equivalence} between a variant of the Caccetta-Häggkvist conjecture and the minimum degree threshold that guarantees no isolated vertices in the $k$-switch graph of balanced bipartite graphs.
Specifically, we relate this threshold to the weakening of the Caccetta-Häggkvist conjecture that constrains semidegree rather than outdegree; this version of the conjecture has also been well-studied, it being closely related to a yet older conjecture of Behzad, Chartrand, and Wall~\cite{BCW70}.

We remark here also that, both for balanced bipartite graphs and general graphs, most of
our results extend to matchings of a fixed size $\gamma n/2$ for any fixed
$\gamma \in (0,1]$ by replacing $n/2$ by $\gamma n/2$ in the different
thresholds. Moreover, most of our results also generalise to Ore-type conditions, whereby rather than mandating minimum degree, we mandate a minimum sum of the degrees of two distinct non-adjacent vertices. 

As is probably already evident by now, the ``volume'' heuristic we offered earlier is indeed crude.
If $G$ is a balanced bipartite graph on
$2n$ vertices that is moreover {\em $d$-regular}, then an important result of
Schrijver~\cite{Sch98} (see also~\cite{Voo79} for the case $d=3$
and~\cite{Gur08,Csi17} for alternative proofs for general $d\ge3$) gives the
following lower bound on the number of perfect matchings in $G$:
\begin{align*}
  \Phi(G) \ge \left(\frac{(d-1)^{d-1}}{d^{d-2}}\right)^n.
\end{align*}
This asymptotically matches the expression in~\eqref{eqn:bipartite,delta}, 
except that it needs no condition like $d \ge n/2$, and so
{\em there are plenty of perfect matchings in $G$ regardless of the value of
  $d\ge 3$}. 
 On the other hand, if $d = d(n) \le n^{\theta}$ for fixed $\theta < 1/(4k-1)$, it is known~\cite{MWW04} for all sufficiently large $n$ that there is a bipartite $d$-regular graph $G$ of (finite) girth greater than $2k$ on $2n$ vertices, in which case $\mathcal H_k(G)$ consists of $\Phi(G)$ isolated vertices.
 Moreover, several of the disconnectedness constructions in Theorems~\ref{thm:k=2},~\ref{thm:k=3} and \ref{thm:k>=2} can be adapted so that $G$ is regular and bipartite. 
  
\subsection{Random sampling and Markov chains on matchings}\label{ssec:sampling}

The problem of counting the number of perfect matchings of a graph attracted
considerable attention over the years. When $G$ is a balanced bipartite graph,
this number is counted by the permanent of the adjacency matrix $A$ of $G$:
$$\perm(A) = \sum_{\sigma \in \mathfrak{S}_n} \prod_{i=1}^nA_{i\sigma(i)}.$$
Computing the permanent is a $\#P$-hard problem~\cite{valiant1979Complexity},
which motivated the study of approximation methods. To that end,
Broder~\cite{broder1986Hard} showed that the problem of approximately counting
the number of perfect matchings can be reduced to the generation of a random
perfect matching from an approximately uniform distribution. To that end, he
defined a Markov chain operating on perfect and near-perfect matchings (that is
matchings with at most two unmatched vertices), and showed that it mixes rapidly
on balanced bipartite graphs with $2n$ vertices and minimum degree at least $n$. Since the ratio
of near-perfect matchings to perfect matchings is bounded in these dense graphs,
this also gives an efficient randomised approximate sampler of perfect
matchings. Broder's rapid mixing proof relied on a coupling argument but
contained some error and was therefore withdrawn. Fortunately, Jerrum and
Sinclair~\cite{jerrum1989Approximating} offered a correct proof of rapid mixing of
this Markov chain (namely in time $O(n^5\log n)$), this time relying on the
canonical paths method. As a result, one can sample approximately uniform perfect matchings of dense balanced bipartite graphs in time $O(n^7\log n)$. Their proof generalises to general graphs whose ratio of
near-perfect to perfect matchings is polynomial and to arbitrary balanced
bipartite graphs. Under the additional assumption that the bipartite graph is
$\gamma n$-regular, perfect matchings can be sampled exactly uniformly, in time
$O(n^{1.5+.5/\gamma}\log n)$ (see~\cite{huber2006Exact}).

The $2$-switch Markov chain was introduced by Diaconis, Graham and
Holmes~\cite{diaconis2001Statistical}, who proved that this Markov chain is
ergodic for convex balanced bipartite graphs, that is, such that all ones are
consecutive in each row of the adjacency matrix. In~\cite{dyer2017Switch}, Dyer,
Jerrum and Müller observed that convex balanced bipartite graphs are a subclass
of chordal bipartite graphs, and proved that the $2$-switch Markov chain is also
ergodic in this wider class. They studied a hierarchy of hereditary graphs
classes: chain, monotone, biconvex, convex, chordal bipartite graphs and showed
that the $2$-switch Markov chain mixes rapidly on monotone graphs, but
exponentially on biconvex graphs.

A direct consequence of Theorem~\ref{thm:k=2} (respectively Theorem~\ref{thm:k=3}) is that the
$2$-switch Markov chain (respectively the $3$-switch Markov chain) could fail to be
ergodic on an $n$-vertex graphs of minimum degree at most $\lfloor (2n-2)/3
\rfloor$ but must be ergodic if $\delta(G) \ge \lfloor 2n/3
\rfloor+1$ (respectively at most $n/2-1$ and $n/2+2$). Furthermore, we prove that at and
above these connectedness thresholds, the $2$-switch and $3$-switch Markov chains
actually have polynomial mixing time.
\begin{theorem}\label{thm:mixing23}
  Let $G$ be an $n$-vertex graph with $\delta(G) \ge n/2+2$ (or respectively
  $\delta(G) \ge \lfloor 2n/3 \rfloor+1$). Alternatively, let $G$ be a
  balanced bipartite graph on $2n$ vertices with
  $\delta(G) \ge \lfloor n/2\rfloor+1 $ (or respectively
  $\delta(G) \ge \lfloor 2n/3 \rfloor+1$).\\
  The random walk on $\mathcal H_3(G)$ (respectively $\mathcal H_2(G)$) has mixing time polynomial
  in $n$.
\end{theorem}

Although the upper bounds we obtain for the mixing times of the switch Markov chains
are larger than those of Jerrum and Sinclair on the mixing time of
Broder's Markov chain, making it suboptimal for sampling purposes, the switch chains have the merit of operating solely on perfect matchings.

\subsection{Organisation}

We prove in Section~\ref{sec:connectedness} all our upper bounds on the connectedness
thresholds. In Section~\ref{sec:lower}, we prove the lower bounds on the connectedness,
giant component, clustering and thawing thresholds for general graphs and
balanced bipartite graphs.  In Section~\ref{sec:random_sampling} we prove the
polynomial mixing of the switch Markov chain. Section~\ref{sec:isolated} focuses on the
threshold for isolated matchings. Finally, we discuss open questions and
perspectives in Section~\ref{sec:open}.

\section{Preliminaries}
We denote $N(u)$ the neighbourhood a vertex $u$.
Given an even cycle $C$, an \defin{even chord} of $C$ is a chord splitting the
cycle in two even cycles.

\subsection{Markov chains and mixing times}\label{ssec:Markov}
A Markov chain is a memoryless stochastic process over a discrete state space
$\Omega$. More precisely, a sequence of random variable $(X_0, X_1, \dots)$ is a
Markov chain with transition matrix $P$ if for all $t \in \mathbb{N}$ and $x_1,
\dots x_{t+1} \in \Omega$,
\begin{align*}
  \Prob(X_{t+1}  = x_{t+1} | \forall i \in [t], X_i = x_i) 
  &= \Prob(X_{n+1} = x_{n+1}|X_n=x_n)
  = P(x_{n},x_{n+1}).
\end{align*}

All Markov chains considered in this article are over a finite state space
$\Omega$. By representing the distribution of $X_t$ over $\Omega$ as a row
vector $\mu_t$ with $\norm{\mu_t}_1 = 1$ and $i$-th coordinate equal to
$\Prob(X_t = x_i)$, we have $\mu_t = \mu_{t-1}P$ and more generally $\mu_t =
\mu_0P^t$. With a slight abuse of notation, we will write $\mu_t(x) =
\Prob(X_t=x)$ and $\mu_t(A) = \Prob(X_t \in A)$.

A Markov chain is \defin{ergodic} if $\mu_t(x) > 0$ for all $x \in \Omega$, all
initial distribution $\mu_0$ and all large enough $t$. In particular, every
ergodic Markov chain admits a unique stationary distribution $\pi$, that is such
that $\pi = \pi P$. Note that $\pi$ is uniform if $P$ is symmetric. Moreover,
ergodic Markov chains converge towards their stationary distribution:
$\sup_{\mu_0}\norm{\mu_0P^t - \pi}_{TV}$ tends to zero as $t$ tends to infinity,
where $\norm{\mu - \nu}_{TV} = \max_{A \subseteq \Omega} |\mu(A) - \nu(A)|$ is
the total variation distance. The rate of this convergence is measured by the
\defin{mixing time}: $\tau_{\mix} = \min\{t: d(t) < 1/4\}$, where $d(t)=
\sup_{\mu_0}\norm{\mu_0P^t - \pi}_{TV}$. The constant $1/4$ might look
arbitrary, but using any other constant $\eps$ instead affects the mixing time
only by a factor of $\log_2(1/\eps)$.

\subsubsection*{The canonical paths method}
A Markov chain is said to be polynomially mixing if its mixing time is bounded
by a polynomial in $\log(|\Omega|)$.  The mixing time is tied to the expansion of
the state space graph $\mathcal H$ of the Markov chain, that is the graph whose vertex set is
$\Omega$ and edges corresponding to non-zero probability transitions. Several
methods exist to bound the mixing time: couplings, spectral methods and the
canonical paths method. We will use the last of these in this article. Let
$\Gamma=(\gamma_{xy})_{x,y\in \Omega}$, be a set of paths called
\defin{canonical paths}, joining any two states $x$ and $y$ of $\Omega$ in the
 state space graph $\mathcal H$. The congestion of $\Gamma$ is defined as
$$\varrho(\Gamma) = \max_{(a,b) \in E(\mathcal H)} \left\{\frac{\sum\limits_{(x,y)\in
      \cp(a,b)} \pi(x)\pi(y)|\gamma_{x,y}|}{\pi(a) P(a,b)}\right\}.$$ The
congestion measures how evenly the set of canonical paths are distributed, and
whether the resulting trace on the state space graph has good expansion, and thus
bounds the mixing time: $$\tau_{\mix} \le 2\varrho\log(|\Omega|).$$ Therefore, the
core of the canonical paths methods consists in defining an appropriate set of
canonical paths, and bounding its congestion.

\subsubsection*{Expansion of $\mathcal H_k(G)$}
It is well known that rapid mixing of a random walk on a graph is tied to
its expansion properties, as the main obstacle to rapid mixing are bottlenecks,
and expanders graphs have bounded positive bottleneck ratio. There are many
different definitions of expander graphs; we consider here the spectral one, and
more specifically by considering the eigenvalues of its random-walk normalised
Laplacian. Given a transition matrix $P$, let
$$\lambda_\star = \max \{|\lambda| \colon \lambda \text{ is an eigenvalue of } P,
\lambda \neq 1\}.$$ The \defin{spectral gap} of $P$ is defined as
$\gamma_\star = 1-\lambda_\star$. Given some constant $c>0$, we say that a graph
is a \defin{$c$-expander} if the
lazy uniform random walk on $\mathcal H$ has spectral gap at least $c$. The spectral gap
and the mixing time are tied as follows (see
\cite[Theorems 12.5 and 12.4]{levin2017Markov}):
\begin{equation}\label{eq:expander_mixing}
  \Omega\left(\frac1{\gamma_\star}\right) \le \tau_{\mix};
\end{equation}
\begin{equation}\label{eq:mixing_expander}
  \tau_{{\mix}} \le O\left(\frac{\log |{\mathcal H}|}{\gamma_\star}\right).
\end{equation}
In particular, \eqref{eq:expander_mixing} shows that a lazy random walk on a
$c$-expander ${\mathcal H}$ mixes in time $O(\log |{\mathcal H}|/c)$. Conversely,
\eqref{eq:mixing_expander} shows that in the regime of Theorem~\ref{thm:mixing23},
$\mathcal{H}_k(G)$ is a {$\Omega(n^{-d})$-expander} for some fixed $d$. As $n$
is the number of vertices of $G$, by \eqref{eqn:delta}, $\log |\mathcal{H}_k(G)| = \Theta(n\log
n)$, hence $\mathcal{H}_k(G)$ is a $\tilde{\Omega}(1)$-expander when $G$
satisfies the conditions of Theorem~\ref{thm:mixing23}.

\subsection{Matchings}
A matching is \defin{near-perfect} if all vertices but two are matched. 
A matching is called a $\gamma n$-matching if exactly $\gamma n$ vertices are
matched. Given two matchings $M$ and $M'$, their symmetric difference has
maximum degree two. We call any cycle of $M \Delta M'$ an \defin{alternating cycle},
 any maximal path of $M \Delta M'$ an \defin{alternating path}, and any path in $M \Delta M'$ an \defin{alternating subpath}.
 
 A key argument in the analysis of our
Markov chain (and also in the analysis of Broder's Markov chain) is that the
ratio of perfect matchings to near-perfect matchings is polynomially bounded
when the minimum degree of $G$ is sufficiently high.

\begin{lemma}\label{lem:near-perfect}
  Let $G$ be an $n$-vertex graph in which any two non-adjacent vertices $u$ and
  $v$ have degrees that sum to at least $n-1$. Alternatively, let $G$ be a
  balanced bipartite graph on $2n$ vertices such that the degrees of each pair
  of non-adjacent vertices $u$ and $v$ in different halves of the bipartition
  sum up to at least $n$. The number of near-perfect matchings of $G$ is at most
  $n^2$ times the number of perfect matchings of $G$.
\end{lemma}
\begin{proof}
  Let $\gamma$ be the application that maps near-perfect matchings to a couple
  composed of a set of two vertices and a perfect matching, defined as follows. Let $N$
  be a near-perfect matching with non-matched vertices $u$ and $v$. If $u$ and
  $v$ are adjacent in $G$, let $\gamma(N) = (\{u,v\}, N \cup uv)$. If $u$ and
  $v$ are non-adjacent, let $V' = V(G) \setminus \{u,v\}$, let $A$ be the subset
  of vertices of $V'$ matched by $N$ to a neighbour of $v$ and $B = N(u) \cap
  V'$ \footnote{If $G$ is a balanced bipartite graph, let $V' = V_1 \setminus
    \{v\}$, where $V_1$ is the half of the bipartition containing $v$. We then
    have $|A| + |B| = \deg(u) + \deg(v) > n-1 = |V'|$.}. We have $|A| + |B| =
  \deg(u) + \deg(v) > n-2 = |V'|$ so by the pigeonhole principle, there exists $x
  \in A \cap B$. In other words, there exists an edge $xy$ of
  $N$ such that $x$ is adjacent to $u$ and $y$ to $v$. Let
  $\gamma(N) = (\{u,v\}, N \cup \{xu, yv\} \setminus xy$).

  It is straightforward to check that the application $\gamma$ is injective, from
  which the result immediately follows.
\end{proof}

\subsection{Thresholds in the switch graph on matchings}\label{ssec:thresholds}

Here, we would like to set on more solid ground the thresholds of interest to us that we have informally referred to in Tables~\ref{tab:thresholds} and \ref{tab:thresholds_bip}.
Let $\mathcal H$ be a graph. 
(Think of $\mathcal H$ as some instantiation of a switch graph for matchings in some graph $G$.)
Let ${\mathcal P}$ be a graph property.
For technical reasons, we will require that ${\mathcal P}({\mathcal H})$ always holds if $\mathcal H$ is the empty graph. Although there are many other $\mathcal P$ that we might also naturally consider, let us specify (most of) the ones we focus on in this paper.
We consider $c\ge 1$ as some fixed real.
\begin{description}
\item[connect] $\mathcal H$ is connected.
\item[giant] $\mathcal H$ has a component of size at least $|{\mathcal H}|/c$.
\item[noiso] $\mathcal H$ has no isolated vertices.
\item[expand] For some absolute polynomial $P$, the simple random walk on
  $\mathcal H$ mixes in time $P(\log |\mathcal H|)$ (see Subsection~\ref{ssec:Markov} for a
  discussion on the connections between polynomial mixing time and the expansion
  properties of $\mathcal H$).
\end{description}

As alluded to in the introduction, we tune the $\mathcal H$ under consideration by varying the minimum degree condition on the underlying graph $G$.
In particular, we are interested in the least minimum degree for $G$ that guarantees the corresponding $\mathcal H$ to have some given property $\mathcal P$.
More precisely, given an integer $k\ge 2$ and a rational $0\le \gamma\le 1$, for all $n$ such that $\gamma n/2$ is integral and for all $\delta \in (0,n)$, let
${\mathcal H}_k(n,\gamma,\delta)$ be the following collection of graphs:
\[
\{{\mathcal H}_{k}(G,\gamma) \colon |G|=n \text{ and }\delta(G)\ge \delta\},
\]
where ${\mathcal H}_k(G,\gamma)$ denotes the $k$-switch reconfiguration graph on the space ${\cal M}_{G,\gamma}$ of $\gamma|G|$-matchings of $G$.
Given some graph property ${\mathcal P}$ as above, we are interested in the threshold minimum degree to guarantee ${\mathcal P}$ in ${\mathcal H}_{k}(G,\gamma)$, defined as
\[
\delta^{\mathcal P}_k(n,\gamma) = \min\left\{ \delta \colon {\mathcal P}({\mathcal H})\text{ holds for all }{\mathcal H}\in {\mathcal H}_k(n,\gamma,\delta) \right\}.
\]
It is straightforward to see that the minimum is well-defined if and only if ${\mathcal P}({\mathcal H}_k(K_n,\gamma))$ holds.
Note that, for the purposes of bounding threshold functions for ${\mathcal P}$, the technical condition on ${\mathcal P}$ lets us consider only those $G$ for which ${\mathcal H}_k(G,\gamma)$ is non-empty.

Since we only consider thresholds of the above form, it makes sense to consider graph properties $\mathcal P$ specific to ${\mathcal H}\in {\mathcal H}_k(n,\gamma,\delta)$, that is, based possibly on properties of the underlying graph $G$. The followings are two such properties we consider, where again $c\ge 1$ is some fixed real.
Given some edge $e$ of the graph $G$ and some connected subgraph $\mathcal C$ of ${\mathcal H}_k(G,\gamma)$, we say that $e$ is \defin{frozen} if $e \in M$ for all $M\in {\mathcal C}$.
\begin{description}
\item[thaw] In every component of $\mathcal H$, where ${\mathcal H} = {\mathcal H}_k(G,\gamma)$ for some $G$, $\gamma$, with $|G|=n$, the number of non-frozen edges is at least $\gamma n/(2c)$,
\item[cluster] If ${\mathcal H} = {\mathcal H}_k(G,\gamma)$ for some $G$, $\gamma$, with $|G|=n$, then $\mathcal H$ has fewer than $c^{n}$ connected components.
\end{description}

Note that each of the three properties defined above according to some parameter $c\ge 1$ is monotone with respect to $c$: if $c' > c$, then for every $\mathcal{H}$, if $\mathcal{P}_c$ is satisfied on $\mathcal{H}$ then $\mathcal{P}_{c'}$ also holds on $\mathcal{H}$. In particular, the closer $c$ is to 1, the harder these properties are to satisfy.

Note that for convenience, we will use the abbreviated forms ${\mathcal H}_k(n,\delta)$ and $\delta^{\mathcal P}_k(n)$ when $\gamma=1$. We also analogously denote by
$\delta^{\mathcal P}_{k,\mathrm{bip}}(n,\gamma)$ and $\delta^{\mathcal P}_{k,\mathrm{reg}}(n,\gamma)$
the threshold functions where the corresponding $G$ must be a
balanced bipartite graph on $2n$ vertices and a regular balanced bipartite
graph on $2n$ vertices, respectively.

For example, the inequality
$\gamma n/2- k - o(k) \le \delta^{\mathbf{giant}}_{k} (n,\gamma) \le \gamma n/2+1$ should
be read as follows. The first inequality states that for any $\eps > 0$, there exists $c \ge 1$ such that for every integer $n$ such that $\gamma n/2$ is integral, there exists an $n$-vertex graph $G$ with minimum degree at least $\gamma n/2 - k -1 -\eps k$ such that $\mathcal H_{k}(G,\gamma)$ contains no component of size at least
$|\mathcal H_{k}(G,\gamma)|/c$. The second inequality states that for all $c \ge 1$ and every
$n$-vertex graph $G$ with minimum degree at least $\gamma n/2+1$,
$\mathcal H_{k}(G,\gamma)$ contains a component of size at least
$|\mathcal H_{k}(G,\gamma)|/c$.

\section{Upper bounds on the connectedness thresholds} \label{sec:connectedness}
In this section, we show the upper bounds on the connectedness
thresholds of Theorems~\ref{thm:k=2} and \ref{thm:k=3}. We recast them here.

\begin{theorem}\label{thm:2SwitchMatching}
  Let $G$ be an $n$-vertex graph with $\delta(G) \ge \lfloor 2n/3 \rfloor+1$ (or a balanced
  bipartite graph on $2n$ vertices, with $\delta(G) \ge \lfloor 2n/3 \rfloor+1$). All perfect
  matchings of $G$ are equivalent via $2$-switches.
\end{theorem}

\begin{theorem}\label{thm:3SwitchMatching}
  Let $G$ be an $n$-vertex graph with $\delta(G) \ge n/2+2$ (or a balanced
  bipartite graph on $2n$ vertices, with $\delta(G) \ge \lfloor n/2\rfloor+1$). All perfect
  matchings of $G$ are equivalent via $3$-switches.
\end{theorem}

To do so, we first prove in Subsection~\ref{ssec:4switch} that perfect matchings are
connected by $4$-switches if $\delta \ge n/2+1$, before proving in
Subsection~\ref{ssec:3switch} that any $4$-switch can be realised by a sequence of $O(1)$
$3$-switches if $\delta \ge n/2+2$, and in Subsection~\ref{ssec:2switch} that any $3$-switch can be realised by a sequence of $O(1)$
$2$-switches if $\delta \ge \lfloor2n/3\rfloor+1$, which proves
Theorem~\ref{thm:3SwitchMatching} and Theorem~\ref{thm:2SwitchMatching}, respectively.

\subsection{Equivalence via 4-switches}\label{ssec:4switch}
We first prove that perfect matchings are connected by $4$-switches if $\delta
\ge n/2+1$.

\begin{theorem}\label{thm:4Switch}
  Let $\gamma \in (0,1]$. Let $G$ be an $n$-vertex graph in which each pair of
  non-adjacent vertices have degrees that sum to at least $\gamma n +
  1$. Alternatively, let $G$ be a balanced bipartite graph on $2n$ vertices, in
  which each pair of non-adjacent vertices in different halves of the bipartition
  have their degrees that sum up to at least $\gamma n$. All
  $\gamma n$-matchings of $G$ are equivalent via $4$-switches.
\end{theorem}

\begin{proof}
  We distinguish the case of perfect matchings from that of $\gamma n$-matchings with
  $\gamma<1$, the latter one being slightly more complex because of the short
  alternating paths. 
  \paragraph{Perfect matchings}
  Let $M$ and $M'$ be two perfect matchings of $G$. We proceed by induction on
  $|M \Delta M'|$  by proving that we can always reduce the size of the
  symmetric difference by performing a $4$-switch in $M$ or in $M'$.
  
  If there exists $C=u_1 \dots u_{2p}$ an alternating subpath of $M \Delta M'$ with
  $u_{2i-1}u_{2i} \in M$ and $u_{2i}u_{2i+1} \in M'$, such that
  $p \in \{2,3,4\}$ and $u_1u_{2p} \in E(G)$, then a $4$-switch in $M$ on
  $u_1\dots u_{2p}$ results in a perfect matching $N$ such that
  $|N \Delta M'|\le |M\Delta M'| - p + 1$. Thus one can assume that no such $C$
  exists.

  Let $C = u_1 \dots u_6$ be an alternating subpath of $M\Delta M'$ with
  $u_{2i-1}u_{2i} \in M$ and $u_{2i}u_{2i+1} \in M'$. By assumption, $u_1$ is
  neither adjacent to $u_4$ nor to $u_6$, and $u_6$ is not adjacent to
  $u_3$. Let $V' = V(G) \setminus\{u_1, \dots u_6\}$. Let $A$ be the subset of
  $V'$ composed of vertices matched in $M$ to a neighbour of $u_1$. The only
  vertices of $\{u_1, \dots u_6\}$ that can be matched in $M$ to neighbours of
  $u_1$ are $u_1$, $u_4$ and $u_6$. Therefore, $|A| \ge \deg(u_1)-3$. Letting
  $B = N(u_6) \cap V'$, we also have $|B| \ge \deg(u_6) -3$. By the pigeonhole
  principle\footnote{If $G$ is a balanced bipartite graph on $2n$ vertices, let
    $V' = V_1 \setminus \{u_1, u_3, u_5\}$, where $V_1$ is the half of the
    bipartition that contains $u_1$. We now have $|A| = \deg (u_1) -1$ and
    $|B| = \deg(u_6) -1$, so $|A| + |B| \ge n-2 > |V'|$.}, since $V'$ has size $n-6$
  and is a superset of $A$ and $B$ and $|A|+|B| \ge n-5$, there exists a vertex
  $x$ in $A\cap B$. This vertex $x$ is by definition adjacent to $u_6$, and
  matched to a vertex $y$ in $M$, itself adjacent to $u_1$ (see
  Figure~\ref{fig:4Switch}). Performing the $4$-switch $u_1 \dots u_6xy$ on $M$ results
  in a perfect matching $N$ containing $u_2u_3$ and $u_4u_5$, so
  $|N \Delta M'| \le |M\Delta M'| - 1$.
  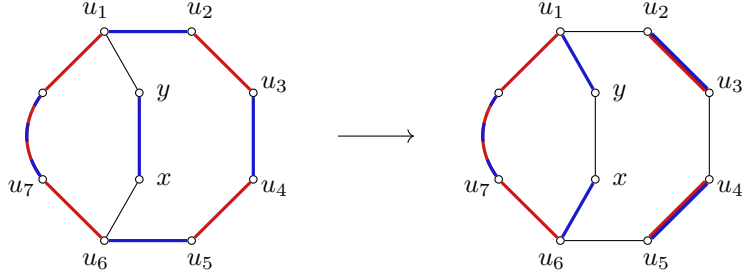
\begin{figure}[ht!]
    \centering        
    \begin{tikzpicture}
      \foreach \i in {1,...,7}{
        \node[circle, draw = black, inner sep = 1pt] (u\i) at
        ($(157.5-\i*45:1.5cm)$) {};
        \node at ($(157.5-\i*45:1.8cm)$) {\small $u_\i$};
      }
      \node[circle, draw = black, inner sep = 1pt] (u8) at (157.5:1.5cm) {};
      %\node at (157.5:1.8cm) {\small $u_{2p}$};

      \foreach \i/\j in {1/2,3/4,5/6}{
        \draw[very thick,color3] (u\i) -- (u\j);
      }
      \foreach \i/\j in {2/3,4/5,6/7}{
        \draw[very thick,color1] (u\i) -- (u\j);
      }
      \draw[color1,very thick] (u8) -- (u1);
      \draw[color1, very thick, postaction={draw, color3, dashone,very thick}] (u7)
      to[out=125,in=-125] (u8);
      
      \node [circle, draw = black, inner sep = 1pt] (y) at ($(22.5:1.5cm)+(-1.5,0)$) {};
      \node [circle, draw = black, inner sep = 1pt] (x) at ($(-22.5:1.5cm)+(-1.5,0)$) {};
      \node[right=1pt of y] {\small $y$}; 
      \node[right=1pt of x] {\small $x$};

      \draw[very thick,color3] (x) -- (y);
      \draw (u1) -- (y) (u6) -- (x);

      \draw[->] (2.5,0) -- (3.5,0);
      \begin{scope}[shift={(6,0)}]
      \foreach \i in {1,...,7}{
        \node[circle, draw = black, inner sep = 1pt] (u\i) at
        ($(157.5-\i*45:1.5cm)$) {};
        \node at ($(157.5-\i*45:1.8cm)$) {\small $u_\i$};
      }
      \node[circle, draw = black, inner sep = 1pt] (u8) at (157.5:1.5cm) {};

      \foreach \i/\j in {1/2,3/4,5/6}{
        \draw (u\i) -- (u\j);
      }
      
      \draw[very thick,color1] (u6) -- (u7)  (u8) -- (u1);
      \draw[color1, very thick, postaction={draw, color3, dashone,very thick}] (u7)
      to[out=125,in=-125] (u8);

      \foreach \i/\j in {2/3,4/5}{
        \draw[very thick,side by side=color3:color1] (u\i) -- (u\j);
      }

      \node [circle, draw = black, inner sep = 1pt] (y) at ($(22.5:1.5cm)+(-1.5,0)$) {};
      \node [circle, draw = black, inner sep = 1pt] (x) at ($(-22.5:1.5cm)+(-1.5,0)$) {};
      \node[right=1pt of y] {\small $y$}; 
      \node[right=1pt of x] {\small $x$};

      \draw (x) -- (y);
      \draw[very thick,color3] (u1) -- (y) (u6) -- (x);
    \end{scope}
    \end{tikzpicture}
    \caption{The reconfiguration sequence of Theorem~\ref{thm:4Switch}. $M$ is
      represented in blue and $M'$ in red.}\label{fig:4Switch}
  \end{figure}
  
  \paragraph{$\gamma n$-matchings}
  Let $M$ and $M'$ be two $\gamma n$-matchings of $G$. We proceed by induction
  on $|M \Delta M'|$ by proving that we can always reduce the size of the
  symmetric difference by performing a $4$-switch in $M$ or in $M'$. We handle
  separately the following cases (in that order): isolated edges of
  $M \Delta M'$, even alternating paths of length at most 6, odd alternating
  subpaths of length 3 or 5 with an edge between their extremities, odd
  alternating subpaths of length 5 with no edges between their extremities and a
  vertex at distance 3 or 5 in the subpath, and odd alternating paths of length
  3 with no edge between their extremities. It is straightforward to check that
  these cases are exhaustive.

  If there exists an edge $xy$ of $M$ such that $x$ and $y$ are both unmatched
  in $M'$, then removing any edge from $M' \setminus M$ and replacing it by $xy$
  yields a $\gamma n$-matching $N$ such that $|N\Delta M| < |M' \Delta M|$. Thus
  one can assume that all the edges of $M$ are incident to edges of $M'$ (and
  likewise that all edges of $M'$ are incident to edges of $M$). Hence, we can
  assume that alternating paths have length at least two.
  
  If there exists an even alternating path $u_1, \dots u_{2p+1}$ with
  $p \in [3]$, such that $u_{2i-1}u_{2i} \in M$ and $u_{2i}u_{2i+1} \in M'$,
  with $u_1$ unmatched in $M'$ and $u_{2p+1}$ unmatched in $M$, then performing
  the $4$-switch on the path $u_1\dots u_{2p+1}$ results in a $\gamma n$-matching $N$ with
  $|N \Delta M'| < |M\Delta M'|$. Thus one can assume $M \Delta M'$ contains no
  even alternating paths of length at most 6.

  Similarly to the perfect matching case, if there exists an alternating subpath
  $u_1, \dots u_{2p}$ with $p \in \{2,3,4\}$, such that $u_{2i-1}u_{2i} \in M$
  and $u_{2i}u_{2i+1} \in M'$, and $u_1u_{2p} \in E(G)$, then performing in $M$ the
  $4$-switch on the cycle $u_1\dots u_{2p}$ results in a $\gamma n$-matching $N$ with
  $|N \Delta M'| < |M\Delta M'|$. Thus one can assume that $M \Delta M'$
  contains no odd alternating subpaths of length 3 or 5 with an edge between
  their extremities.

  Assume that there exists an odd alternating subpath $P=u_1 \dots u_{6}$ of
  length 5 in $M \Delta M'$ with $u_{2i-1}u_{2i} \in M$ and
  $u_{2i}u_{2i+1} \in M'$. By assumption, $u_1$ is neither adjacent to $u_4$,
  nor to $u_6$, and $u_6$ is not adjacent to $u_3$. If $u_1$ has a neighbour $y$
  that is unmatched in $M$, then
  $N = M \cup \{yu_1,u_2u_3, u_4u_5\} \setminus \{u_1u_2,u_3u_4,u_5u_6\}$ is
  a $\gamma n$-matching such that $|N \Delta M'| \le |M \Delta M'|-2$. Likewise,
  if $u_6$ has a neighbour $x$ unmatched in $M$, then
  $N = M \cup \{u_2u_3, u_4u_5, u_6x\} \setminus \{u_1u_2,u_3u_4,u_5u_6\}$
  satisfies the same properties. Thus one can assume that all neighbours of
  $u_1$ and $u_6$ are matched in $M$. Let $V'$ be the subset of vertices of
  $V(G) \setminus \{u_1, \dots u_6\}$ matched in $M$ and $A$ be the subset of
  $V'$ composed of vertices matched in $M$ to a neighbour of $u_1$. The only
  vertices of $\{u_1, \dots u_6\}$ that can be matched in $M$ to neighbours of
  $u_1$ are $u_1$, $u_4$ and $u_6$. Therefore, $|A| \ge \deg(u_1)-3$. Letting
  $B = N(u_6) \cap V'$, we also have $|B| \ge \deg(u_6) -3$. By the pigeonhole
  principle\footnote{If $G$ is a balanced bipartite graph on $2n$ vertices, let
    $V' = V_1 \setminus \{u_1, u_3, u_5\}$, where $V_1$ is the half of the
    bipartition containing $u_1$. We now have $|A| = \deg (u_1) -1$ and
    $|B| = \deg(u_6) -1$, so $|A| + |B| \ge n-2$.}, since $V'$ has size
  $\gamma n -6$ and is a superset of $A$ and $B$ and
  $|A| + |B| \ge \gamma n -4$, there exists a vertex $x$ in $A\cup B$. This
  vertex $x$ is by definition adjacent to $u_6$, and matched to a vertex $y$ in
  $M$, itself adjacent to $u_1$ (see Figure~\ref{fig:4Switch}). Performing the
  $4$-switch $u_1 \dots u_6xy$ on $M$ results in a $\gamma n$-matching $N$
  containing $u_2u_3$ and $u_4u_5$, so $|N \Delta M'| \le |M\Delta M'| - 1$.

  Thus, one can assume that $|M\Delta M'|$ is a disjoint union of alternating
  paths of length 3, each without non-adjacent extremities. Note that exactly
  half of them contain two edges of $M$. Let $P=u_1u_2u_3u_4$ and
  $Q=v_1v_2v_3v_4$ be two alternating paths of $M\Delta M'$, with
  $A=\{u_1u_2, u_3u_4,v_2v_3\} \subseteq M$ and
  $B=\{u_2u_3, v_1u_2,v_3v_4\} \subseteq M'$. Replacing $A$ by $B$ in $M$
  results in a $\gamma n$-matching $N$ with
  $|N\Delta M'| < |N\Delta M'|$.
\end{proof}

\subsection{Equivalence via 3-switches}\label{ssec:3switch}
We now prove Theorem~\ref{thm:3SwitchMatching}, which immediately follows from
Theorem~\ref{thm:4Switch} and the following lemma.

\begin{lemma}\label{lem:3Switch}
  Let $G$ be an $n$-vertex graph in which each pair of non-adjacent vertices have their degrees that sum to at least $n+3$.
  % $\delta(G) \ge n/2+2$
  Alternatively, let $G$ be a balanced bipartite graph on $2n$ vertices, in which each pair of non-adjacent vertices in different halves of the bipartition have their degrees that sum up to at least $n+1$.
  %$(or a balanced bipartite graph on $2n$-vertices, such that $\delta(G) \ge \lfloor (n+2)/2 \rfloor$). 
  Every two perfect matchings of $G$ that differ by one $4$-switch are equivalent via $3$-switches.
\end{lemma}

\begin{proof}
  Let $M$ and $M'$ be two perfect matchings of $G$ such that $M \Delta M'$ is a
  8-cycle $C = u_1\dots u_8$, with $u_{2i-1}u_{2i} \in M$ and
  $u_{2i}u_{2i+1} \in M'$. If $C$ admits an even chords, say $u_iu_j$ with
  $i < j$ and $i$ even, then $M'$ can be obtained from $M$ after a switch on
  $u_j\dots u_8u_1 \dots u_i$, followed by a switch on $u_i\dots u_j$. One of
  these cycles has length 4 and the other length 6, so both are $6$-switches.
  
  Hence one can assume that $C$ has no even chord. 
  We first prove in
  Claim~\ref{cl:3Switch} that the existence of an edge of $M\cap M'$ with certain
  adjacencies implies the desired reconfiguration sequence, and secondly that
  such an edge exists.

  \begin{claim}\label{cl:3Switch}
    Let $xy$ be an edge of $M \cap M'$ such that $y$ is a neighbour of
    $u_1$ and $x$ of $u_2$ and $u_6$. The matchings $M$ and $M'$ are equivalent
    via a sequence of at most three $3$-switches. 
  \end{claim}
  \begin{poc}
    Starting from $M$, the sequence
    of $3$-switches $u_1u_2xy$, $yu_2\dots u_6$ and $yu_6u_7u_8u_1x$ results in
    $M'$ (see Figure~\ref{fig:3Switch}).
    \begin{figure}[ht!]
      \centering        
      \begin{tikzpicture}
        \node[circle, inner sep=1.1cm] (A) at (0,0) {};

        \foreach \i in {1,...,8}{
          \node[circle, draw = black, inner sep = 1pt] (u\i) at
          ($(157.5-\i*45:1cm)$) {};
          \node at ($(157.5-\i*45:1.3cm)$) {\small $u_\i$};
        }

        \foreach \i/\j in {1/2,3/4,5/6,7/8}{
          \draw[very thick,color3] (u\i) -- (u\j);
        }
        \foreach \i/\j in {2/3,4/5,6/7,8/1}{
          \draw[very thick,color1] (u\i) -- (u\j);
        }
        
        \node [circle, draw = black, inner sep = 1pt] (y) at ($(22.5:1cm)+(-1,0)$) {};
        \node [circle, draw = black, inner sep = 1pt] (x) at ($(-22.5:1cm)+(-1,0)$) {};
        \node[left=1pt of y] {\small $y$}; 
        \node[right=1pt of x] {\small $x$};

        \draw[very thick,side by side=color1:color3] (x) -- (y);
        \draw (u1) -- (y)  (u6) -- (x) edge[bend right] (u2);

        \begin{scope}[shift={(4,0)}]
          \node[circle, inner sep=1.1cm] (B) at (0,0) {};

          \foreach \i in {1,...,8}{
            \node[circle, draw = black, inner sep = 1pt] (u\i) at
            ($(157.5-\i*45:1cm)$) {};
            \node at ($(157.5-\i*45:1.3cm)$) {\small $u_\i$};
          }

          \foreach \i/\j in {3/4,5/6,7/8}{
            \draw[very thick,color3] (u\i) -- (u\j);
          }
          \draw (u1) -- (u2);
          \foreach \i/\j in {2/3,4/5,6/7,8/1}{
            \draw[very thick,color1] (u\i) -- (u\j);
          }
          
          \node [circle, draw = black, inner sep = 1pt] (y) at ($(22.5:1cm)+(-1,0)$) {};
          \node [circle, draw = black, inner sep = 1pt] (x) at ($(-22.5:1cm)+(-1,0)$) {};
          \node[left=1pt of y] {\small $y$}; 
          \node[right=1pt of x] {\small $x$};

          \draw[very thick, color1] (x) -- (y);
          \draw[very thick, color3] (u1) -- (y)  (x) edge[bend right] (u2);
          \draw (u6) -- (x);
        \end{scope}
        
        \begin{scope}[shift={(8,0)}]
          \node[circle, inner sep=1.1cm] (C) at (0,0) {};

          \foreach \i in {1,...,8}{
            \node[circle, draw = black, inner sep = 1pt] (u\i) at
            ($(157.5-\i*45:1cm)$) {};
            \node at ($(157.5-\i*45:1.3cm)$) {\small $u_\i$};
          }

          \foreach \i/\j in {1/2,3/4,5/6}{
            \draw (u\i) -- (u\j);
          }
          \draw[very thick,color3] (u7) -- (u8);
          \foreach \i/\j in {6/7,8/1}{
            \draw[very thick,color1] (u\i) -- (u\j);
          }
          \foreach \i/\j in {2/3,4/5}{
            \draw[very thick,side by side=color3:color1] (u\i) -- (u\j);
          }
          
          \node [circle, draw = black, inner sep = 1pt] (y) at ($(22.5:1cm)+(-1,0)$) {};
          \node [circle, draw = black, inner sep = 1pt] (x) at ($(-22.5:1cm)+(-1,0)$) {};
          \node[left=1pt of y] {\small $y$}; 
          \node[right=1pt of x] {\small $x$};

          \draw[very thick, color1] (x) -- (y);
          \draw[very thick, color3] (u1) -- (y)  (u6) -- (x);
          \draw (x) edge[bend right] (u2);
        \end{scope}
        
        \begin{scope}[shift={(12,0)}]
          \node[circle, inner sep=1.1cm] (D) at (0,0) {};

          \foreach \i in {1,...,8}{
            \node[circle, draw = black, inner sep = 1pt] (u\i) at
            ($(157.5-\i*45:1cm)$) {};
            \node at ($(157.5-\i*45:1.3cm)$) {\small $u_\i$};
          }

          \foreach \i/\j in {1/2,3/4,5/6,7/8}{
            \draw (u\i) -- (u\j);
          }
          \foreach \i/\j in {2/3,4/5,6/7,8/1}{
            \draw[very thick,side by side=color3:color1] (u\i) -- (u\j);
          }
          
          \node [circle, draw = black, inner sep = 1pt] (y) at ($(22.5:1cm)+(-1,0)$) {};
          \node [circle, draw = black, inner sep = 1pt] (x) at ($(-22.5:1cm)+(-1,0)$) {};
          \node[left=1pt of y] {\small $y$}; 
          \node[right=1pt of x] {\small $x$};

          \draw[very thick, side by side=color1:color3] (x) -- (y);
          \draw (u1) -- (y)  (u6) -- (x) edge[bend right] (u2);
        \end{scope}
        \draw[->] (A) -> (B);
        \draw[->] (B) -> (C);
        \draw[->] (C) -> (D);

      \end{tikzpicture}
      \caption{The reconfiguration sequence of Lemma~\ref{lem:3Switch}. $M$ is
        represented in blue and $M'$ in red.}\label{fig:3Switch}
    \end{figure}
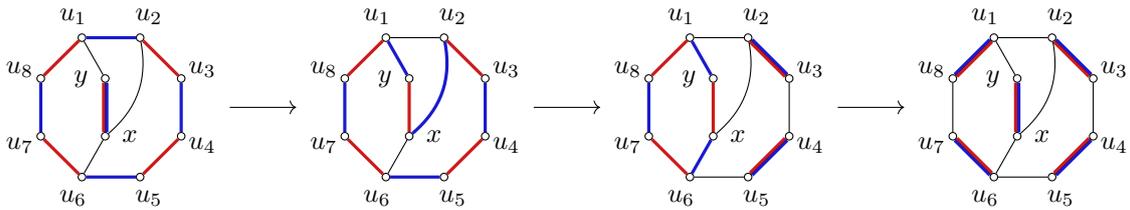
  \end{poc}
  
  \paragraph{Bipartite case}
  For each $i \in [8]$, let $E_i$ be the set of edges
  $xy$ of $M\cap M'$ such that either $x$ or $y$ is adjacent to $u_i$.  We have
  $|E_i| = |N(u_i) \cap (V(G) \setminus V(C))| = |N(u_i) \setminus \{u_{i-1 \pmod
    8}, u_{i+1 \pmod 8}\}|$ because $C$ contains no even chords and $G$ is bipartite. So
  $|E_i| \ge \deg(u_i) - 2$. As $C$ has no even chord, for every $i \in [8]$, $u_i$ and $u_{i+3\pmod 8}$ are non-adjacent and thus $\deg(u_i) + \deg(u_{i+3\pmod 8}) \ge n+1$. Hence, 
  \begin{align*}
      \sum_{i \in [8]} |E_i|  &= \sum_{i \in \{1,3,5,7\}} (\deg(u_i)+\deg(u_{i+3 \pmod 8}) -4)\\
      &\ge 4(n-3) > 4(n -4).
      \end{align*}
  Since $|M \cap M'| = n-4$, a counting
  argument shows that there exists an edge $xy \in M \cap M'$ with endpoints
  adjacent to at least five vertices of $C$. Otherwise, we have
  $\sum_{i \in [8]} |E_i| \le 4|M \cap M'| = 4n -16$, a contradiction.
  
  By the pigeonhole principle, there exists $i \in [4]$
  such that $u_i$ and $u_{i+4}$ are adjacent to some endpoint of $xy$, and since
  $G$ is bipartite, this endpoint is the same for $u_i$ and $u_{i+4}$, say
  $x$. Without loss of generality, say that $i$ is even. As $G$ is bipartite and
  $xy$ is adjacent to at least five vertices of $C$, $y$ is adjacent to at least one
  vertex $u_j$ with $j$ odd. Up to relabelling $i$ and $j$, the premise of
  Claim~\ref{cl:3Switch} is satisfied, which concludes the proof for the bipartite
  case.

  \paragraph{Non-bipartite case}
  For $i \in [8]$, let
  $E_i = \{(x,y) \colon xy \in M\cap M' \text{ and } x \in N(u_i)\}$. Let
  $V' = V(G) \setminus V(C)$.  For each $i \in [8]$, we have
  $|E_i| = |N(u_i) \cap V'| \ge \deg(u_i) -5$ because $C$ has no even
  chords. As $C$ has no even chord, for every $i \in [8]$, $u_i$ and $u_{i+3\pmod 8}$ are non-adjacent and thus $\deg(u_i) + \deg(u_{i+3\pmod 8}) \ge n+3$. Consider the weight function $w$ over all pairs $(x,y)$ such that
  $xy \in M \cap M'$ defined as follows: $w(x,y)$ is the number of indices
  $i \in [8]$ such that $x$ is adjacent to $u_i$ (in particular, note that
  $w(x,y)$ may differ from $w(x,y)$). By double counting, we have
  \begin{align*}
    \sum_{xy \in M\cap M'} w(x,y) + w(y,x) & = \sum_{(x,y) : xy \in M\cap M'}w(x,y)
                                           = \sum_{i \in [8]} |E_i|\\
                                           &\ge \sum_{i \in \{1,3,5,7\}} (\deg(u_i) +\deg(u_{i+3\pmod 8}) - 10)\\
                                           &\ge 4(n-7) > 8(n/2-4).
  \end{align*}
  Since there are $n/2-4$ edges in $M\cap M'$, this implies that there exists
  $xy \in M\cap M'$ with $w(x,y) + w(y,x) > 8$.

  For each $u_i$, let $w'(u_i) = |\{x,y\} \cap N(u_i)|$. We have
  $\sum_{i=1}^8 w'(u_i)= w(x,y) +w(y,x) > 8$. We will shift the weights of $w'$
  towards eight targets $x_1, \dots x_4$ and $y_1, \dots y_4$, to prove that, up
  to relabelling, the adjacencies required by Claim~\ref{cl:3Switch} are satisfied by
  $xy$.

  Consider the following discharging rules:
  \begin{itemize}
  \item For each $i \in [8]$, if $u_i$ is adjacent $x$, then $u_i$ sends weight $1/2$ to
    $x_{i \pmod 4}$, weight $1/4$ to $y_{i+1\pmod 4}$ and weight $1/4$ to
    $y_{i+3\pmod 4}$.
  \item For each $i \in [8]$, if $u_i$ is adjacent $y$, then $u_i$ sends weight $1/2$ to
    $y_{i \pmod 4}$, weight $1/4$ to $x_{i+1\pmod 4}$ and weight $1/4$ to
    $x_{i+3\pmod 4}$.
  \end{itemize}
  After applying the discharging rules, since the total weight of $w'$ is
  greater than 8 and there are eight targets, one of them, say $x_1$ received
  weight greater than 1. Only four vertices (namely $u_i$ for even $i$) could send
  weight 1/4 to $x_1$, hence $x_1$ received a weight of $1/2$ by at least one
  vertex ($u_1$ or $u_5$).

  If $x_1$ received a weight of $1/2$ by exactly one vertex, then it also
  received weight $1/4$ by at least three vertices, thus there exists
  $i \in \{2,4\}$ such that $y$ is adjacent to $u_{i}$ and $u_{i+4}$. Since $x$
  is also adjacent to $u_1$ or $u_5$, the premise of Claim~\ref{cl:3Switch}
  is satisfied up to relabelling.

  If $x_1$ received a weight of $1/2$ by two vertices, then $x$ is adjacent to
  $u_1$ and $u_5$. Since $x_1$ also received a weight of $1/4$ by at least one
  vertex, $y$ is also adjacent to a vertex $u_i$ with $i$ even. Up to
  relabelling, this is again the adjacency pattern required by
  Claim~\ref{cl:3Switch}.
\end{proof}

\subsection{Equivalence via 2-switches}\label{ssec:2switch}
Theorem~\ref{thm:2SwitchMatching} follows from Theorem~\ref{thm:3SwitchMatching} combined with
the following lemma, by observing that $\frac{\lfloor 4n/3\rfloor+1}2 \le \lfloor 2n/3\rfloor+1$.

\begin{lemma}\label{lem:2Switch}
  Let $G$ be an $n$-vertex graph in which each pair of non-adjacent vertices have their degrees that sum to at least $\lfloor 4n/3\rfloor+1$.
  %$\delta(G) \ge \lfloor (2n+3)/3 \rfloor$ 
  Alternatively, let $G$ be a balanced bipartite graph on $2n$ vertices, in which each pair of non-adjacent vertices in different halves of the bipartition have their degrees that sum to at least $\lfloor 4n/3\rfloor+1$.
  % with $\delta(G) \ge \lfloor (2n+3)/3\rfloor$). 
  Every two perfect matchings of $G$ that differ by one $3$-switch are equivalent via $2$-switches.
\end{lemma}
\begin{proof}
  First note that for every integer $d$, $d \ge \lfloor (4n+3)/3 \rfloor$ if and only if $d > 4n/3$. Thus, every pair of non-adjacent vertices (in different halves of the bipartition if $G$ is bipartite) have their degrees that sums up to more than $4n/3$.
  %First note that $d \ge \lfloor (2n+3)/3 \rfloor$ is equivalent to $d > 2n/3$ for all $d \in \mathbb{N}$.
  Let $M$ and $M'$ be two perfect matchings of $G$ such that $M \Delta M'$ is a
  6-cycle $C = u_1, \dots u_6$ with $u_{2i-1}u_{2i} \in M$ and
  $u_{2i}u_{2i+1} \in M'$. Here again, if $C$ contains an even chord
  $u_iu_{i+3}$ for some $i \in [3]$, then it divides $C$ into two 4-cycles, and
  performing a $2$-switch on each of these cycles transforms $M$ into $M'$. Hence
  one can assume that $C$ has no even chord. We first prove in Claim~\ref{cl:2Switch}
  that the existence of an edge of $M\cap M'$ with certain adjacencies implies
  the desired reconfiguration sequence, and secondly that such an edge exists.

  \begin{claim}\label{cl:2Switch}
    If there exists an edge $xy \in M \cap M'$ with $x$ adjacent to $u_1$ and
    $u_5$, and $y$ to $u_2$ and $u_4$, then $M$ is equivalent to $M'$ via a
    sequence of four $2$-switches. 
  \end{claim}
  \begin{poc}
    Starting from $M$, the sequence of $2$-switches $u_1u_2yx$, $yu_2u_3u_4$,
    $u_1xu_5u_6$ and $yu_4u_5x$ results in $M'$ (see Figure~\ref{fig:2Switch}).
    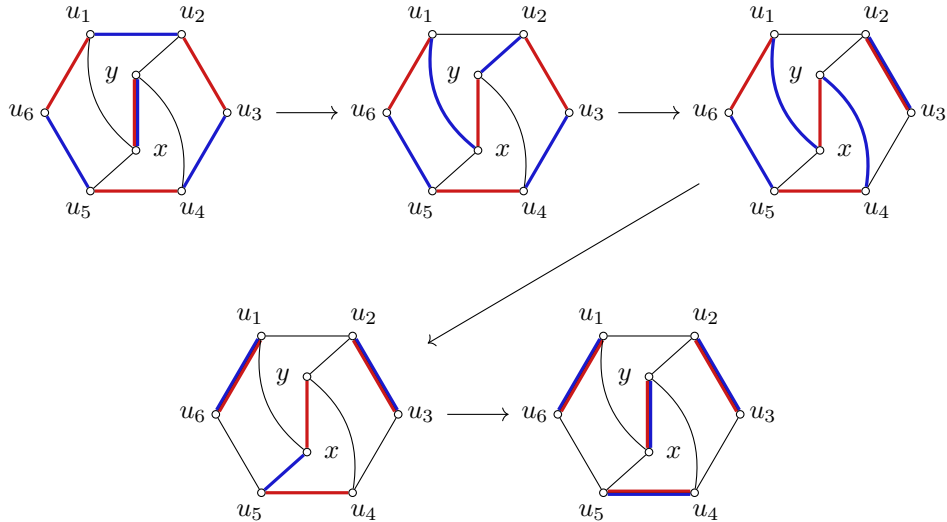
\begin{figure}[ht!]
    \centering        
    \begin{tikzpicture}
      \node[circle, inner sep=1.3cm] (A) at (0,0) {};
      \foreach \i in {1,...,6}{
        \node[circle, draw = black, inner sep = 1pt] (u\i) at
        ($(180-\i*60:1.2cm)$) {};
        \node at ($(180-\i*60:1.5cm)$) {\small $u_\i$};
      }

      \foreach \i/\j in {1/2,3/4,5/6}{
        \draw[very thick,color3] (u\i) -- (u\j);
      }
      \foreach \i/\j in {2/3,4/5,6/1}{
        \draw[very thick,color1] (u\i) -- (u\j);
      }
      
      \node [circle, draw = black, inner sep = 1pt] (y) at (0,.5) {};
      \node [circle, draw = black, inner sep = 1pt] (x) at (0,-.5) {};
      \node[left=1pt of y] {\small $y$}; 
      \node[right=1pt of x] {\small $x$};

      \draw[very thick,side by side=color1:color3] (x) -- (y);
      \draw (u1) edge[bend right] (x);
      \draw (x) -- (u5);
      \draw (u4) edge[bend right] (y);
      \draw (y) -- (u2);

      \begin{scope}[shift={(4.5,0)}]
        \node[circle, inner sep=1.3cm] (B) at (0,0) {};

        \foreach \i in {1,...,6}{
          \node[circle, draw = black, inner sep = 1pt] (u\i) at
          ($(180-\i*60:1.2cm)$) {};
          \node at ($(180-\i*60:1.5cm)$) {\small $u_\i$};
        }

        \foreach \i/\j in {3/4,5/6}{
          \draw[very thick,color3] (u\i) -- (u\j);
        }
        \draw (u1) -- (u2);
        \foreach \i/\j in {2/3,4/5,6/1}{
          \draw[very thick,color1] (u\i) -- (u\j);
        }
        
        \node [circle, draw = black, inner sep = 1pt] (y) at (0,.5) {};
        \node [circle, draw = black, inner sep = 1pt] (x) at (0,-.5) {};
        \node[left=1pt of y] {\small $y$}; 
        \node[right=1pt of x] {\small $x$};

        \draw[very thick,color1] (x) -- (y);
        \draw[very thick,color3] (u1) edge[bend right] (x);
        \draw (x) -- (u5);
        \draw (u4) edge[bend right] (y);
        \draw[very thick, color3] (y) -- (u2);
      \end{scope}

      \begin{scope}[shift={(9,0)}]
        \node[circle, inner sep=1.3cm] (C) at (0,0) {};

        \foreach \i in {1,...,6}{
          \node[circle, draw = black, inner sep = 1pt] (u\i) at
          ($(180-\i*60:1.2cm)$) {};
          \node at ($(180-\i*60:1.5cm)$) {\small $u_\i$};
        }

        \foreach \i/\j in {1/2,3/4}{
          \draw (u\i) -- (u\j);
        }
        \draw[very thick,color3] (u5) -- (u6);
        \foreach \i/\j in {4/5,6/1}{
          \draw[very thick,color1] (u\i) -- (u\j);
        }
        \draw[very thick,side by side=color3:color1] (u2) -- (u3);
        
        \node [circle, draw = black, inner sep = 1pt] (y) at (0,.5) {};
        \node [circle, draw = black, inner sep = 1pt] (x) at (0,-.5) {};
        \node[left=1pt of y] {\small $y$}; 
        \node[right=1pt of x] {\small $x$};

        \draw[very thick,color1] (x) -- (y);
        \draw[very thick,color3] (u1) edge[bend right] (x);
        \draw (x) -- (u5);
        \draw[very thick, color3] (u4) edge[bend right] (y);
        \draw (y) -- (u2);
      \end{scope}

      \begin{scope}[shift={(2.25,-4)}]
        \node[circle, inner sep=1.3cm] (D) at (0,0) {};

        \foreach \i in {1,...,6}{
          \node[circle, draw = black, inner sep = 1pt] (u\i) at
          ($(180-\i*60:1.2cm)$) {};
          \node at ($(180-\i*60:1.5cm)$) {\small $u_\i$};
        }

        \foreach \i/\j in {1/2,3/4,5/6}{
          \draw (u\i) -- (u\j);
        }
        \foreach \i/\j in {2/3,6/1}{
          \draw[very thick,side by side=color3:color1] (u\i) -- (u\j);
        }
        \draw[very thick,color1] (u4) -- (u5);
        
        \node [circle, draw = black, inner sep = 1pt] (y) at (0,.5) {};
        \node [circle, draw = black, inner sep = 1pt] (x) at (0,-.5) {};
        \node[left=1pt of y] {\small $y$}; 
        \node[right=1pt of x] {\small $x$};

        \draw[very thick,color1] (x) -- (y);
        \draw (u1) edge[bend right] (x);
        \draw[very thick, color3] (x) -- (u5);
        \draw (u4) edge[bend right] (y);
        \draw (y) -- (u2);
      \end{scope}

      \begin{scope}[shift={(6.75,-4)}]
        \node[circle, inner sep=1.3cm] (E) at (0,0) {};

        \foreach \i in {1,...,6}{
          \node[circle, draw = black, inner sep = 1pt] (u\i) at
          ($(180-\i*60:1.2cm)$) {};
          \node at ($(180-\i*60:1.5cm)$) {\small $u_\i$};
        }

        \foreach \i/\j in {1/2,3/4,5/6}{
          \draw (u\i) -- (u\j);
        }
        \foreach \i/\j in {2/3,4/5,6/1}{
          \draw[very thick,side by side=color3:color1] (u\i) -- (u\j);
        }
        
        \node [circle, draw = black, inner sep = 1pt] (y) at (0,.5) {};
        \node [circle, draw = black, inner sep = 1pt] (x) at (0,-.5) {};
        \node[left=1pt of y] {\small $y$}; 
        \node[right=1pt of x] {\small $x$};

        \draw[very thick,side by side=color1:color3] (x) -- (y);
        \draw (u1) edge[bend right] (x);
        \draw (x) -- (u5);
        \draw (u4) edge[bend right] (y);
        \draw (y) -- (u2);
      \end{scope}
      \draw[->] (A) -> (B);
      \draw[->] (B) -> (C);
      \draw[->] (C) -> (D);
      \draw[->] (D) -> (E);
    \end{tikzpicture}
    \caption{The reconfiguration sequence of Claim~\ref{cl:2Switch}. $M$ is
      represented in blue and $M'$ in red.}\label{fig:2Switch}
  \end{figure}
  \end{poc}

  \paragraph{Bipartite case}
  For $i \in \{1,2,3\}$, consider $E_i$ the set of edges of $M \cap M'$ with
  opposite endpoints adjacent to $u_i$ and $u_{i+3}$ respectively. Let
  $V'_i = V_i \setminus V(C)$ where $V_i$ is the half of the bipartition that
  contains $u_i$. Let $A_i$ be the set of vertices of $V'_i$ that are \emph{not}
  matched in $M \cap M'$ to a neighbour of $u_i$. We have
  $|A_i| = n-3 - (\deg(u_{i}) - 2) = n -1 - \deg(u_{i})$, because $C$ contains no even
  cycle and $G$ is bipartite so the only neighbours of $u_{i}$ in $V(C)$ are
  $u_{i+1}$ and $u_{i-1 \pmod 6}$. Let $B_i = V'_i \setminus N(u_{i+3})$. By the
  same reasoning, we also have $|B_i| = n - 1 - \deg(u_{i+3})$. As $u_i$ and $u_{i+3}$ are non-adjacent because $C$ has no even chords, $|A_i \cup B_i| = 2n - 2 - \deg(u_i) - \deg(u_{i+3}) < 2n/3-2$. So for every $i \in [3]$, we have
  \begin{align*}
      |E_i| &= |V'_i \setminus (A_i \cup B_i)|
       >  (n-3) - (2n/3 -2)
       > n/3 - 1.
  \end{align*}
  As $|M \cap M'| = n-3$, by the pigeonhole principle, there exists an edge
  $xy$ of $M \cap M'$ that belongs to at least two of the $E_i$, say without loss of generality
  $E_1$ and $E_2$. Namely, $x$ is adjacent to $u_1$ and $u_5$, and $y$ to $u_2$
  and $u_4$, which concludes the proof of the bipartite case by
  Claim~\ref{cl:2Switch}.

  \paragraph{Non-bipartite case}
  For $i \in [6]$, let
  \[E_i = \{(x,y) \colon xy \in M \cap M', x \in N(u_i), y \in N(u_{i+3 \pmod
    6})\}.\] Let $A_i$ be the set of vertices of
  $V' = V(G) \setminus \{u_1, \dots u_6\}$ that are \emph{not} matched to a
  neighbour of $u_i$. We have $|A_i| \le n-6 - (\deg(u_i) -4) \le n-2 - \deg(u_i)$ because $C$
  contains no even chord. Letting
  $B_i = V' \setminus N(u_{i+3 \pmod 6})$, we also have $|B_i| \le n-2 - \deg(u_{i+3\pmod 6})$. So for every $i \in [6]$, we have
  \begin{align*}
    |E_i| &= |V' \setminus (A_i \cup B_i)|
     \ge |V'| - (|A_i|+|B_i)\\
    & \ge n-6 - (2n- 4 - \deg(u_i) - \deg(u_{i+3 \pmod 6})) \\
    & > n/3 -2
  \end{align*}
  
  Consider the weight function $w$ over all pairs $(x,y)$ such that
  $xy \in M \cap M'$ defined as follows: $w(x,y)$ is the number of integers $i \in [6]$
  such that $x$ is adjacent to $u_i$ and $y$ to $u_{i +3 \pmod 6}$. By double counting, we have
  $$\sum_{(x,y) \colon xy \in M \cap M'} w(x,y) = \sum_{i \in [6]} |E_i| >
  6(n/3-2)= 2(n-6).$$ Since there are $n-6$ such couples $(x,y)$, there exists $(x,y)$
  with $w(x,y) > 2$. So without loss of generality, $x$ is adjacent to $u_1$ and
  $u_5$ and $y$ to $u_2$ and $u_4$, which concludes the proof by
  Claim~\ref{cl:2Switch}.
\end{proof}

\section{Lower bounds}\label{sec:lower}

We first give in Subsection~\ref{ssec:generic} general constructions yielding lower bounds on the
thawing, clustering, giant component and connectedness thresholds for any
$k$. In Subsection~\ref{ssec:precise} we give constructions yielding sharper bounds for
the connectedness thresholds when $k \in \{2,3\}$, as well as lower bounds for
the connectedness threshold of regular balanced bipartite graph.

\subsection{Lower bounds for generic $k$}\label{ssec:generic}

We first define two families of graphs, one composed of balanced bipartite
graphs, the other of general graphs, that we will use for several of our lower
bounds. For every $\gamma \in (0,1]$, for every integers $k \ge 2$, $p \ge 1$ and
$n$ such that $\gamma n/2$ is integral and at least  $p(k+1)$, let
$G_{k,p,\gamma,n}$ be the $n$-vertex graph constructed as follows. We partition
the vertex set $V(G_{k,p,\gamma,n})$ into three sets $X \sqcup Y \sqcup Z$ with
$|X| = 2p(k+1)$, $|Y| = \gamma n/2 - p(k+1)$ and
$|Z| = (1- \gamma/2)n - p(k+1)$. Note that these quantities are integral because
we assumed that $\gamma n/2$ is integral so that the graph could have a matching
of size exactly $\gamma n/2$. We add the edges such that the set $X$ induces a collection of $p$
vertex-disjoint cycles of length $2(k+1)$, $Y$ and $Z$ are independent sets, and
all edges between $Y$ and $X\cup Z$ are present (see
Figure~\ref{fig:non_connectedness}).
  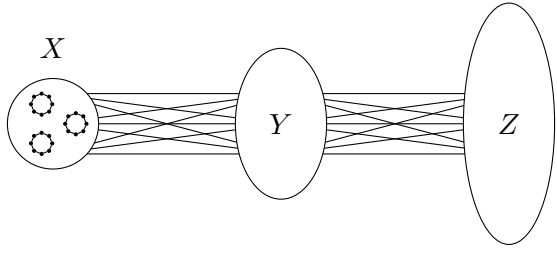
\begin{figure}[ht!]
    \centering        
    \begin{tikzpicture}[scale =2]
      \foreach \i in {1,2,3}{
        \foreach \x in {1,...,3}{
          \node (a\i\x) at (-3+1.5*\i,-.4+.2*\x) {};
        }
      }

      \foreach \x in {1,...,3}{
        \foreach \y in {1,...,3}{
          \draw (a1\x) -- (a2\y) -- (a3\x);
        }
      }

      \draw[black,fill=white] (-1.5,0) circle (.3cm); 
      \draw[black,fill=white] (0,0) ellipse (.3cm and .5cm); 
      \draw[black,fill=white] (1.5,0) ellipse (.3cm and .8cm);

      \node at (1.5,0) {$Z$};
      \node at (0,0) {$Y$};
      \node at (-1.5,.5) {$X$};

      \foreach \i in {1,2,3}{
        \foreach \j in {1,...,8}{
          \node[circle, fill=black, draw = black, inner sep = .4pt] (c\i\j) at ($(-1.5,0)+(\i*120:.15cm)+(\j*45:.07cm)$) {};
        }
        \foreach \j in {1,...,8}{
          \pgfmathsetmacro{\jj}{int(Mod(\j,8)+1)}
          \draw (c\i\j) -- (c\i\jj);
        }
      }
    \end{tikzpicture}
    \caption{The graph $G_{k,p,\gamma,n}$, here with $p=3$ and $k=3$.}\label{fig:non_connectedness}
  \end{figure}

\begin{lemma}\label{lem:non_connectedness}
  Let $\gamma \in (0,1]$ and $k, p, n$ be integers with $k \ge 2$, $p\ge 1$ and
  $\gamma n/2 \ge p(k+1)$ an integer. The graph $G_{k,p,\gamma,n}$ has minimum degree
  $\gamma n/2-p(k+1)$ and the $k$-switch graph $\mathcal{H}_k(G_{k,p,\gamma,n}, \gamma)$ of $\gamma n$-matchings has exactly $2^p$ connected components, which have equal
  size. Moreover, the restrictions of the $\gamma n$-matchings to $X$ are constant
  within each of these connected components, i.e. $p(k+1)$ frozen edges.
\end{lemma}
\begin{proof}
  The vertices in $Y$ have degree $(1-\gamma/2)n + p(k+1) \ge n/2$, the vertices
  of $X$ have degree larger than those in $Z$, which have degree
  $\gamma n/2 - p(k+1)$, so $\delta(G_{k,p,\gamma,n}) = \gamma n/2 - p(k+1)$.

  Let $M$ be a $\gamma n$-matching of $G_{k,p,\gamma,n}$. The number of edges in
  $M$ is $\gamma n/2= |Y| + |X|/2$. Since $Z$ and $Y$ are independent sets and $N(Z) = Y$, $M$ has $|Y|$ edges connecting $Y$ to
  $Z$ and $|X|/2$ edges within $X$. In other words, $M$ induces a perfect
  matching on $X$, and all vertices of $Y$ are matched to a vertex of $Z$. As
  $G_{k,p,\gamma,n}[X]$ is a collection of $p$ cycles of length $2(k+1)$, it has
  exactly $2^p$ perfect matchings. Any two of these matchings differ on at least
  one cycle, that is on $2(k+1)$ edges, which proves that $G_{k,p,\gamma,n}[X]$
  (and respectively $G_{k,p,\gamma,n}$) has at least $2^p$ equivalence classes
  of $\gamma n$-matchings under $k$-switches. As $G_{k,p,\gamma,n}[Y \cup Z]$ is
  a bipartite complete graph, all matchings of $G_{k,p,\gamma,n}[Y\cup Z]$ saturating $Y$ are
  equivalent under $2$-switches. Therefore, the equivalence classes of the
  $\gamma n$-matchings of $G_{k,p,\gamma,n}$ under $k$-switches are in one-to-one
  correspondence with the restriction of the matchings to $X$, which proves that
  the number of connected components of $\mathcal{H}_k(G_{k,p,\gamma,n}, \gamma)$ is exactly $2^p$, and that the restrictions of the $\gamma n$-matchings to $X$ are constant within each of the connected components, and finally that the connected components all have equal size.
\end{proof}

We now construct a balanced bipartite graph similar to $G_{k,p,\gamma,n}$. For
every $\gamma \in (0,1]$, for every integer $k \ge 2$, $p \ge 1$ and $n$ such that
$\gamma n$ is integral and at least $p(k+1)$, let
$G_{k,p,\gamma,n}^{(\mathcal B)}$ be the balanced bipartite graph on $2n$
vertices constructed as follows. Write $V_1 \sqcup V_2$ for the bipartition of
$V(G_{k,p,\gamma,n}^{(\mathcal B)})$. For each $i$, we partition $V_i$ in three
sets $X_i \sqcup Y_i \sqcup Z_i$ with $|X_i| = p(k+1)$,
$|Y_1| = \lceil \gamma n/2 - p(k+1)/2 \rceil$ and
$|Z_1| = \lfloor (1- \gamma/2)n - p(k+1)/2 \rfloor$, and $Y_2$ and $Z_2$ of size
like $Y_1$ and $Z_1$ with the floors and ceilings permuted.  For each $i \in [2]$
we add all the edges between $Y_i$ and $X_{3-i} \cup Z_{3-i}$. Finally,
$G_{k,p,\gamma,n}^{(\mathcal B)}[X_i \cup X_2]$ is a disjoint collection $p$
cycles of length $2(k+1)$ alternating between $X_1$ and $X_2$ (see
Figure~\ref{fig:non_connectedness_bip}).
  \begin{figure}[ht!]
    \centering        
    \begin{tikzpicture}[scale =2]
      \foreach \i in {1,...,3}{
        \foreach \x in {1,...,3}{
          \node (a\i\x) at (-.5+\i,-.4+.2*\x) {};
          \node (b\i\x) at (.5-\i,-.4+.2*\x) {};
        }
      }

      \foreach \x in {1,...,3}{
        \foreach \y in {1,...,3}{
          \draw (a1\x) -- (a2\y) -- (a3\x);
          \draw (b1\x) -- (b2\y) -- (b3\x);
        }
      }

      \draw[black,fill=white] (-.5,0) ellipse (.3cm and .5cm); 
      \draw[black,fill=white] (-1.5,0) ellipse (.3cm and .6cm); 
      \draw[black,fill=white] (-2.5,0) ellipse (.3cm and .8cm);

      \draw[black,fill=white] (.5,0) ellipse (.3cm and .5cm); 
      \draw[black,fill=white] (1.5,0) ellipse (.3cm and .6cm); 
      \draw[black,fill=white] (2.5,0) ellipse (.3cm and .8cm);

      \node at (.65,0) {$X_1$};
      \node at (1.5,0) {$Y_2$};
      \node at (2.5,0) {$Z_1$};
      \node at (-.65,0) {$X_2$};
      \node at (-1.5,0) {$Y_1$};
      \node at (-2.5,0) {$Z_2$};

      \foreach \i in {1,2}{
        \foreach \j in {1,...,4}{
          \node[circle, fill=black, draw = black, inner sep = .4pt] (c\i\j) at (-1.5+\i,\j*.1) {};
          \node[circle, fill=black, draw = black, inner sep = .4pt] (d\i\j) at (-1.5+\i,-\j*.1) {};
        }
      }
      \foreach \j in {1,...,4}{
        \pgfmathsetmacro{\jj}{int(Mod(\j,4)+1)}
        \draw (c1\j) -- (c2\j) -- (c1\jj);
        \draw (d1\j) -- (d2\j) -- (d1\jj);
      }

    \end{tikzpicture}
    \caption{The graph $G_{k,p,\gamma,n}^{(\mathcal{B})}$, here with $p = 2$ and $k = 3$.}\label{fig:non_connectedness_bip}
  \end{figure}
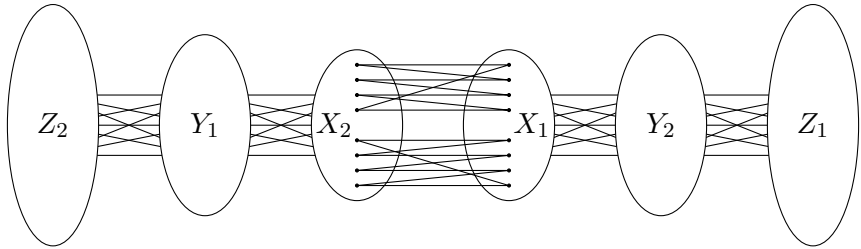

\begin{lemma}\label{lem:non_connectedness_bip}
  Let $\gamma \in (0,1]$ and $k, p, n$ be integers with $k \ge 2$, $p \ge 1$ and
  $\gamma n \ge p(k+1)$ an integer. The graph $G_{k,p,\gamma,n}^{(\mathcal{B})}$ has minimum degree
  $\lfloor (\gamma n-p(k+1))/2 \rfloor$ and under $k$-switches, the $k$-switch graph $\mathcal{H}_k(G_{k,p,\gamma,n}^{(\mathcal{B})},\gamma)$ of $\gamma n$-matchings
   has exactly $2^p$ connected components,
  which have equal size. Moreover, the restrictions of the
  $\gamma n$-matchings  to $X_1 \cup X_2$ are constant within each of these connected components, i.e. $p(k+1)$ frozen edges.
\end{lemma}

\begin{proof}
  This proof is highly similar to that of Lemma~\ref{lem:non_connectedness}. The vertices of $Y_i$ have degree at least $p(k+1)+ \lfloor(1-\gamma/2)n - p(k+1)/2\rfloor \ge n/2$, the vertices of $X_i$  at least $\lfloor \gamma n - p(k+1)/2\rfloor +2$ and the vertices of $Z_i$ at least $\lfloor \gamma n - p(k+1)/2\rfloor$, so $\delta(G_{k,p,\gamma,n}^{(\mathcal{B})}) = \lfloor \gamma n - p(k+1)/2\rfloor$.

  Let $M$ be a perfect matching of $G_{k,p,\gamma,n}^{(\mathcal{B})}$. The number of edges in $M$ is $\gamma n = |Y_1| + |Y_2| + |X_1 \cup X_2|/2$. As the vertices in $Z_1 \cup Z_2$ can only be matched to $Y_1 \cup Y_2$, $M$ induces a perfect matching on $X_1 \cup X_2$ and $M$ saturates all the vertices of $Y_1 \cup Y_2$, which are matched to some vertices of $Z_1 \cup Z_2$. As $X_1 \cup X_2$ induces a collection of $p$ cycles of length $2(k+1)$, $G_{k,p,\gamma,n}^{(\mathcal{B})}[X_1 \cup X_2]$ has exactly $2^p$ perfect matchings. Any two of these matchings differ on at least
  one cycle, that is on $2(k+1)$ edges, which proves that $G_{k,p,\gamma,n}^{(\mathcal{B})}[X_1 \cup X_2]$
  (and respectively $G_{k,p,\gamma,n}^{(\mathcal{B})}$) have at least $2^p$ equivalence classes
  of $\gamma n$-matchings under $k$-switches. As $G_{k,p,\gamma,n}^{(\mathcal{B})}[Y_1 \cup Z_2]$ and $G_{k,p,\gamma,n}^{(\mathcal{B})}[Y_2 \cup Z_1]$ are
  bipartite complete graphs, all matchings of $G_{k,p,\gamma,n}^{(\mathcal{B})}[Y_1 \cup Y_2 \cup Z_1 \cup Z_2]$ saturating $Y_1 \cup Y_2$ are
  equivalent under $2$-switches. Therefore, the connected components of $\mathcal{H}_k(G_{k,p,\gamma,n}^{(\mathcal{B})},\gamma)$ are in one-to-one
  correspondence with the restriction of the matchings to $X_1 \cup X_2$, which proves that
  the number of connected components is exactly $2^p$, and that the restrictions of the $\gamma n$-matchings to $X_1 \cup X_2$ are constant within each of the connected components, and finally that the connected components all have equal size.
\end{proof}

By adjusting the value of
$p$ in $G_{k,p,\gamma,n}$ and $G_{k,p,\gamma,n}^{(\mathcal B)}$, we obtain the following four corollaries.
\begin{corollary}\label{cor:clustering}
  For every $\eps > 0$, there exists $c>1$ such that for every $k \ge 2$, $\gamma \in (0,1]$ and $n$ such that $\gamma n/2$ is integral and at least $2(k+1)$, there exists an $n$-vertex graph $G$ with minimum degree at least $\gamma n/2 - (k+1) - \eps k n$ such that $\mathcal{H}_{k}(G,\gamma)$ has $c^n$ connected components (alternatively, $G$ is a balanced bipartite graph on $2n$ vertices, where $\gamma n$ is integral and greater than $k$, with minimum degree $\lfloor(\gamma n - (k+1))/2\rfloor - \eps k n$ and the same condition on $\mathcal{H}_{k}(G,\gamma)$).\\ 
  In other words, we have the following inequalities:
  \[
  \delta^{\mathbf{cluster}}_k(n,\gamma)\ge \frac{\gamma n}2 - k - o(kn),
  \quad \text{ and } \quad
  \delta^{\mathbf{cluster}}_{k,\mathrm{bip}}(n,\gamma)\ge\left\lfloor\frac{\gamma n - (k-1)}2\right\rfloor- o(kn).
  \]
\end{corollary}

\begin{corollary}\label{cor:freezing}
   For every $\eps > 0$, there exists $c>1$ such that for every $k \ge 2$, $\gamma \in (0,1]$ and $n$ such that $\gamma n/2$ is integral and at least $2(k+1)$, there exists an $n$-vertex graph $G$ with minimum degree at least $\gamma n/2 - \eps n$ such that each connected component of $\mathcal{H}_{k}(G,\gamma)$ has at most $\gamma n/(2c)$ non-frozen edges (alternatively, $G$ is a balanced bipartite graph on $2n$ vertices, with minimum degree $\lfloor(\gamma n - (k+1))/2\rfloor - \eps n$ and the same condition on $\mathcal{H}_{k}(G,\gamma)$).\\ 
  In other words, we have the following inequalities:
  \[
  \delta^{\mathbf{thaw}}_k(n,\gamma)\ge \frac{\gamma n}2 - k - o(n),
  \quad \text{ and } \quad
  \delta^{\mathbf{thaw}}_{k,\mathrm{bip}}(n,\gamma) \ge \left\lfloor\frac{\gamma n - (k-1)}2\right\rfloor- o(n).
  \]
\end{corollary}

\begin{corollary}\label{cor:giant}
   For every $\eps > 0$, there exists $c>1$ such that for every $k \ge 2$, $\gamma \in (0,1]$ and $n$ such that $\gamma n/2$ is integral and greater than $k$, there exists an $n$-vertex graph $G$ with minimum degree at least $\gamma n/2 - (1+\eps)(k+1)$ such that all connected components of $\mathcal{H}_{k}(G,\gamma)$ have size less than $|\mathcal{H}_{k}(G,\gamma)|/c$ (alternatively, $G$ is a balanced bipartite graph on $2n$ vertices, with minimum degree at least $\left\lfloor \frac{\gamma n - (1+\eps)(k+1)}2 \right\rfloor$ and the same condition on $\mathcal{H}_{k}(G,\gamma)$).\\ 
  In other words, we have the following inequalities:
  \[
  \delta^{\mathbf{giant}}_k(n,\gamma)\ge\frac{\gamma n}2-(1+o(1))k,  
  \quad \text{ and } \quad
  \delta^{\mathbf{giant}}_{k,\mathrm{bip}}(n,\gamma)\ge\left\lfloor\frac{\gamma n-(1+o(1))(k-1)}2\right\rfloor.
  \]
\end{corollary}

\begin{corollary}\label{cor:non_connectedness}
  For every $k \ge 2$, $\gamma \in (0,1]$ and $n$ such that $\gamma n/2$ is integral and greater than $k$, there exists an $n$-vertex graph $G$ with minimum degree at least $\gamma n/2 - (k+1)$ such that $\mathcal{H}_{k}(G,\gamma)$ is disconnected (alternatively, $G$ is a balanced bipartite graph on $2n$ vertices, with minimum degree at least $\left\lfloor \frac{\gamma n - (k+1)}2 \right\rfloor$ and the same condition on $\mathcal{H}_{k}(G,\gamma)$).\\ 
  In other words, we have the following inequalities:
  \[
  \delta^{\mathbf{connect}}_k(n,\gamma)\ge    \frac{\gamma n}2-k
  \quad \text{ and } \quad
  \delta^{\mathbf{connect}}_{k,\mathrm{bip}}(n,\gamma)
  \ge \left\lfloor \frac{\gamma n - (k-1)}2 \right\rfloor.
  \]
\end{corollary}

\begin{proof}[Proofs of
  Corollaries~\ref{cor:clustering}, \ref{cor:freezing}, \ref{cor:giant} and \ref{cor:non_connectedness}] Let $\eps >0$.
  
  For the clustering threshold, assume without loss of generality that $\eps < \frac{\gamma}{4k}$ and let $p = \lceil \eps n/2\rceil$. We have $p(k+1)\le (k+1) +\eps n(k+1)/2 \le (k+1) +\eps nk \le \gamma n/2$. Thus, $G_{k,p,\gamma,n}$ has minimum degree at least $\gamma n/2- (k+1) - \eps kn$ and by Lemma~\ref{lem:non_connectedness}, $\mathcal{H}_{k}(G_{k,p,\gamma,n},\gamma)$ has $2^p \ge c^n$ connected components, where $c=2^{\eps/2}>1$.

  For the thawing threshold, assume without loss of generality that $\eps < \gamma/4$ and let $p = \lceil \eps n/(k+1)\rceil$. We have $p(k+1) \le \eps n + k+1 \le \gamma n/2$. Thus, $G_{k,p,\gamma,n}$ has minimum degree at least $\gamma n/2- (k+1) -\eps n$ and by Lemma~\ref{lem:non_connectedness}, $\mathcal{H}_{k}(G_{k,p,\gamma,n},\gamma)$ has $2^p$ connected components, each with $p(k+1)$ frozen edges, and thus at most $\gamma n/2 - p(k+1) \le (\frac{\gamma}2 - \frac{\eps}{k+1})n \le  \frac{\gamma n}{2c}$ non-frozen edges, where $c =\frac{\gamma(k+1)}{\gamma(k+1)-2\eps}>1$.

  For the connectedness threshold, let $p = 1$. Then $G_{k,p,\gamma,n}$ has minimum degree $\gamma n/2 -(k+1)$ and by Lemma~\ref{lem:non_connectedness}, $\mathcal{H}_{k}(G_{k,p,\gamma,n},\gamma)$ has $2$ connected components, of size $|\mathcal{H}_{k}(G_{k,p,\gamma,n},\gamma)|/2$, which already handles the giant component threshold with $c=2$. 
  
  More interestingly, let $p = \lceil\eps\rceil$. We have $p(k+1) \le (1+\eps)(k+1)$. Thus, $G_{k,p,\gamma,n}$ has minimum degree $\gamma n/2 - p(k+1) > \gamma n/2- (1+\eps)(k+1)$ and by Lemma~\ref{lem:non_connectedness}, $\mathcal{H}_{k}(G_{k,p,\gamma,n},\gamma)$ has $2^p$ connected components of equal size, hence all of them have size at most $|\mathcal{H}_{k}(G_{k,p,\gamma,n},\gamma)|/c$ where $c = 2^p = 2^{\lceil \eps \rceil}$.

  We obtain the bounds for balanced bipartite graphs by the same arguments.
\end{proof}

\subsection{Precise bounds for disconnectedness and regular bipartite graphs}\label{ssec:precise}
We construct two families of graphs, one composed of balanced bipartite (regular)
graphs, the other of general graphs. For any $k \ge 2$ and any even $n$, let
$F_{k,n}$ be the graph constructed as follows.  First assume that $n$ is a
multiple of $(k+1)$ to avoid divisibility issues. We partition the vertex set
$V(F_{k,n})$ in $(k+1)$ sets $\bigsqcup_{i \in [k+1]} X_i$ such that
$|X_i| = n/(k+1)$. The sets $X_1$ and $X_{k+1}$ each induce a
complete graph, all other $X_i$ induce an independent set and
$F_{k,n}[X_i \cup X_{i+1}]$ induces a complete bipartite graph for each
$i \in [k]$ (Figure~\ref{fig:non_connected_precise}). Now, if $n$ is not a multiple of
$k+1$, we add the remaining $n \pmod{k+1}$ vertices in different sets $X_i$, starting
by $X_2$, then $X_{k}$ (if $k\neq 2$), and then $X_1$ and $X_{k+1}$. In particular, for $k = 2$ and $k=3$, $\delta(F_{k,n})$ is attained by the vertices of $X_{k+1}$, thus  $\delta(F_{2,n})= \left\lceil
  \frac{n}3 \right\rceil + \left\lfloor
  \frac{n}3 \right\rfloor-1 = \left\lfloor \frac{2n-2}3\right\rfloor$ and  $\delta(F_{3,n}) = \left\lfloor
  \frac{n-2}4 \right\rfloor + \left\lfloor
  \frac{n}4 \right\rfloor = \frac{n}2-1$ because $n$ is even.

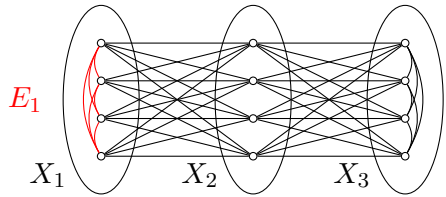
\begin{figure}[ht!]
  \centering        
  \begin{tikzpicture}
    \pgfmathsetmacro{\p}{4}
    \foreach \i in {1,2,3}{
      \draw (2*\i,1.25) ellipse (.5cm and 1.25cm);
      \foreach \x in {1,...,\p}{
        \node[circle, draw = black, inner sep = 1pt] (a\i\x) at
        (2*\i,.5*\x) {};
      }
    }

    \foreach \a/\b in {1/2,2/3}{
      \foreach \x in {1,...,\p}{
        \foreach \y in {1,...,\p}{
          \draw (a\a\x) -- (a\b\y);
        }
      }
    }
    \node at (1,1.25) {\textcolor{red}{$E_1$}};
    \node at (1.3,.25) {$X_1$};
    \node at (3.3,.25) {$X_2$};
    \node at (5.3,.25) {$X_3$};
    \foreach \x in {1,...,\p}{
      \foreach \y in {\x,...,\p}{
        \draw[red] (a1\x) edge[bend left] (a1\y);
        \draw (a3\x) edge[bend right] (a3\y);
      }
    }
  \end{tikzpicture}
  \caption{The graph $F_{2,n}$. The edges in the set $E_1$ (defined in the proof of Lemma~\ref{lem:non_connected_precise}) are drawn in red.}\label{fig:non_connected_precise}
\end{figure}
\begin{lemma}\label{lem:non_connected_precise}
  Let $k \ge 2$ and $n \ge 4(k+1)$ be an even integer. The switch graph $\mathcal H_{k}(F_{k,n})$ is
  disconnected.
\end{lemma}
\begin{proof}
  Let $E_1$ be the subset of edges with both endpoints in $X_1$ and $E_2$ the
  remaining edges. The number of matched edges of $E_1$ in a
  perfect matching is invariant under $k$-switches. Indeed, let $M$ be a perfect
  of $F_{k,n}$ and $C$ an alternating cycle of $M$ of length at most $2k$
  containing some edge of $X_1$. The distance between $X_1$ and $X_{k+1}$ is
  $k$, so $C$ does not intersect $X_{k+1}$. Hence, $C$ can be decomposed as the
  concatenation of paths $P_1,Q_1, \dots P_\ell, Q_\ell$, where the paths $P_i$
  use only edges of $E_1$ and the paths $Q_i$ use only edges of $E_2$. In
  particular, each $Q_i$ has even length, because the subgraph on the vertex set
  $V(F_{k,n}) \setminus X_{k+1}$ containing the edges $E_2$ is bipartite. So
  performing a switch on the alternating cycle $C$ preserves the number of edges
  in $E_2$ and thus the number of edges in $E_1$.

  There exists a perfect matching with exactly $p$ edges
  of $E_1$ for each $p \in \{1, \dots \lfloor n/(2(k+1)) \rfloor\}$, so
  $\mathcal H_{k}(F_{k,n})$ is disconnected because $n \ge 4(k+1)$.
\end{proof}

We now construct an analogous balanced bipartite graph. Let $k \ge 2$ and $n\ge 2(k+1)$ be two integers. Assume for now that $n$ is a
multiple of $(k+1)$. Let $F_{k,n}^{(\mathcal B)}$ be the balanced bipartite graph
on $2n$ vertices obtained by replacing each vertex of a $2(k+1)$ cycle by an
independent set of size $n/(k+1)$ (see Figure~\ref{fig:non_connected_precise}). Note
that $F_{k,n}^{(\mathcal B)}$ is $2n/(k+1)$-regular.
\begin{figure}[ht!]
  \centering        
  \begin{tikzpicture}
    \pgfmathsetmacro{\p}{3}
    \foreach \i in {1,2,3}{
      \foreach \j in {1,2}{
        \draw (2*\j, 1.5*\i) circle[radius=.5cm];
        \foreach \x in {1,...,\p}{
          \node[circle, draw = black, inner sep = 1pt] (a\j\i\x) at
          ($(2*\j,1.5*\i)+(360/\p*\x:.3cm)$) {};
        }
      }
    }

    \foreach \a/\b in {1/2,3/2,2/3,2/1}{
      \foreach \x in {1,...,\p}{
        \foreach \y in {1,...,\p}{
          \draw (a1\a\x) -- (a2\b\y);
        }
      }
    }
    \foreach \x in {1,...,\p}{
      \foreach \y in {1,...,\p}{
        \draw[red] (a11\x) -- (a21\y);
        \draw[blue] (a13\x) -- (a23\y);
      }
    }

    \node at (3,.8) {\textcolor{red}{$E_1$}};
    \node at (3,5.2) {\textcolor{blue}{$E_3$}};

  \end{tikzpicture}
  \caption{The graph $F_{2,n}^{(\mathcal R)}$, i.e. a blow-up of the vertices of a
    6-cycle into independent set of size $n/3$. The edges in the sets $E_1$ and $E_3$ (defined in the proof of Lemma~\ref{lem:blow-up_cycle}) are drawn in red and blue respectively.}\label{fig:bipartite_2n3}
\end{figure}
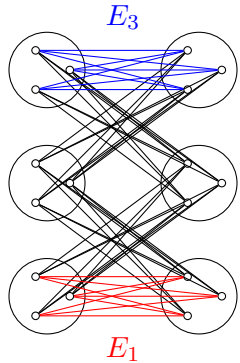

Now if $n$ is not a multiple of $(k+1)$, we construct $F_{k,n}^{(\mathcal B)}$
similarly, by distributing the additional vertices one-by-one in pairs of consecutive blown-up
vertices. The graph we obtain is not regular anymore, but still balanced
bipartite and we have $\delta(F_{2,n}) = \lfloor 2n/3 \rfloor$ and
$\delta(F_{3,n})= \lfloor n/2 \rfloor$.

\begin{lemma}\label{lem:blow-up_cycle}
  Let $k \ge 2$ and $n \ge (k+1)$ an integer. The switch graph $\mathcal H_{k}(F_{k,n}^{(\mathcal B)})$ is disconnected.
\end{lemma}
\begin{proof}
  This proof is similar to that of Lemma~\ref{lem:non_connected_precise}. Denote by $u_1, \dots, u_{2(k+1)}$ the vertices of $C_{2(k+1)}$. In $F_{k,n}^{(\mathcal B)}$, denote by $E_1$ the set of edges joining the blow-ups of  $u_1$ and $u_2$, $E_3$ the set of edges joining the blow-ups of $u_{k+1}$ and $u_{k+2}$, and $E_2$ the set of remaining edges. As in
  Lemma~\ref{lem:non_connected_precise}, the number of matched edges of $E_1$ in a
  perfect matching is invariant under $k$-switches. Indeed, any cycle of length at most $2k$ cannot simultaneously contain an edge of $E_1$ and an edge of $E_3$. Thus an alternating cycle $C$ of length at most $2k$ containing an edge of $M \cap E_1$ can be decomposed as the concatenation of paths $P_1$, $Q_1$ \dots $P_{\ell}$, $Q_\ell$ where the paths $P_i$ use only edges of $E_1$ and the paths $Q_i$ only edges of $E_2$. In particular, as the edges of $E_2$ form a bipartite graph, the paths $Q_i$ are of even length and a switch on $C$ preserves the number of edges in $E_2$ and thus the number of edges in $E_1$. For each
  $p \in \{0, \dots \lfloor n/(k+1) \rfloor\}$, there exists a perfect matching with exactly $p$
  edges of $E_1$, so $\mathcal H_{k}(F_{k,n}^{(\mathcal R)})$ is disconnected.
\end{proof}

Lemmas~\ref{lem:non_connected_precise} and~\ref{lem:blow-up_cycle} directly give us tighter
lower bounds than Lemmas~\ref{lem:non_connectedness} and~\ref{lem:non_connectedness_bip} on the
connected thresholds for $k \in \{2,3\}$, as well as a lower bound on the
connectedness threshold for balanced bipartite regular graphs.
\begin{corollary}\label{cor:connectedness_lower}
  For $n$ large enough, we have the following inequalities.
  \begin{itemize}
  \item 
  $\delta^{\mathbf{connect}}_2(n) \ge \left\lfloor \frac{2n+1}3
    \right\rfloor$.
  \item 
  $\delta^{\mathbf{connect}}_3(n) \ge \frac{n}2$.
  \item 
  $\delta^{\mathbf{connect}}_{2,\mathrm{bip}}(n) = \lfloor \frac{2n}3 \rfloor+1$.
  \item 
  $\delta^{\mathbf{connect}}_{3,\mathrm{bip}}(n) = \lfloor \frac{n}2 \rfloor+1$.
  \item 
  For all $k \ge 2$, if $(k+1)$ divides $n$, then $\delta^{\mathbf{connect}}_{k,\mathrm{reg}}(n) \ge 2\frac{n}{k+1}+1$.
  \end{itemize}
\end{corollary}

\section{Random sampling}\label{sec:random_sampling}
Similarly as for our connectedness results, we first prove in
Subsection~\ref{ssec:mixing4} that the $4$-switch Markov chain mixes polynomially, before
deriving from it the polynomial mixing of the $2$-switch and $3$-switch Markov
chain in Subsection~\ref{ssec:mixing23}. 

\subsection{Polynomial mixing time of the 4-switch Markov chain}\label{ssec:mixing4}
In the rest of this section, $G$ is either an $n$-vertex graph or a balanced
bipartite graph on $2n$ vertices. Consider the following Markov chain $\Gamma$
on the perfect matchings of $G$. Given a perfect matching $M$ of $G$:
\begin{enumerate}
\item Select independently and uniformly at random four vertices $u_1,u_2, u_3$
  and $u_4$, (if $G$ is balanced bipartite, select all them from the same side of the
  bipartition). Let $u_iv_i$ be the edge of $M$ for each $i$.
\item If $u_l = u_1$ for some $l \in \{2,3,4\}$, let $l$ be the smallest such
  index and let \[C = (u_1, v_1, u_2, v_2 \dots u_{l-1}, v_{l-1}).\] Otherwise, let
  $C$ be the sequence $(u_1, v_1, u_2, v_2, u_3, v_3, u_4, v_4)$.
\item If $C$ is a simple cycle, set $M' \leftarrow M \Delta C$ with probability
  $1/2$ and $M' \leftarrow M$ with probability $1/2$.
\end{enumerate}

Note that the probability of performing a $4$-switch on a fixed cycle $C$ of
length $\ell \in \{4,6,8\}$ is $\Theta(1/n^{\ell/2})$. Moreover, the transition
matrix of $\Gamma$ is symmetric, so $\Gamma$ admits the uniform distribution on
the space of perfect matchings of $G$ as a stationary distribution. The
probability of not performing any switch is at least $1/2$ because of laziness
of the third step, hence $\Gamma$ is also aperiodic. Hence, by
Theorem~\ref{thm:4Switch}, $\Gamma$ is ergodic and converges towards the uniform
distribution on perfect matchings if $G$ has sufficiently high minimum degree:

\begin{theorem}\label{thm:mixing}
  Let $G$ be an $n$-vertex graph in which the degrees of each pair of non-adjacent
  vertices sum up to at least $n+2$. 
  Alternatively, let $G$ be a balanced bipartite graph on $2n$ vertices in which
  each pair of non-adjacent vertices in different halves of the bipartition  have
  their degrees that sum up to at least $n+1$.\\
  The Markov chain $\Gamma$ on the perfect matchings of $G$ converges towards
  the uniform distribution on the set of perfect matchings of $G$. It has mixing
  time
  $$\tau_{\mix} \le O(n^8\log(n)).$$
\end{theorem}

The proof of the mixing time of Theorem~\ref{thm:mixing} is a textbook application of
the canonical paths method. The canonical paths we will use are close to the
reconfiguration sequences used to prove Theorem~\ref{thm:4Switch}, but operate on the
symmetric difference between the endpoints of the path following a precise
order.
\begin{proof}

  To bound the congestion of our Markov chain $\Gamma$, we will define a
  canonical path $\gamma_{S,T}$ connecting each pair of perfect matchings $S$
  and $T$. Consider an arbitrary order on the vertices of $G$. Given $S$ and $T$
  two perfect matchings of $G$, our canonical path between $S$ and $T$ is
  defined as follows. Let $C_1, \dots C_p$ be the alternating cycles in
  $S \Delta T$, such that for each $i$, the smallest vertex in $C_i$ is smaller
  than the smallest vertex in $C_{i+1}$. Let $M_1, \dots M_{p+1}$ be the perfect
  matchings of $G$ such that $M_i$ is equal to $T$ on all cycles $C_j$ with
  $j < i$, equal to $S$ on all cycles $C_j$ with $j \ge i$ and equal to $S$ (and
  $T$) around all remaining vertices. Our canonical path will go from $S=M_1$ to
  $T=M_{p+1}$ and will visit all $M_i$ in increasing order. Thus we only need to
  explain how to go from $M_i$ to $M_{i+1}$.

  \subsubsection*{Connecting $M_i$ to $M_{i+1}$}
  Let $v_1, v_2, \dots v_{2q}$ be the consecutive vertices of $C_i$, such that
  $v_1$ is the smallest vertex and $v_1v_2 \in M_i$.
  \begin{claim}\label{cl:connecting_Mi}
    There exists a sequence $M_i= N_0, \dots N_q= M_{i+1}$ of
    perfect matchings of $G$ such that for all $j$, $N_{j-1}$ and $N_j$ differ
    by a $4$-switch and one of the following conditions is verified:
    \begin{itemize}
    \item There exists $k \in [q]$ such that $N_j$ is equal to $M_i$ on
      $G \setminus C_i$ and contains the edges
      $v_1v_{2k}, v_{2k+1}v_{2k+2}, \dots v_{2q-1}v_{2q}$ and
      $v_2v_3, \dots v_{2k-2}v_{2k-1}$. In other words, $N_j$ and $M_i$ are
      equal on all vertices but $\{v_1, \dots v_{2k}\}$, where $N_j$ contains
      $v_1v_{2k}$ and is equal to $M_{i+1}$ on all other vertices. We will say
      that $N_j$ is \defin{representative} and call $k$ its \defin{progress index}.
    \item There exists $k \in [q]$ and $xy \in N_{j-1}$ such that $N_j$ is equal
      to $M_i$ on $G \setminus (C_i \cup \{x,y\})$ and contains the edges
      $v_1x, yv_{2k}, v_{2k+1}v_{2k+2}, \dots v_{2q-1}v_{2q}$ and
      $v_2v_3, \dots v_{2k-2}v_{2k-1}\}$. In other words, $N_j$ and $M_i$ are
      equal on all vertices but $\{v_1, \dots v_{2k}, x, y\}$, where $N_j$
      contains $v_1x$ and $yv_{2k}$ and is equal to $M_{i+1}$ on all other
      vertices. Moreover, neither $v_{2k}$ nor $v_{2k+2}$ are adjacent to
      $v_1$. We will call such $N_j$ \defin{misleading}, $k$ its \defin{progress
        index} $xy$ the \defin{deceptive} edge.
    \end{itemize}
  \end{claim}
  \begin{poc}
    We build the sequence $N_0, \dots N_q$ inductively. Suppose we already
    constructed $N_0, \dots N_{j-1}$. We distinguish several cases, depending on
    whether $N_{j-1}$ and the matching $N_j$ we construct are representative or
    misleading.

    \paragraph{Case 1} First, assume that $N_{j-1}$ is representative, let $k$ be the progress
    index of $N_{j-1}$ and assume that there exists some $l \in [3]$ such that $v_{2(k+l)}$ is a
    neighbour of $v_1$. Then let $l$ be the largest such integer. Let $N_{j}$ be the
    perfect matching obtained by performing the $4$-switch on
    $C=v_1v_{2k}v_{2k+1}\dots v_{2(k+l)}$. The vertices affected by this $4$-switch
    all belong to $C_i$. Since $N_{j-1}$ was representative and $N_{j}$ contains the
    edges $v_1v_{2(k+l)}$, $v_{2(k+l)+1}v_{2(k+l)+2}$, \dots $v_{2q-1}v_{2q}$
    and $v_2v_3$, \dots $v_{2(k+l)-2}v_{2(k+l)-1}$, $N_j$ is also representative with
    progress index $k+l$ (see Figure~\ref{fig:FF}). 
    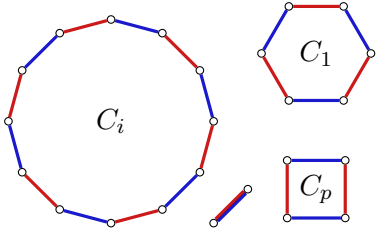
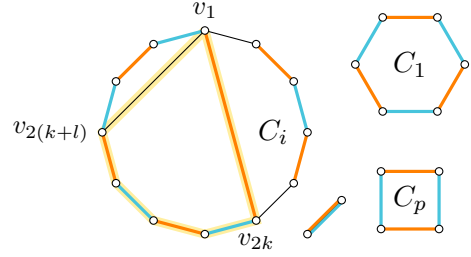
\begin{figure}[ht!]
      \centering
      \begin{subfigure}[t]{.45\linewidth}
        \centering
        \begin{tikzpicture}[scale = .9]
          \clip (-1.8,-1.8) rectangle (4,2);
          \foreach \i in {1, ..., 12}{
            \node[circle, draw = black, inner sep = 1pt] (v\i) at ($(120-\i*30:1.5cm)$) {};
          }
          \foreach \i in {1, ...,6}{
            \node[circle, draw = black, inner sep = 1pt] (u\i) at ($(3,1) +(120-\i*60:.8cm)$) {};
          }

          \foreach \i in {1, ...,4}{
            \node[circle, draw = black, inner sep = 1pt] (w\i) at ($(3,-1) +(45-\i*90:.6cm)$) {};
          }

          \node[circle, draw = black, inner sep = 1pt] (a) at (2,-1) {};
          \node[circle, draw = black, inner sep = 1pt] (b) at (1.5,-1.5) {};

          \foreach \i in {1,..., 6}{
            \pgfmathsetmacro{\I}{int(Mod(2*\i-1,12))}
            \pgfmathsetmacro{\II}{int(Mod(2*\i-1,12)+1)}
            \pgfmathsetmacro{\III}{int(Mod(2*\i,12)+1)}
            \draw[very thick, color3] (v\I) -- (v\II);
            \draw[very thick, color1] (v\II) -- (v\III);
            }

            \foreach \i in {1,2,3}{
              \pgfmathsetmacro{\I}{int(Mod(2*\i-1,6))}
              \pgfmathsetmacro{\II}{int(Mod(2*\i-1,6)+1)}
              \pgfmathsetmacro{\III}{int(Mod(2*\i,6)+1)}
              \draw[very thick, color3] (u\I) -- (u\II);
              \draw[very thick, color1] (u\II) -- (u\III);
            }

            \foreach \i in {1,2}{
              \pgfmathsetmacro{\I}{int(Mod(2*\i-1,4))}
              \pgfmathsetmacro{\II}{int(Mod(2*\i-1,4)+1)}
              \pgfmathsetmacro{\III}{int(Mod(2*\i,4)+1)}
              \draw[very thick, color3] (w\I) -- (w\II);
              \draw[very thick, color1] (w\II) -- (w\III);
            }
            \draw[very thick,side by side=color3:color1] (a) -- (b);

            \node (C1) at (3,1) {$C_1$};
            \node (C1) at (0,0) {$C_i$};
            \node (C1) at (3,-1) {$C_p$};

          \end{tikzpicture}
          \caption{The perfect matchings $S$, and $T$ are represented in
            blue and red respectively.}
      \end{subfigure}
      \hfill
      \begin{subfigure}[t]{.5\linewidth}
        \centering
        \begin{tikzpicture}[scale = .9]
          \clip (-2.9,-1.8) rectangle (4,2);
          \foreach \i in {1, ..., 12}{
            \coordinate (v\i) at ($(120-\i*30:1.5cm)$);
          }
          \foreach \i in {1, ...,6}{
            \coordinate (u\i) at ($(3,1) +(120-\i*60:.8cm)$);
          }

          \foreach \i in {1, ...,4}{
            \coordinate (w\i) at ($(3,-1) +(45-\i*90:.6cm)$);
          }

          \coordinate (a) at (2,-1);
          \coordinate (b) at (1.5,-1.5);

          \node (C1) at (3,1) {$C_1$};
          \node (C1) at (1,0) {$C_i$};
          \node (C1) at (3,-1) {$C_p$};

          \draw[line width=.12cm, opacity = .5, \colora] (v1) -- (v6) -- (v7) -- (v8)
          -- (v9) -- (v10) -- (v1);
          \foreach \i in {4,..., 6}{
            \pgfmathsetmacro{\I}{int(Mod(2*\i-1,12))}
            \pgfmathsetmacro{\II}{int(Mod(2*\i-1,12)+1)}
            \pgfmathsetmacro{\III}{int(Mod(2*\i,12)+1)}
            \draw[very thick, \colord] (v\I) -- (v\II);
            \draw[very thick, \colorb] (v\II) -- (v\III);
          }

          \draw (v2) -- (v1) -- (v10) (v5) -- (v6) ;
          \draw[very thick, \colord] (v1) -- (v6) (v2) -- (v3) (v4) -- (v5) ;
          \draw[very thick, \colorb] (v4) -- (v3) (v6) -- (v7) ;

          \foreach \i in {1,2,3}{
            \pgfmathsetmacro{\I}{int(Mod(2*\i-1,6))}
            \pgfmathsetmacro{\II}{int(Mod(2*\i-1,6)+1)}
            \pgfmathsetmacro{\III}{int(Mod(2*\i,6)+1)}
            \draw[very thick, \colorb] (u\I) -- (u\II);
            \draw[very thick, \colord] (u\II) -- (u\III);
          }

          \foreach \i in {1,2}{
            \pgfmathsetmacro{\I}{int(Mod(2*\i-1,4))}
            \pgfmathsetmacro{\II}{int(Mod(2*\i-1,4)+1)}
            \pgfmathsetmacro{\III}{int(Mod(2*\i,4)+1)}
            \draw[very thick, \colord] (w\I) -- (w\II);
            \draw[very thick, \colorb] (w\II) -- (w\III);
          }
          \draw[very thick,side by side=\colorb:\colord] (a) -- (b);

          \foreach \i in {1,...,12}{
            \node[circle, draw = black, fill= white, inner sep = 1pt] at (v\i) {};
          }
          \foreach \i in {1,...,6}{
            \node [circle, draw = black, fill= white, inner sep = 1pt] at (u\i) {};
          }
          \foreach \i in {1,...,4}{
            \node [circle, draw = black, fill= white, inner sep = 1pt] at (w\i) {};
          }
          \foreach \i in {a,b}{
            \node [circle, draw = black, fill= white, inner sep = 1pt] at (\i) {};
          }

          \node[above=1pt of v1] {\small $v_1$};
          \node[below=1pt of v6] {\small $v_{2k}$};
          \node[left=1pt of v10] {\small $v_{2(k+l)}$};
        \end{tikzpicture}
        \caption{The matchings $N_{j-1}$ and $L$ (defined in Claim~\ref{cl:analysis}), represented in orange and
          light blue respectively. $N_j$ is obtained by performing a $4$-switch on
          $C$, which is here highlighted in yellow.}  
      \end{subfigure}
      \caption{Constructing $N_{j}$ when $N_{j-1}$ is representative and $v_1$ is a
        neighbour of $v_{2(k+l)}$ for some $l \in [3]$ (here $l =2$). The
        resulting $N_{j}$ is also representative.}\label{fig:FF}
    \end{figure}

    \paragraph{Case 2}
    Assume now no such $l$ exists, but $N_{j-1}$ is still representative. Let
    $A$ be the subset of vertices of
    $V'=V(G) \setminus \{v_1, v_{2k}, v_{2k+1},\dots v_{2k+4}\}$ matched in
    $N_{j-1}$ to a neighbour of $v_1$. Since $v_{2k+2}$ and $v_{2k+4}$ are not
    neighbours of $v_1$, we have $|A| \ge \deg(v_1) - 3$\footnote{If $G$ is
      balanced bipartite, we set
      $V' = V_1 \setminus \{v_1, v_{2k+1}, v_{2k+3}\}$ where $V_1$ is the half
      of the bipartition containing $v_1$. We then have $|A| \ge
      \deg(v_1)-1$. We also have
      $|B| = |N(v_{2k+4}) \cap V'| \ge \deg(v_{2k+4}) -3$ because $v_1$ and
      $v_{2k+4}$ are not adjacent. Since $|V'| = n-3$ and $|A| + |B| \ge n-2$,
      by the pigeonhole principle, $A$ and $B$ intersect.}. Likewise, $v_{2k+4}$ is
    not a neighbour of $v_1$ so $B = N(v_{2k+4}) \cap V'$ has size at least
    $\deg(v_{2k+4})-4$.  We have $|V'| = n-6$ and $|A| + |B| \ge n-5$ because
    $v_1$ and $v_{2k+4}$ are not adjacent. So by the pigeonhole principle, there
    exists $x'y' \in N_{j-1}$, with
    $x', y' \notin \{v_1, v_{2k}, \dots v_{2k+4}\}$, such that $x'$ is a
    neighbour of $v_1$ and $y'$ of $v_{2k+4}$. By performing the $4$-switch
    $C= v_1v_{2k}\dots v_{2k+4}y'x'$ in $N_{j-1}$, one obtains the misleading
    matching $N_j$ with the deceptive edge $x'y'$ and progress index $k+2$ (see
    Figure~\ref{fig:FU}). Moreover, note that since $l$ did not exist, neither
    $v_{2k+4}$ nor $v_{2k+6}$ is adjacent to $v_1$.
    \begin{figure}[ht!]  
      \centering
      \begin{subfigure}[t]{.45\linewidth}
        \centering
        \begin{tikzpicture}[scale = .9]
          \clip (-1.8,-1.8) rectangle (4,2);
          \foreach \i in {1, ..., 12}{
            \node[circle, draw = black, inner sep = 1pt] (v\i) at ($(120-\i*30:1.5cm)$) {};
          }
          \foreach \i in {1, ...,6}{
            \node[circle, draw = black, inner sep = 1pt] (u\i) at ($(3,1) +(120-\i*60:.8cm)$) {};
          }

          \foreach \i in {1, ...,4}{
            \node[circle, draw = black, inner sep = 1pt] (w\i) at ($(3,-1) +(45-\i*90:.6cm)$) {};
          }

          \node[circle, draw = black, inner sep = 1pt] (a) at (2,-1) {};
          \node[circle, draw = black, inner sep = 1pt] (b) at (1.5,-1.5) {};

          \node[circle, draw = black, inner sep = 1pt] (x') at (-.5,1) {};
          \node[circle, draw = black, inner sep = 1pt] (y') at (-1,.5) {};

          \foreach \i in {1,..., 6}{
            \pgfmathsetmacro{\I}{int(Mod(2*\i-1,12))}
            \pgfmathsetmacro{\II}{int(Mod(2*\i-1,12)+1)}
            \pgfmathsetmacro{\III}{int(Mod(2*\i,12)+1)}
            \draw[very thick, color3] (v\I) -- (v\II);
            \draw[very thick, color1] (v\II) -- (v\III);
            }

            \foreach \i in {1,2,3}{
              \pgfmathsetmacro{\I}{int(Mod(2*\i-1,6))}
              \pgfmathsetmacro{\II}{int(Mod(2*\i-1,6)+1)}
              \pgfmathsetmacro{\III}{int(Mod(2*\i,6)+1)}
              \draw[very thick, color3] (u\I) -- (u\II);
              \draw[very thick, color1] (u\II) -- (u\III);
            }

            \foreach \i in {1,2}{
              \pgfmathsetmacro{\I}{int(Mod(2*\i-1,4))}
              \pgfmathsetmacro{\II}{int(Mod(2*\i-1,4)+1)}
              \pgfmathsetmacro{\III}{int(Mod(2*\i,4)+1)}
              \draw[very thick, color3] (w\I) -- (w\II);
              \draw[very thick, color1] (w\II) -- (w\III);
            }
            \draw[very thick,side by side=color3:color1] (a) -- (b);

            \node (C1) at (3,1) {$C_1$};
            \node (C1) at (0,0) {$C_i$};
            \node (C1) at (3,-1) {$C_p$};

            \draw[very thick, color3] (x') -- (y');
            \draw[very thick, color1] (x') -- (-.2,1);
            \draw[very thick, color1] (y') -- (-1,.2);

          \end{tikzpicture}
          \caption{The perfect matchings $S$, and $T$ are represented in
            blue and red respectively.}
      \end{subfigure}
      \hfill
      \begin{subfigure}[t]{.5\linewidth}
        \centering
        \begin{tikzpicture}[scale = .9]
          \clip (-2.9,-1.8) rectangle (4,2);
          \foreach \i in {1, ..., 12}{
            \coordinate (v\i) at ($(120-\i*30:1.5cm)$);
          }
          \foreach \i in {1, ...,6}{
            \coordinate (u\i) at ($(3,1) +(120-\i*60:.8cm)$);
          }

          \foreach \i in {1, ...,4}{
            \coordinate (w\i) at ($(3,-1) +(45-\i*90:.6cm)$);
          }

          \coordinate (a) at (2,-1);
          \coordinate (b) at (1.5,-1.5);

          \coordinate (x') at (-.5,1);
          \coordinate (y') at (-1,.5);

          \node (C1) at (3,1) {$C_1$};
          \node (C1) at (1,0) {$C_i$};
          \node (C1) at (3,-1) {$C_p$};

          \draw[line width=.12cm, opacity = .5, \colora] (v1) -- (v6) -- (v7) -- (v8)
          -- (v9) -- (v10) -- (v1);
          \foreach \i in {4,..., 6}{
            \pgfmathsetmacro{\I}{int(Mod(2*\i-1,12))}
            \pgfmathsetmacro{\II}{int(Mod(2*\i-1,12)+1)}
            \pgfmathsetmacro{\III}{int(Mod(2*\i,12)+1)}
            \draw[very thick, \colord] (v\I) -- (v\II);
            \draw[very thick, \colorb] (v\II) -- (v\III);
          }

          \draw (v2) -- (v1) -- (v10) (v5) -- (v6) ;
          \draw[very thick, \colord] (v1) -- (v6) (v2) -- (v3) (v4) -- (v5) ;
          \draw[very thick, \colorb] (v4) -- (v3) (v6) -- (v7) ;

          \draw[very thick, \colord] (x') -- (y');
          \draw[very thick, \colorb] (x') -- (-.2,1);
          \draw[very thick, \colorb] (y') -- (-1,.2);

          \foreach \i in {1,2,3}{
            \pgfmathsetmacro{\I}{int(Mod(2*\i-1,6))}
            \pgfmathsetmacro{\II}{int(Mod(2*\i-1,6)+1)}
            \pgfmathsetmacro{\III}{int(Mod(2*\i,6)+1)}
            \draw[very thick, \colorb] (u\I) -- (u\II);
            \draw[very thick, \colord] (u\II) -- (u\III);
          }

          \foreach \i in {1,2}{
            \pgfmathsetmacro{\I}{int(Mod(2*\i-1,4))}
            \pgfmathsetmacro{\II}{int(Mod(2*\i-1,4)+1)}
            \pgfmathsetmacro{\III}{int(Mod(2*\i,4)+1)}
            \draw[very thick, \colord] (w\I) -- (w\II);
            \draw[very thick, \colorb] (w\II) -- (w\III);
          }
          \draw[very thick,side by side=\colorb:\colord] (a) -- (b);

          \foreach \i in {1,...,12}{
            \node[circle, draw = black, fill= white, inner sep = 1pt] at (v\i) {};
          }
          \foreach \i in {1,...,6}{
            \node [circle, draw = black, fill= white, inner sep = 1pt] at (u\i) {};
          }
          \foreach \i in {1,...,4}{
            \node [circle, draw = black, fill= white, inner sep = 1pt] at (w\i) {};
          }
          \foreach \i in {a,b,x',y'}{
            \node [circle, draw = black, fill= white, inner sep = 1pt] at (\i) {};
          }

          \node[above=1pt of v1] {\small $v_1$};
          \node[below=1pt of v6] {\small $v_{2k}$};
          \node[left=1pt of v10] {\small $v_{2(k+2)}$};
          \node[below right=0pt of x'] {\small $x'$};
          \node[below right=0pt of y'] {\small $y'$};

        \end{tikzpicture}
        \caption{The matchings $N_{j-1}$ and $L$ (defined in Claim~\ref{cl:analysis}), represented in orange and
          light blue respectively. $N_j$ is obtained by performing a $4$-switch on
          $C$, which is here highlighted in yellow.}  
      \end{subfigure}
      \caption{Constructing $N_{j}$ when $N_{j-1}$ is representative and $v_1$ has no
        neighbours among $v_{2(k+1)}$, $v_{2(k+2)}$ and $v_{2(k+3)}$. The resulting $N_{j}$ is misleading.}\label{fig:FU}
    \end{figure}
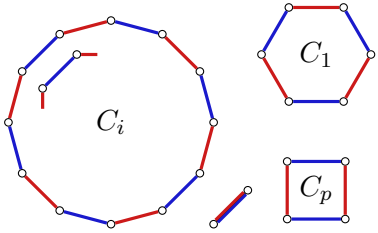
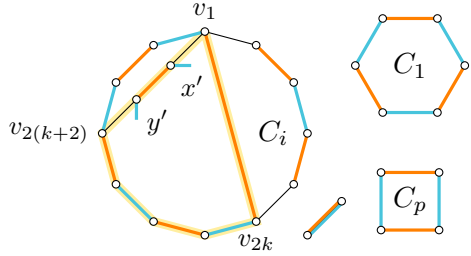

    \paragraph{Setup for the remaining cases}
    Now, assume that $N_{j-1}$ is misleading. Let $xy$ be the deceptive edge of
    $N_{j-1}$ and $k$ the progress index of $N_{j-1}$. We distinguish cases
    depending on the value of $(x,y)$: the most generic case is when
    $\{x,y\} \cap \{v_{2k+1}, \dots v_{2k+4}\} = \emptyset$, otherwise
    $\{x,y\} = \{v_{2k+1},v_{2k+2}\}$ or $\{x,y\} = \{v_{2k+3},v_{2k+4}\}$
    because $xy$ is an edge of $N_{j-2}$. Note that we cannot have
    $(x,y) = (v_{2k+2}, v_{2k+1})$ because $N_{j-1}$ is deceptive, so $v_{2k+2}$
    and $v_1$ are not adjacent. We handle the case
    $(x,y) = (v_{2k+4}, v_{2k+3})$ separately, and the other cases all at once
    by defining an the alternating path $P$ of length 5, going from $v_1$ to
    $v_{2k+l}$ for some $l \in \{1,2\}$, that we will use later to define the
    appropriate switch:
   \begin{description}
   \item[Case 3] If $(x,y) = (v_{2k+4}, v_{2k+3})$, then since
     $yv_{2k} = v_{2k+3}v_{2k}$ and $v_{2k+1}v_{2k+2}$ are edges of $N_{j-1}$,
     by performing the $2$-switch $v_{2k+3}v_{2k}v_{2k+1}v_{2k+2}$, one obtains a
     matching the representative matching $N_j$ with progress index equal to
     $k+2$.
   \item[Case 4] If $(x,y) = (v_{2k+3}, v_{2k+4})$, let
     $P = v_1v_{2k+3}v_{2k+2}v_{2k+1}v_{2k}v_{2k+4}$ and $l=2$. Note that this case
     cannot occur if $G$ is bipartite because $v_1$ and $v_{2k+3}$ are then in the
     same half of the bipartition.
   \item[Case 5] If $(x,y) = (v_{2k+1}, v_{2k+2})$, let
     $P = v_1v_{2k+1}v_{2k}v_{2k+2}v_{2k+3}v_{2k+4}$ and $l=2$. Note again that this
     case cannot occur if $G$ is bipartite because $v_1$ and $v_{2k+1}$ are then in the
     same half of the bipartition.
    \item[Case 6] Otherwise, we are in the general case that $\{x,y\}
    \cap \{v_{2k+1}, \dots v_{2k+4}\}$. Therefore we let $P = v_1xyv_{2k}\dots v_{2k+2}$ and
    $l=1$.
    \end{description}

    As Case 3 is already settled, we assume from now on that
    $(x,y) \neq (v_{2k+4}, v_{2k+3})$. We handle Cases 4, 5 and 6
    simultaneously. In each of these cases, one can check that $P$ is an
    alternating path of $N_{j-1}$. We distinguish two subcases, depending on
    whether there exists some $m \in \{0,1\}$ such that $v_{2(k+l+m)}$ is a
    neighbour of $v_1$.

    \paragraph{Cases 4a, 5a and 6a} Assume that there exists some
    $m \in \{0,1\}$ such that $v_{2(k+l+m)}$ is a neighbour of $v_1$, and let
    $m$ be the largest such integer. Let $C$ be the alternating cycle obtained
    by concatenating $P$ with $v_{2(k+l)} \dots v_{2(k+l+m)}v_1$.  The cycle $C$
    has length six or eight because $P$ hes length five. Let $N_{j}$ be the
    perfect matching obtained by performing the $4$-switch on $C$ (see for example
    Figure~\ref{fig:UF} for Case 6a, with $P = v_1xyv_{2k}\dots v_{2k+4}$ and $l = 2$
    and $m=0$). The vertices affected by this $4$-switch all belong to
    $C_i \cup \{x,y\}$. If we had $(x,y) = (v_{2k+3}, v_{2k+4})$ or
    $(x,y) = (v_{2k+1}, v_{2k+2})$ (that is Cases 4a and 5a respectively), $N_j$
    is representative with progress index $k+l+m$: it is equal to $M_i$ on all
    vertices but $\{v_1, \dots v_{2(k+l+m)}\}$, where $N_j$ contains
    $v_1v_{2(k+l+m)}$ and is equal to $M_{i+1}$ on all other vertices. In Case
    6a, recall that $P = v_1xyv_{2k}\dots v_{2k+2}$. Since $N_{j-1}$ was
    misleading, $xy$ belonged to $N_{j-2}$ and belongs to $N_{j}$ as well
    because $x$ and $y$ are consecutive in $P$. Thus, $N_j$ is representative
    with progress index $k+l+m$: it is equal to $M_i$ on all vertices but
    $\{v_1, \dots v_{2(k+l+m)}\}$, where $N_j$ contains $v_1v_{2(k+l+m)}$ and is
    equal to $M_{i+1}$ on all other vertices.
       \begin{figure}[ht!]
      
      \centering
      \begin{subfigure}[t]{.45\linewidth}
        \centering
        \begin{tikzpicture}[scale = .9]
          \clip (-1.8,-1.8) rectangle (4,2);
          \foreach \i in {1, ..., 12}{
            \node[circle, draw = black, inner sep = 1pt] (v\i) at ($(120-\i*30:1.5cm)$) {};
          }
          \foreach \i in {1, ...,6}{
            \node[circle, draw = black, inner sep = 1pt] (u\i) at ($(3,1) +(120-\i*60:.8cm)$) {};
          }

          \foreach \i in {1, ...,4}{
            \node[circle, draw = black, inner sep = 1pt] (w\i) at ($(3,-1) +(45-\i*90:.6cm)$) {};
          }

          \node[circle, draw = black, inner sep = 1pt] (a) at (2,-1) {};
          \node[circle, draw = black, inner sep = 1pt] (b) at (1.5,-1.5) {};

          \node[circle, draw = black, inner sep = 1pt] (x) at (.31,.31) {};
          \node[circle, draw = black, inner sep = 1pt] (y) at (.48,-.3) {};

          \foreach \i in {1,..., 6}{
            \pgfmathsetmacro{\I}{int(Mod(2*\i-1,12))}
            \pgfmathsetmacro{\II}{int(Mod(2*\i-1,12)+1)}
            \pgfmathsetmacro{\III}{int(Mod(2*\i,12)+1)}
            \draw[very thick, color3] (v\I) -- (v\II);
            \draw[very thick, color1] (v\II) -- (v\III);
            }

            \foreach \i in {1,2,3}{
              \pgfmathsetmacro{\I}{int(Mod(2*\i-1,6))}
              \pgfmathsetmacro{\II}{int(Mod(2*\i-1,6)+1)}
              \pgfmathsetmacro{\III}{int(Mod(2*\i,6)+1)}
              \draw[very thick, color3] (u\I) -- (u\II);
              \draw[very thick, color1] (u\II) -- (u\III);
            }

            \foreach \i in {1,2}{
              \pgfmathsetmacro{\I}{int(Mod(2*\i-1,4))}
              \pgfmathsetmacro{\II}{int(Mod(2*\i-1,4)+1)}
              \pgfmathsetmacro{\III}{int(Mod(2*\i,4)+1)}
              \draw[very thick, color3] (w\I) -- (w\II);
              \draw[very thick, color1] (w\II) -- (w\III);
            }
            \draw[very thick,side by side=color3:color1] (a) -- (b);

            \node (C1) at (3,1) {$C_1$};
            \node (C1) at (0,0) {$C_i$};
            \node (C1) at (3,-1) {$C_p$};

            \draw[very thick, color1] (x) -- (y);
            \draw[very thick, color3] (y) --++ (.2,-.1);
            \draw[very thick, color3] (x) --++ (.05,.23);

          \end{tikzpicture}
          \caption{The perfect matchings $S$, and $T$ are represented in
            blue and red respectively.}
      \end{subfigure}
      \hfill
      \begin{subfigure}[t]{.5\linewidth}
        \centering
        \begin{tikzpicture}[scale = .9]
          \clip (-2.9,-1.8) rectangle (4,2);
          \foreach \i in {1, ..., 12}{
            \coordinate (v\i) at ($(120-\i*30:1.5cm)$);
          }
          \foreach \i in {1, ...,6}{
            \coordinate (u\i) at ($(3,1) +(120-\i*60:.8cm)$);
          }

          \foreach \i in {1, ...,4}{
            \coordinate (w\i) at ($(3,-1) +(45-\i*90:.6cm)$);
          }

          \coordinate (a) at (2,-1);
          \coordinate (b) at (1.5,-1.5);

          \coordinate (x) at (.31,.31);
          \coordinate (y) at (.48,-.3);

          \node (C1) at (3,1) {$C_1$};
          \node (C1) at (1,0) {$C_i$};
          \node (C1) at (3,-1) {$C_p$};

          \draw[line width=.12cm, opacity = .5, \colora] (v1) -- (v6) -- (v7) --
          (v8) -- (v9) -- (v10) -- (v1);
          \foreach \i in {4,..., 6}{
            \pgfmathsetmacro{\I}{int(Mod(2*\i-1,12))}
            \pgfmathsetmacro{\II}{int(Mod(2*\i-1,12)+1)}
            \pgfmathsetmacro{\III}{int(Mod(2*\i,12)+1)}
            \draw[very thick, \colord] (v\I) -- (v\II);
            \draw[very thick, \colorb] (v\II) -- (v\III);
          }

          \draw (v2) -- (v1) -- (v10) (v5) -- (v6) ;
          \draw[very thick, \colord] (v1) -- (x) (y)--(v6) (v2) -- (v3) (v4) -- (v5) ;
          \draw[very thick, \colorb] (v4) -- (v3) (v6) -- (v7);

          \draw (x) -- (y);
          \draw[very thick, \colorb] (y) --++ (.2,-.1);
          \draw[very thick, \colorb] (x) --++ (.05,.23);

          \foreach \i in {1,2,3}{
            \pgfmathsetmacro{\I}{int(Mod(2*\i-1,6))}
            \pgfmathsetmacro{\II}{int(Mod(2*\i-1,6)+1)}
            \pgfmathsetmacro{\III}{int(Mod(2*\i,6)+1)}
            \draw[very thick, \colorb] (u\I) -- (u\II);
            \draw[very thick, \colord] (u\II) -- (u\III);
          }

          \foreach \i in {1,2}{
            \pgfmathsetmacro{\I}{int(Mod(2*\i-1,4))}
            \pgfmathsetmacro{\II}{int(Mod(2*\i-1,4)+1)}
            \pgfmathsetmacro{\III}{int(Mod(2*\i,4)+1)}
            \draw[very thick, \colord] (w\I) -- (w\II);
            \draw[very thick, \colorb] (w\II) -- (w\III);
          }
          \draw[very thick,side by side=\colorb:\colord] (a) -- (b);

          \foreach \i in {1,...,12}{
            \node[circle, draw = black, fill= white, inner sep = 1pt] at (v\i) {};
          }
          \foreach \i in {1,...,6}{
            \node [circle, draw = black, fill= white, inner sep = 1pt] at (u\i) {};
          }
          \foreach \i in {1,...,4}{
            \node [circle, draw = black, fill= white, inner sep = 1pt] at (w\i) {};
          }
          \foreach \i in {a,b,x,y}{
            \node [circle, draw = black, fill= white, inner sep = 1pt] at (\i) {};
          }

          \node[above=1pt of v1] {\small $v_1$};
          \node[below=1pt of v6] {\small $v_{2k}$};
          \node[left=1pt of v10] {\small $v_{2(k+l)}$};
          \node[left=0pt of x] {\small $x$};
          \node[left=0pt of y] {\small $y$};

        \end{tikzpicture}
        \caption{The matchings $N_{j-1}$ and $L$ (defined in Claim~\ref{cl:analysis}), represented in orange and
          light blue respectively. $N_j$ is obtained by performing a $4$-switch on
          $C$, which is here highlighted in yellow.}  
      \end{subfigure}
      \caption{Constructing $N_{j}$ when $N_{j-1}$ is misleading and $v_1$ is
        adjacent to $v_{2(k+l)}$ for some $l \in [2]$ (here $l=2$). The resulting $N_{j}$ is representative.}\label{fig:UF}
    \end{figure}
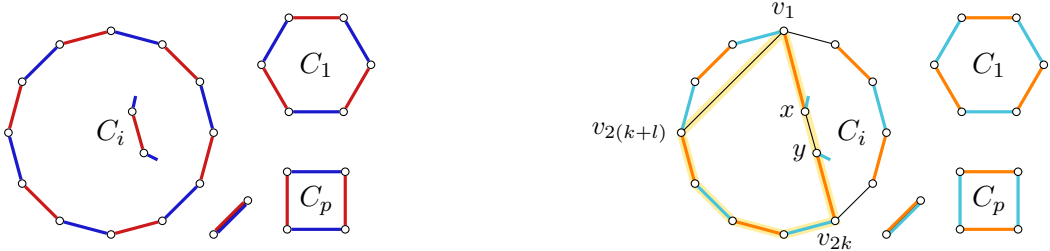

    \paragraph{Cases 4b, 5b, and 6b} Finally, assume that no such $m$
    exists. Let $V' = V(G) \setminus V(P)$ and $A$ be the set of vertices of
    $V'$ matched in $N_{j-1}$ to a neighbour of $v_1$\footnote{If $G$ is a
      balanced bipartite graph, recall that $P= v_1xyv_{2k}\dots v_{2k+2}$ and
      $l = 1$. Let $V' = V_1 \setminus \{v_1, y, v_{2k+1}\}$. We have
      $|A| \ge \deg(v_1) -1$ because $v_1$ is not adjacent to $v_{2k}$ or
      $v_{2k+2}$. Let $B = N(v_{2k+2} \cap V')$, we have
      $|B| \ge \deg(v_{2k+2}) - 2$ because $v_1$ is not adjacent to
      $v_{2k+2}$. By assumption, $|A| + |B| \ge n-2$ so by the pigeonhole principle,
      since $|V'| = n-3$, $A$ and $B$ intersect.}. Since $v_1$ is not adjacent
    to $v_{2k}$ and $v_{2(k+l)}$, it has at least $\deg(v_1) -3$ neighbours
    among $V'$, so $|A| \ge \deg(v_1)-3$. Likewise, $v_{2(k+l)}$ has at least
    $\deg(v_{2(k+l)})-4$ neighbours in $V'$ because $v_{2(k+l)}$ and $v_1$ are
    not adjacent. So $|A| + |B| \ge n-5 > |V'|$ and by the pigeonhole principle $A$
    and $B$ intersect. In other words, there exists $x'y' \in N_{j-1}$, with
    $x', y' \notin V(P)$, such that $x'$ is a neighbour of $v_1$ and $y'$ of
    $v_{2(k+l)}$. Let $C$ be the cycle of length eight obtained by concatenating
    $P$ with $v_{2(k+l)}y'x' v_1$. By performing the $4$-switch on the $C$ in
    $N_{j-1}$, one obtains the misleading matching $N_j$ with the deceptive edge
    $x'y'$ and progress index $k+l$ (see for example Figure~\ref{fig:UU} for Case 6b
    where $l=1$ and $P=v_1xyv_{2k}\dots v_{2k+2}$). Indeed, if we had
    $(x,y) = (v_{2k+1}, v_{2k+2})$ or $(x,y) = (x,y) = (v_{2k+3}, v_{2k+4})$
    (that is Cases 4b and 5b respectively), then recall that
    $P = v_1v_{2k+1}v_{2k}v_{2k+2}v_{2k+3}v_{2k+4}$ and
    $v_1v_{2k+3}v_{2k+2}v_{2k+1}v_{2k}v_{2k+4}$ respectively, so $N_j$ is equal
    to $M_i$ on all vertices but $\{v_1, \dots v_{2(k+2)}, x', y'\}$, where
    $N_j$ contains $v_1v_{2k+2}$ and $x'y'$, and is equal to $M_{i+1}$ on all
    other vertices. Since $m$ did not exists, neither $v_{2k+4}$ nor $v_{2k+6}$
    are adjacent to $v_1$. In the remaining case, recall that
    $P = v_1xyv_{2k}\dots v_{2k+2}$. Since $N_{j-1}$ was misleading with
    deceptive edge $xy$, the edge $xy$ belonged to $N_{j-2}$ and also belongs to
    $N_j$ because $x$ and $y$ are consecutive in $P$. Thus $N_j$ is equal to
    $M_i$ on all vertices but $\{v_1, \dots v_{2(k+2)}, x', y'\}$, where $N_j$
    contains $v_1v_{2k+2}$ and $x'y'$, and is equal to $M_{i+1}$ on all other
    vertices. Finally, note that by assumption, since $m$ did not exists,
    neither $v_{2k+2}$ nor $v_{2k+4}$ are adjacent to $v_1$.
    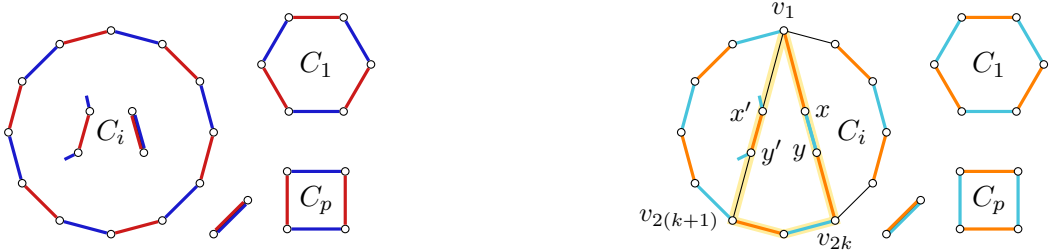
\begin{figure}[ht!]     
      \centering
      \begin{subfigure}[t]{.45\linewidth}
        \centering
        \begin{tikzpicture}[scale = .9]
          \clip (-1.8,-1.8) rectangle (4,2);
          \foreach \i in {1, ..., 12}{
            \node[circle, draw = black, inner sep = 1pt] (v\i) at ($(120-\i*30:1.5cm)$) {};
          }
          \foreach \i in {1, ...,6}{
            \node[circle, draw = black, inner sep = 1pt] (u\i) at ($(3,1) +(120-\i*60:.8cm)$) {};
          }

          \foreach \i in {1, ...,4}{
            \node[circle, draw = black, inner sep = 1pt] (w\i) at ($(3,-1) +(45-\i*90:.6cm)$) {};
          }

          \node[circle, draw = black, inner sep = 1pt] (a) at (2,-1) {};
          \node[circle, draw = black, inner sep = 1pt] (b) at (1.5,-1.5) {};

          \node[circle, draw = black, inner sep = 1pt] (x') at (-.31,.31) {};
          \node[circle, draw = black, inner sep = 1pt] (y') at (-.48,-.3) {};
          \node[circle, draw = black, inner sep = 1pt] (x) at (.31,.31) {};
          \node[circle, draw = black, inner sep = 1pt] (y) at (.48,-.3) {};

          \foreach \i in {1,..., 6}{
            \pgfmathsetmacro{\I}{int(Mod(2*\i-1,12))}
            \pgfmathsetmacro{\II}{int(Mod(2*\i-1,12)+1)}
            \pgfmathsetmacro{\III}{int(Mod(2*\i,12)+1)}
            \draw[very thick, color3] (v\I) -- (v\II);
            \draw[very thick, color1] (v\II) -- (v\III);
            }

            \foreach \i in {1,2,3}{
              \pgfmathsetmacro{\I}{int(Mod(2*\i-1,6))}
              \pgfmathsetmacro{\II}{int(Mod(2*\i-1,6)+1)}
              \pgfmathsetmacro{\III}{int(Mod(2*\i,6)+1)}
              \draw[very thick, color3] (u\I) -- (u\II);
              \draw[very thick, color1] (u\II) -- (u\III);
            }

            \foreach \i in {1,2}{
              \pgfmathsetmacro{\I}{int(Mod(2*\i-1,4))}
              \pgfmathsetmacro{\II}{int(Mod(2*\i-1,4)+1)}
              \pgfmathsetmacro{\III}{int(Mod(2*\i,4)+1)}
              \draw[very thick, color3] (w\I) -- (w\II);
              \draw[very thick, color1] (w\II) -- (w\III);
            }
            \draw[very thick,side by side=color3:color1] (a) -- (b);

            \node (C1) at (3,1) {$C_1$};
            \node (C1) at (0,0) {$C_i$};
            \node (C1) at (3,-1) {$C_p$};

            \draw[very thick, color1] (x') -- (y');
            \draw[very thick, color3] (y') --++ (-.2,-.1);
            \draw[very thick, color3] (x') --++ (-.05,.23);
            \draw[very thick, side by side=color3:color1] (x) -- (y);

          \end{tikzpicture}
          \caption{The perfect matchings $S$, and $T$ are represented in
            blue and red respectively.}
      \end{subfigure}
      \hfill
      \begin{subfigure}[t]{.5\linewidth}
        \centering
        \begin{tikzpicture}[scale = .9]
          \clip (-2.9,-1.8) rectangle (4,2);
          \foreach \i in {1, ..., 12}{
            \coordinate (v\i) at ($(120-\i*30:1.5cm)$);
          }
          \foreach \i in {1, ...,6}{
            \coordinate (u\i) at ($(3,1) +(120-\i*60:.8cm)$);
          }

          \foreach \i in {1, ...,4}{
            \coordinate (w\i) at ($(3,-1) +(45-\i*90:.6cm)$);
          }

          \coordinate (a) at (2,-1);
          \coordinate (b) at (1.5,-1.5);

          \coordinate (x') at (-.31,.31);
          \coordinate (y') at (-.48,-.3);
          \coordinate (x) at (.31,.31);
          \coordinate (y) at (.48,-.3);

          \node (C1) at (3,1) {$C_1$};
          \node (C1) at (1,0) {$C_i$};
          \node (C1) at (3,-1) {$C_p$};

          \draw[line width=.12cm, opacity = .5, \colora] (v1) -- (v6) -- (v7) -- (v8) -- (v1);
          \foreach \i in {4,..., 6}{
            \pgfmathsetmacro{\I}{int(Mod(2*\i-1,12))}
            \pgfmathsetmacro{\II}{int(Mod(2*\i-1,12)+1)}
            \pgfmathsetmacro{\III}{int(Mod(2*\i,12)+1)}
            \draw[very thick, \colord] (v\I) -- (v\II);
            \draw[very thick, \colorb] (v\II) -- (v\III);
          }

          \draw (v2) -- (v1) -- (v8) (v5) -- (v6) ;
          \draw[very thick, \colord] (v1) -- (x) (y)--(v6) (v2) -- (v3) (v4) -- (v5) ;
          \draw[very thick, \colorb] (v4) -- (v3) (v6) -- (v7) (x) -- (y) ;

          \draw[very thick, \colord] (x') -- (y');
          \draw[very thick, \colorb] (y') --++ (-.2,-.1);
          \draw[very thick, \colorb] (x') --++ (-.05,.23);
          \draw[very thick, \colorb] (x) -- (y);         

          \foreach \i in {1,2,3}{
            \pgfmathsetmacro{\I}{int(Mod(2*\i-1,6))}
            \pgfmathsetmacro{\II}{int(Mod(2*\i-1,6)+1)}
            \pgfmathsetmacro{\III}{int(Mod(2*\i,6)+1)}
            \draw[very thick, \colorb] (u\I) -- (u\II);
            \draw[very thick, \colord] (u\II) -- (u\III);
          }

          \foreach \i in {1,2}{
            \pgfmathsetmacro{\I}{int(Mod(2*\i-1,4))}
            \pgfmathsetmacro{\II}{int(Mod(2*\i-1,4)+1)}
            \pgfmathsetmacro{\III}{int(Mod(2*\i,4)+1)}
            \draw[very thick, \colord] (w\I) -- (w\II);
            \draw[very thick, \colorb] (w\II) -- (w\III);
          }
          \draw[very thick,side by side=\colorb:\colord] (a) -- (b);

          \foreach \i in {1,...,12}{
            \node[circle, draw = black, fill= white, inner sep = 1pt] at (v\i) {};
          }
          \foreach \i in {1,...,6}{
            \node [circle, draw = black, fill= white, inner sep = 1pt] at (u\i) {};
          }
          \foreach \i in {1,...,4}{
            \node [circle, draw = black, fill= white, inner sep = 1pt] at (w\i) {};
          }
          \foreach \i in {a,b,x,y,x',y'}{
            \node [circle, draw = black, fill= white, inner sep = 1pt] at (\i) {};
          }

          \node[above=1pt of v1] {\small $v_1$};
          \node[below=1pt of v6] {\small $v_{2k}$};
          \node[left=1pt of v8] {\small $v_{2(k+1)}$};
          \node[right=0pt of x] {\small $x$};
          \node[left=0pt of y] {\small $y$};
          \node[left=0pt of x'] {\small $x'$};
          \node[right=0pt of y'] {\small $y'$};

        \end{tikzpicture}
        \caption{The matchings $N_{j-1}$ and $L$ (defined in Claim~\ref{cl:analysis}), represented in orange and
          light blue respectively. $N_j$ is obtained by performing a $4$-switch on
          $C$, which is here highlighted in yellow.}  
      \end{subfigure}
      \caption{Constructing $N_{j}$ when $N_{j-1}$ is misleading and $v_1$ has no
        neighbours among $v_{2(k+1)}$ and $v_{2(k+2)}$. The resulting $N_{j}$ is misleading.}\label{fig:UU}
    \end{figure}
  \end{poc}

  Note that the progress index increases at each step of our reconfiguration
  sequence, so the sequence between $M_i$ and $M_{i+1}$ has length at most
  $|C_i| = |M_i \Delta M_{i+1}|$. So the canonical path between $S$ and $T$ has
  length at most $n$.

  \subsubsection*{Analysis of the congestion} 
  Given some transition $M \rightarrow M'$, let $\cp(M,M')$ denote the set of
  couple $(S,T)$ such that the canonical path between $S$ and $T$ includes the
  transition $M \rightarrow M'$. Recall that the congestion of our set of
  canonical paths is defined by
  \begin{align*}
    \varrho(\Gamma) &= \max_{M \rightarrow M'} \left\{\frac{\sum\limits_{(S,T)\in
                      \cp(M,M')} \pi(S)\pi(T)|\gamma_{S,T}|}{\pi(M) \Prob(M
                      \rightarrow M')}\right\} &&\parbox{4cm}{ where  $\pi$ is the
                                                  uniform distribution on
                                                  perfect matchings}\\
    &\le n \max_{M \rightarrow M'} \frac{\pi(M)|cp(M,M')|}{\Prob(M\rightarrow
      M')}. && \parbox{4cm}{ Because $|\gamma_{S,T}| \le n$}
  \end{align*}
  Thus we need to bound the number of canonical paths using a transition
  $M \rightarrow M'$. Given some $c \in \{1,2,3,4a,4b,5a,5b,6a,6b\}$, denote by
  $\cp_c(M,M')$ the couples $(S,T) \in \cp(M,M')$ such that in the canonical
  path from $S$ to $T$, $M'$ is constructed using Case $c$ of
  Claim~\ref{cl:connecting_Mi}.
  \begin{claim}\label{cl:analysis}
    For any fixed $c$ and transition $M \rightarrow M'$, we have
    $$\frac{|\cp_c(M,M')|}{\Prob(M,M')} = O(\pi(M)^{-1}n^6).$$
  \end{claim}
  \begin{poc}
    Let $c \in \{1,2,3,4a,4b,5a,5b,6a,6b\}$ and $M \rightarrow M'$ be a fixed
    transition and $C= M \Delta M'$ the cycle on which the $4$-switch is
    performed. We define an injective application on $\cp_c(M,M')$ to count
    it. The Case 3 is somewhat different from the others, because it is the only
    case where $v_1$ does not belong to the switch that is performed. For this
    reason, the definition of our injective application differs slightly for $c=3$.
    If $c=3$, let $f_3$ be the application that associates to any
    $(S,T) \in \cp_c(M,M')$ a tuple $(L,a,v,w)$ where $L$ is perfect or
    near-perfect matching with some additional information $a \in \{0,1\}$ and
    $v,w \in V(G)$ to be defined later. Otherwise, let $f_c$ be the application
    that associates to any $(S,T) \in \cp_c(M,M')$ a tuple $(L,a,b)$, where $L$
    is again a perfect or near-perfect matching, with some additional
    information $a \in \{0,1\}$ and $b \in [8]$ to be defined later.

    Let $(S, T) \in \cp(M,M')$. Recall that $C_1, \dots C_p$ are the alternating
    cycles in $S \Delta T$ ordered by smallest vertex and that for all $i$,
    $M_i$ is the perfect matching that is equal to $T$ on all the cycles $C_j$
    with $j < i$ and to $S$ on all the remaining vertices. Let $i$ such that
    $M \rightarrow M'$, was used when connecting $M_i$ to $M_{i+1}$ and denote
    again by $v_1, v_2, \dots v_{2q}$ the consecutive vertices of $C_i$, such that
    $v_1$ is the smallest vertex and $v_1v_2 \in M_i$. Let $k$ and $k'$ be the
    progress indices of $M$ and $M'$ respectively. Let $L$ be the matching
    defined by
    \begin{align*}
      L :=& (S \cap T) \ \cup \ \left(S \cap (\bigcup_{j <i} C_j)\right) \ \cup \  \left(T \cap (\bigcup_{j >i} C_j)\right) \\
          & \cup \ \{v_{2i+1}v_{2i+2} : 1 \le  i \le k-2\}  \ \cup \ 
            \{v_{2i}v_{2i+1}: k \le i \le q\}.
    \end{align*}
    In other words, $L$ is equal to $S$ on all vertices in
    $\bigcup_{j<i} V(C_j) \cup \{v_3, \dots v_{2k-2}\}$, to $T$ on all vertices
    in $\bigcup_{j>i} V(C_j) \cup \{v_{2k}, \dots v_{2q}, v_1\}$ and to $S$ (and
    $T$) on all vertices of $V(G) \setminus \bigcup_{j=1}^p V(C_j)$. Note that
    $L$ is a perfect matching if and only if $k = 1$ that is if $M =
    M_i$. Otherwise, $L$ is a near-perfect matching with unmatched vertices
    $v_2$ and $v_{2k-1}$.

    Let $a$ denote whether $v_2$ is greater than $v_{2k-1}$. If $c=3$, then $v_1$
    does not belong to the cycle $C = v_{2k}v_{2k+1}v_{2k+2}v_{2k+3}$ that is
    switched and let $v=v_1$ and $w = v_{2k}$. In
    all other cases, $v_1$ belongs to $C$ and we set $b$ to denote the rank
    within the order of $v_1$ among the (at most) eight vertices of $C$.

    We will now prove that $f_c$ is injective. Let $(L,a,v) = f_3(S,T)$ or
    $(L,a,b) = f_c(S,T)$ for some $(S,T) \in \cp(M,M')$. The support of the
    $4$-switch can be recovered: $C = M \Delta M'$. The vertex $v_1$ can also be
    recovered: if $c =3$ it is $v$, if $c \neq 3$ then it has rank $b$ within
    the order of the vertices of $C$. 

    We now prove that $v_2$ and the alternating cycles in $S \Delta T$ can be
    recovered as well. First assume that $L$ is perfect. Then $M = M_i$, so $M$
    (respectively $L$) is equal to $T$ (respectively $S$) on all vertices in
    $\bigcup_{j<i} V(C_j)$, to $S$ (respectively $T$) on all vertices in
    $\bigcup_{j\ge i} V(C_j)$ and to $S \cap T$) on all other vertices. In other
    words $L \Delta M = S \Delta T$ and we have recovered the set of alternating
    cycles. The vertex $v_2$ is matched to $v_1$ in $M$, so we can recover $v_2$
    as well. Thus we can assume that $L$ is not perfect. Let $N=L$ if $c \neq 3$
    and $N=L \setminus \{v_1v_{2k+4}\}$ if $c=3$ (this is well defined as we
    have already recovered $v_1$ and its neighbour in $L$ is
    $v_{2k+4}$). Consider $H = N\Delta (M \cap M')$, it is a collection of
    alternating paths and cycles. Note that the endpoints of the alternating
    paths of $H$ are the unmatched vertices of $N$ and $M \cap M'$, that is
    $A = V(C) \cup \{v_2,v_{2k-1}\}$ if $c \neq 3$ and
    $A = V(C) \cup \{v_2,v_{2k-1}, v_1, v_{2k+4}\}$ if $c=3$. The matching $N$
    is equal to $S$ on
    $\bigcup_{j<i} V(C_j) \cup \{v_2, \dots v_{2k-1}\} \setminus A$, to $T$ on
    $\bigcup_{j<i} V(C_j) \cup \{v_{2k+1}, \dots v_{2q}\} \setminus A$, and to
    $S \cap T$ on all remaining vertices of $V(G) \setminus A$. On the other
    hand, regardless of whether $M$ and $M'$ were representative or misleading
    in the sequence going from $S$ to $T$, they were both equal to $T$ on
    $\bigcup_{j<i} V(C_j) \cup \{v_2, \dots v_{2k-1}\} \setminus A$, to $S$ on
    $\bigcup_{j<i} V(C_j) \cup \{v_{2k+1}, \dots v_{2q}\} \setminus A$, and to
    $S \cap T$ on all remaining vertices of $V(G) \setminus A$. So all
    alternating cycles in $H$ are alternating cycles of $S \Delta T$ and the
    alternating paths in $H$ are subpaths of the alternating paths of
    $S \Delta T$. By identifying the labels of the vertices in $A$, the value of
    $c$ will then be enough information to recover how to combine the
    alternating paths in $H$ into alternating cycles of $S \Delta T$. Recall
    that $A = V(C) \cup \{v_2,v_{2k-1}\}$ if $c \neq 3$ and
    $A = V(C) \cup \{v_2,v_{2k-1}, v_1, v_{2k+4}\}$ if $c=3$. The vertices $v_2$ and $v_{2k-1}$ can be recovered
    using $a$ and the fact that $v_{2}$ and $v_{2k-1}$ are the only unmatched
    vertices of $L$. If $c \neq 3$, then the labels of $V(C)$ are fully
    determined by $v_1$, $M \cap C$ and $M'\cap C$. If $c=3$,
    the labels of $V(C)$ are fully determined by $M \cap C$, $M'\cap C$ and the
    value of $w = v_{2k} \in V(C)$.
    
    Now that we have recovered the alternating cycles in $S \Delta T$,
    their numbering can be recovered using the ordering on the vertices. $S$ is
    equal to $L$ on all cycles $C_j$ with $j < i$ and equal to $C_j \Delta L$ on
    all cycles with $j > i$ (and vice versa for $T$). Regarding $C_i$, $S$ and
    $T$ alternate on it and $S$ contains the edge $v_1v_2$, which completes the
    description of $S$ and $T$.

    By Lemma~\ref{lem:near-perfect}, the number of values that $f_c$ can take is
    $O(n^2\pi(M)^{-1})$. It follows that there are $O(n^2\pi(M)^{-1})$ canonical paths in
    $\cp_c(M,M')$ for $c\neq 3$ and $O(n^4\pi(M)^{-1})$ canonical paths in
    $\cp_c(M,M')$ for $c = 3$. If $c=3$, then the switch $M \to M'$ is a
    $2$-switch, so $\Prob(M \to M') = \Theta(1/n^2)$. On the other hand $\Prob(M
    \to M') = \Omega(1/n^4)$ for all $4$-switches. So for all $c$,
    \[\frac{|\cp_c(M,M')|}{\Prob(M,M')} = O(\pi(M)^{-1}n^6).\qedhere \]
  \end{poc}
  It follows directly that 
    $$\varrho(\Gamma) \le n \max_{M \rightarrow M'} \frac{\pi(M)|cp(M,M')|}{\Prob(M\rightarrow
      M')} = O(n^7).$$

  As the number of perfect matching of $G$ is $\pi(M)^{-1} \le n!/(n/2)!$, we have
  \[\tau_{\mix}(\Gamma) = O(\varrho(\Gamma)\log(\pi(M)^{-1})) = O(n^8\log(n)).\qedhere\]
\end{proof}

\subsection{Polynomial mixing time of the 2- and 3-switch Markov chains}\label{ssec:mixing23}

A direct consequence of Theorem~\ref{thm:3SwitchMatching} and
Theorem~\ref{thm:2SwitchMatching}, is that the $2$-switch and $3$-switch Markov chains also
have polynomial mixing times in these respective regimes, which implies
Theorem~\ref{thm:mixing23}.

%\mixing*

\begin{proof}[Proof of Theorem~\ref{thm:mixing23}]
  By Theorem~\ref{thm:3SwitchMatching}, $4$-switches can
  be realised by a sequence of length $O(1)$ of $3$-switches if
  $\delta(G)\ge n/2 +2$ (alternatively $\delta(G) \ge \lfloor n/2\rfloor+1$ if $G$ is balanced bipartite). Likewise, $3$-switches can be realised by sequences of length $O(1)$ of $2$-switches if $\delta(G)
  \ge \lfloor 2n/3\rfloor+1$). For each
  $4$-switch $(M \to M')$, define the canonical path between $M$ and $M'$ to be an
  arbitrary such sequence. The number of vertices involved in this sequence is
  bounded by a constant $c$, thus the number of canonical paths using any fixed
  $3$-switch or $2$-switch is bounded by $n^{O(1)}$. By Markov chain comparison,
  this proves polynomial mixing time for the random walk on $\mathcal H_3(G)$
  (respectively on $\mathcal H_2(G)$) in the aforementioned regimes.
\end{proof}

\section{Isolated matchings}\label{sec:isolated}
The celebrated Caccetta-Häggkvist conjecture \cite{caccetta1978Minimal} asserts
that every oriented $n$-vertex graph with minimum outdegree at least $n/k$
contains a directed cycle of length at most $k$. The weaker conjecture obtained
by replacing the minimum outdegree condition by having minimum semidegree
(minimum of outdegrees and indegrees over all vertices) at least $n/k$ is also
still open. The following two theorems show how each of these conjectures is related to the thresholds
$\delta^{\mathbf{no iso}}_k$ and $\delta^{\mathbf{no iso}}_{k,\mathrm{bip}}$ of
appearance of isolated vertices in the $k$-switch graph.
\begin{theorem}\label{thm:isolated_Caccetta}
  Let $k \ge 2$ and $n$ be an even integer. Let $d_k(n)$ be the largest integer such
  that there exists an $n$-vertex graph with minimum \emph{outdegree} $d_k(n)$ and no
  directed cycle of length at most $k$. Then
  $$\delta^{\mathbf{no iso}}_k(n) \le d_k(n)+2.$$
\end{theorem}

\begin{theorem}\label{thm:isolated_Caccetta_bip}
  For all $k \ge 2$, all $d$ and all $n$, there exists an oriented $n$-vertex
  graph with minimum \emph{semidegree} $d$ and no directed cycle of length at
  most $k$ if and only if there exists a bipartite graph $G$ on $2n$ vertices
  with $\delta(G) = d+1$ such that $\mathcal H_k(G)$ has an isolated vertex.
\end{theorem}

\begin{proof}[Proof of Theorem~\ref{thm:isolated_Caccetta}]
  
  Let $G$ be an $n$-vertex graph with a
  perfect matching $M$ that is an isolated vertex of $\mathcal{H}_k(G)$. Let
  $\delta$ be the minimum degree of $G$. Write
  $V(G) = \{u_1, \dots, u_{n/2}\} \cup \{v_1, \dots, v_{n/2}\}$, where $u_i$ is
  matched to $v_i$ in $M$ for each $i \in [n/2]$. For convenience, we will
  denote by $\bar x$ the vertex to which a vertex $x$ is matched to in $M$.  Let
  $D$ be the oriented graph on the same vertex set
  $\{u_1, \dots, u_{n/2}\} \cup \{v_1, \dots, v_{n/2}\}$ defined as follows. For
  each $i \in [n/2]$, for each vertex $w \in V(D) \setminus \{u_i,v_i\}$, add an
  arc from $u_i$ to $w$ if $v_i$ and $w$ are adjacent in $G$, and symmetrically
  an arc from $v_i$ to $w$ if $u_i$ and $w$ are adjacent in $G$. In other words,
  we have $E(D) = \{\bar xw \colon xw \in E(G) \setminus E(M)\}$. Let $C = x_1\dots x_p$
  be a directed cycle of $D$, we will show that $p \ge k+1$. For each
  $i \in [p]$, $x_i\bar{x_i}$ is by definition an edge of $M$, and by
  construction, $\bar{x_i}$ and $x_{i+1 \pmod p}$ are adjacent in $G$. So
  $x_1\bar{x_1}\dots x_p\bar{x_p}$ is an alternating cycle in $G$, which
  contradicts the assumption that $M$ is an isolated vertex of
  $\mathcal{H}_k(G)$ if $p\le k$. Hence $D$ has no directed cycles of length at
  most $k$, and in particular, $D$ is oriented. The outdegree
  of each vertex $x \in V(D)$ is equal to $\deg_G(\bar x)-1$, so the digraph $D$
  has minimum outdegree equal to $\delta(G) -1$. So $\delta(G) \le d_k(n) +1$
  for all $n$-vertex graphs $G$ with an isolated matching in
  $\mathcal{H}_k(G)$. Hence $\delta^{\mathbf{no iso}}_k(n) \le d_k(n)+2$.
\end{proof}

\begin{proof}[Proof of Theorem~\ref{thm:isolated_Caccetta_bip}]
  We first prove the direct implication. Let $D$ be an oriented $n$-vertex graph with
  minimum semidegree $d$ and no directed cycle of length at most $k$. Let us
  write $V(D) = \{u_1, \dots, u_n\}$. Let $G$ be the graph on the vertex set
  $\{a_1, \dots a_n\} \cup \{b_1, \dots b_n\}$ constructed as follows. For each
  $i$, connect $a_i$ to $b_i$. For each distinct $i$ and $j$, connect $a_i$ to
  $b_j$ if $u_i \to u_j \in E(D)$. The graph $G$ has minimum degree $d+1$. Let
  $M$ be the matching $\{a_ib_i \colon i \in [n]\}$. A $k$-switch on $M$ is a
  cycle $a_{i_1}b_{i_1} \dots a_{i_\ell}b_{i_\ell}$ of length $2\ell$ for some
  $\ell \le k$. But no such cycle exists as it would imply that $u_{i_1}\dots
  u_{i_\ell}$ is a directed cycle of length at most $k$ in $D$, so $M$ is an
  isolated vertex of $\mathcal H_2(D)$.

  Conversely, let $G$ be a balanced bipartite graph on $2n$ vertices with
  $\delta(G) = d+1$ and an isolated matching $M = \{a_ib_i \colon i \in [n]\}$
  in $\mathcal H_k(G)$. Let $D$ be the oriented $n$-vertex graph constructed as
  follows. Write $V(D) = \{u_1, \dots u_n\}$ and for all distinct $i$ and $j$,
  add the arc $u_i \to u_j$ if $a_ib_j \in E(G)$. The minimum semidegree of $D$
  is $d$. A directed cycle $u_{i_1} \dots u_{i\ell}$ in $D$ corresponds to the
  cycle $a_{i_1}b_{i_1} \dots a_{i_\ell}b_{i_\ell}$. Since $M$ is isolated in
  $\mathcal H_k(G)$, such cycle has length at least $2(k+1)$, so $D$ contains no
  directed cycle of length at most $k$ (in particular $D$ is oriented because
  $k \ge 2$.
\end{proof}

In the case $k=2$, we have obtained much finer control on the threshold for the possibility of an isolated vertex in the $k$-switch graph $\mathcal H_k(G)$ both when $G$ is a general graph and when $G$ is a balanced bipartite graph.

\begin{theorem}\label{thm:isolated_k=2}
  Let $G$ be $n$-vertex graph (or respectively a balanced bipartite graph on
  $2n$ vertices). If $\delta(G) \ge n/2+1$ (respectively
  $\delta(G) \ge \lfloor (n+3)/2\rfloor$), then for every perfect matching $M$
  of $G$ and every $e \in E(M)$, there exists a $2$-switch that contains $e$.\\
  On the other hand, $\mathcal H_2(G)$ may have isolated vertices if
  $\delta(G) \le n/2 -2$ (respectively $ \lfloor (n+1)/2\rfloor$). In other words,
  $$ \frac{n}2-1 \le \delta^{\mathbf{no iso}}_2(n) \le \frac{n}2+1  \quad \text{ and } \quad
  \delta^{\mathbf{no iso}}_{2,\mathrm{bip}}(n) = \left\lfloor \frac{n+3}2 \right\rfloor.$$
\end{theorem}
\begin{proof}
  We first prove the existence of such a switch if the minimum degree is sufficiently high. Let $M$ be a perfect matching of $G$ and $uv$ be an edge of $M$. Let $V' = V(G) \setminus \{u,v\}$, $A$ be the subset of vertices of
  $V'$ that are matched in $M$ to a neighbour of $u$, and $B = N(v) \cap V'$
  \footnote{If $G$ is a balanced bipartite graph, then let
    $V' = U \setminus \{u\}$, where $U$ is the half of the bipartition
    containing $u$, and let $A$ and $B$ be defined identically. Then $A$ and $B$ have
    size at least $\lfloor (n+1)/2\rfloor$, so $|A| + |B| \ge n > |V'|$ and by the
    pigeonhole principle, $A$ and $B$ intersect.}. We have
  $|A| \ge \deg(u) -1 \ge n/2$ and $|B| \ge n/2$. By the pigeonhole principle, $A$
  and $B$ intersect, so there exists a $2$-switch that contains $uv$.

  Conversely, we construct graphs with high minimum degree and an isolated
  matching. To do so, we first construct an auxiliary oriented graph $H_p$ on
  $p$ vertices with maximal semidegree
  $\delta(H_p) := \max_{u_i} \max \{\delta^+(u_i), \delta^-(u_i)\}$. The
  semidegree of an oriented graph on $p$ vertices is at most
  $\lfloor (p-1)/2 \rfloor$. This bound is attained by taking $H_p$ to be the
  oriented graph such that there is an arc from $u_i$ to $u_j$ is
  $j-i \in \{1, \dots, \lfloor (p-1)/2 \rfloor\} \pmod p$. By
  Theorem~\ref{thm:isolated_Caccetta}, there exists a balanced bipartite graph
  $G_n^{(\mathcal B)}$ on $2n$ vertices with minimum degree
  $\lfloor (n-1)/2 \rfloor +1 = \lfloor (n+1)/2\rfloor$, such that
  $\mathcal H_2(G_n^{(\mathcal B)})$ contains an isolated vertex.

  We now give the construction for general graphs. Let $n$ be an even integer and $p = n/2$
  and consider the graph $G_n$ on the vertex set
  $\{a_1, \dots a_{p}\} \cup \{b_1, \dots b_p\}$ built from $H_p$ as
  follows. For each $i$, connect $a_i$ with $b_i$. For each vertex
  $u_i \in V(H_p)$, we partition $N^+(u_i)$ into $A_i \cup B_i$ such that
  $||A_i|- |B_i||\le 1$. We connect $a_i$ to $a_j$ and $b_j$ for all
  $u_j \in A_i$, and $b_i$ to $a_j$ and $b_j$ for all $u_j \in B_i$. The graph
  $G_n$ constructed this way has minimum degree equal to
  $1 + 2\lfloor \delta^+(H_p)/2\rfloor + \delta^-(H_p)$. Moreover, the matching
  $M = \{a_ib_i \colon i \in [n/2]\}$ is isolated in $\mathcal H_2(G_n)$: since $H_p$ is
  an oriented graph, for each distinct $i$ and $j$, the subgraph of $G_n$
  induced by $\{a_i,b_i,a_j,b_j\}$ is isomorphic to a triangle with an edge
  attached to one of the vertices, so the edges $a_ib_i$ and $a_jb_j$ cannot be
  part of a $2$-switch. The minimum degree of $G_n$ is
  \begin{align*}
    \delta(G_n) &= 1 + 2 \left\lfloor \frac{p-1}4 \right \rfloor +  \left\lfloor
                  \frac{p-1}2 \right \rfloor
    = 2\left\lceil \frac{p}4 \right \rceil +  \left\lceil
      \frac{p}2 \right \rceil -2\\
    &= 2\left\lceil \frac{n}8 \right \rceil +  \left\lceil
      \frac{n}4 \right \rceil -2 \ge n/2 -2. \qedhere
  \end{align*}
\end{proof}

Note that the first part of Theorem~\ref{thm:isolated_k=2} implies that no edge of $M$ is frozen, and so 
$$\delta^{\mathbf{thaw}}_k(n) \le \frac{n}2+1 \quad \text{ and } \quad \delta^{\mathbf{thaw}}_{k,\mathrm{bip}}(n) \le \left\lfloor \frac{n+3}2\right\rfloor.$$

\section{Open questions and perspectives}\label{sec:open}

\subsection{Perfect matchings}
We highlight three questions about geometry in the space of perfect matchings.
What happens for regular balanced bipartite graphs? What is the threshold for
the appearance of a giant component in the $2$-switch graph? What is the threshold
guaranteeing no isolated vertices in the $k$-switch graph for $k \ge 3$?

\paragraph{Regular graphs}
For any $d$, the number $\phi(G)$ of perfect matching of any $d$-regular balanced
bipartite graph $G$ is exponential~\cite{Sch98,Voo79,Gur08,Csi17}:
\begin{align*}
  \Phi(G) \ge \left(\frac{(d-1)^{d-1}}{d^{d-2}}\right)^n.
\end{align*}
The ``volume'' heuristic suggests that for all $d$, $k$ and $n$, and any
$d$-regular balanced bipartite graph $G$ on $2n$ vertices, $\mathcal H_k(G)$ is
connected. This is false as for all $d$, there exists $d$-regular balanced
bipartite graphs with arbitrarily large girth. This is witnessed by a uniform
random $d$-regular graph, but also by blowing-up of the vertices of a
$2(k+1)$-cycle into independent sets of equal size. The resulting graph has a
disconnected $k$-switch graph (see Lemma~\ref{lem:blow-up_cycle}) and degree
$\frac{2n}{k+1}$, where $2n$ is the number of vertices.  We conjecture that this
blow-up maximises $\delta(G)$ among regular balanced bipartite graphs on $2n$
vertices with disconnected $k$-switch graph.
\begin{conjecture}
  There exists $c \ge 0$ such that for all $n$ and $k$,
  $$\delta^{\mathbf{connect}}_{k,\mathrm{reg}}(n) \in \left[ \frac{2n}{k+1} -c, \frac{2n}{k+1} +c\right].$$
\end{conjecture}
Note that Theorems~\ref{thm:2SwitchMatching} and \ref{thm:3SwitchMatching} confirm the
upper bound of this conjecture for $k=2$ and $k = 3$, and that
Corollary~\ref{cor:connectedness_lower} confirms the lower bound when $k+1$ divides
$n$. More generally, we believe that the switch graph of $d$-regular balanced
bipartite graphs on $2n$ vertices undergoes a phase transition around $2n/(k+1)$
that is similar to that described in this article for balanced bipartite graphs
and general graphs.

\paragraph{Giant component}
Theorem~\ref{thm:2SwitchMatching} and Corollary~\ref{cor:giant} place the threshold for the
appearance of a giant component in the $2$-switch graph between $n/2 - \eps$ and
$\lfloor (2n + 3)/3 \rfloor$. We conjecture that the giant
component threshold is near the lower bound.
\begin{conjecture}\label{conj:giant}
  For each $c \ge 1$, there exists $\eps \ge 0$ such that for all $n$, for all
  $n$-vertex graphs $G$ (or alternatively all balanced bipartite graphs on $2n$
  vertices) with minimum degree at least $n/2 + \eps$, the $2$-switch graph
  $\mathcal H_2(G)$ contains a component of size at least $|\mathcal H_2(G)|/c$.
\end{conjecture}

\paragraph{Isolated matchings}

For isolated matchings, we have seen in Section~\ref{sec:isolated} that the conjecture of Caccetta and Häggkvist about the maximal minimum outdegree of an oriented graph without directed cycles of length $k$ would directly imply, if true, an upper bound on the threshold $\delta^{\mathbf{no iso}}_k(n)$ at which the $k$-switch graph (if non-empty) is guaranteed to have positive minimum degree. A weaker version of this conjecture, in which outdegree is replaced by semidegree, is actually equivalent to the minimum degree threshold that guarantees no isolated vertices in the $k$-switch graphs of balanced bipartite graphs. This suggests that for all $k$ and $n$,
$\delta^{\mathbf{no iso}}_{k,\mathrm{bip}}(n) = \lceil n/k \rceil +1$. This motivates the following
conjecture for general graphs.

\begin{conjecture}\label{conj:isolated}
  There exists $c > 0$ such that for all $k$ and
  $n$, $$\delta^{\mathbf{no iso}}_k(n) \in [n/k-c, n/k +c].$$
\end{conjecture}

This conjecture holds for $k =2$ by Theorem~\ref{thm:isolated_k=2}.

\subsection{Other Dirac-type results}

One might naturally look beyond perfect matchings to other spanning structures that are guaranteed based on some Dirac-type condition.
Such structures have already been intensively studied with respect to similar ``phase transitions'' after the threshold guaranteeing existence, albeit for different phenomena, like supersaturation~\cite{komlos1995Proof}, packings~\cite{CKLOT16}, robustness~\cite{krivelevich2014Robust,allen2024Robust}, or, more recently, spread distributions~\cite{pham2022Toolkit,kelly2024Optimal,bastide2024Random}.

Our work together with a related work~\cite{kleer2020Hamiltonian} for Hamiltonian cycles point to an interesting breadth of problems for exploring such phase transitions in terms of the geometry of their respective reconfiguration spaces.
Given a spanning structure $(S_n)_{n\in \mathbb N}$,
for any $n$-vertex graph $G$ and integer $k$, let $\mathcal H_k(G)$ be the
reconfiguration graph of the embeddings of $S_n$, where two copies of $S_n$ in
$G$ are adjacent if they differ by at most $k$ edges. 
We propose the following. 
\begin{problem}
  For a given spanning structure $(S_n)_{n\in \mathbb N}$ with non-emptiness threshold
  $\delta_n$, when does there exists an integer $k = O(1)$ such that the geometry of $\mathcal H_k(G)$ undergoes a phase transition (as defined in Subsection~\ref{ssec:thresholds}) close to $\delta_n$?
\end{problem}

\subsection*{Acknowledgments}
The first author was partially supported by the grant OCENW.M20.009 of the Dutch Research Council (NWO) and the Gravitation Programme NETWORKS (024.002.003) of the Dutch Ministry of Education, Culture and Science (OCW).
The second author was supported by the National Science Center of Poland under grants 
UMO-2019/34/E/ST6/00443 and UMO-2023/05/Y/ST6/00079 within the WEAVE-UNISONO program.

\subsection*{Open access statement}
For the purpose of open access, a CC BY public copyright license is applied to any Author Accepted Manuscript (AAM) arising from this submission.

\nocite{*}
\bibliographystyle{abbrv}
\bibliography{biblio}

@article {Dir52,
    AUTHOR = {Dirac, G. A.},
     TITLE = {Some theorems on abstract graphs},
 journal = {Proc. Lond. Math. Soc. (3)},
  FJOURNAL = {Proceedings of the London Mathematical Society. Third Series},
    VOLUME = {2},
      YEAR = {1952},
     PAGES = {69--81},
      ISSN = {0024-6115},
   MRCLASS = {56.0X},
  MRNUMBER = {47308},
MRREVIEWER = {W. T. Tutte},
       DOI = {10.1112/plms/s3-2.1.69},
       URL = {https://doi.org/10.1112/plms/s3-2.1.69},
}

@article{huber2006Exact,
 author = {Huber, Mark},
 title = {Exact sampling from perfect matchings of dense regular bipartite graphs},
 fjournal = {Algorithmica},
 journal = {Algorithmica},
 issn = {0178-4617},
 volume = {44},
 number = {3},
 pages = {183--193},
 year = {2006},
 language = {English},
 doi = {10.1007/s00453-005-1175-9}
}

@incollection{diaconis2001Statistical,
 author = {Diaconis, Persi and Graham, Ronald and Holmes, Susan P.},
 title = {Statistical problems involving permutations with restricted positions},
 booktitle = {State of the art in probability and statistics. Festschrift for Willem R. van Zwet. Papers from the symposium, Leiden, Netherlands, March 23--26, 1999},
 isbn = {0-940600-50-1},
 pages = {195--222},
 year = {2001},
 publisher = {Beachwood, OH: IMS, Institute of Mathematical Statistics},
 language = {English},
 doi = {10.1214/lnms/1215090070}
}

@article{dyer2017Switch,
 author = {Dyer, Martin and Jerrum, Mark and M{\"u}ller, Haiko},
 title = {On the switch {Markov} chain for perfect matchings},
 fjournal = {Journal of the ACM},
 journal = {J. ACM},
 issn = {0004-5411},
 volume = {64},
 number = {2},
 pages = {33},
 note = {Id/No 12},
 year = {2017},
 language = {English},
 doi = {10.1145/2822322}
}

@inproceedings{broder1986Hard,
  title={How hard is it to marry at random? ({O}n the approximation of the permanent)},
  author={Broder, Andrei Z},
  booktitle={Proceedings of the 18th Annual {ACM} Symposium on Theory of Computing},
  pages={50--58},
  year={1986}
}

@article{jerrum1989Approximating,
 author = {Jerrum, Mark and Sinclair, Alistair},
 title = {Approximating the permanent},
 fjournal = {SIAM Journal on Computing},
 journal = {SIAM J. Comput.},
 issn = {0097-5397},
 volume = {18},
 number = {6},
 pages = {1149--1178},
 year = {1989},
 language = {English},
 doi = {10.1137/0218077}
}

@article{kleer2020Hamiltonian,
 author = {Kleer, Pieter and Patel, Viresh and Stroh, Fabian},
 title = {Switch-based {Markov} chains for sampling {Hamiltonian} cycles in dense graphs},
 fjournal = {The Electronic Journal of Combinatorics},
 journal = {Electron. J. Comb.},
 issn = {1077-8926},
 volume = {27},
 number = {4},
 pages = {research paper p4.29, 25},
 year = {2020},
 language = {English},
 doi = {10.37236/9503}
}

@article{feghali2016Brooks,
 author = {Feghali, Carl and Johnson, Matthew and Paulusma, Dani{\"e}l},
 title = {A reconfigurations analogue of {Brooks}' theorem and its consequences},
 fjournal = {Journal of Graph Theory},
 journal = {J. Graph Theory},
 issn = {0364-9024},
 volume = {83},
 number = {4},
 pages = {340--358},
 year = {2016},
 language = {English},
 doi = {10.1002/jgt.22000}
}

@article{bousquet2025Colorings,
 author = {Bousquet, Nicolas and Feuilloley, Laurent and Heinrich, Marc and Rabie, Mika{\"e}l},
 title = {Short and local transformations between {{\(({{\Delta}} +1)\)}}-colorings},
 fjournal = {Innovations in Graph Theory},
 journal = {Innov. Graph Theory},
 issn = {3050-743X},
 volume = {2},
 pages = {119--156},
 year = {2025},
 language = {English},
 doi = {10.5802/igt.8}
}

@article{diaconis1993Comparison,
 author = {Diaconis, Persi and Saloff-Coste, Laurent},
 title = {Comparison theorems for reversible {Markov} chains},
 fjournal = {The Annals of Applied Probability},
 journal = {Ann. Appl. Probab.},
 issn = {1050-5164},
 volume = {3},
 number = {3},
 pages = {696--730},
 year = {1993},
 language = {English},
 doi = {10.1214/aoap/1177005359}
}

@article{randall12000Analizing,
 author = {Randall, Dana and Tetali, Prasad},
 title = {Analyzing {Glauber} dynamics by comparison of {Markov} chains},
 fjournal = {Journal of Mathematical Physics},
 journal = {J. Math. Phys.},
 issn = {0022-2488},
 volume = {41},
 number = {3},
 pages = {1598--1615},
 year = {2000},
 language = {English},
 doi = {10.1063/1.533199}
}

@article {Bre73,
    AUTHOR = {Br\`egman, L. M.},
     TITLE = {Certain properties of nonnegative matrices and their
              permanents},
   JOURNAL = {Dokl. Akad. Nauk SSSR},
  FJOURNAL = {Doklady Akademii Nauk SSSR},
    VOLUME = {211},
      YEAR = {1973},
     PAGES = {27--30},
      ISSN = {0002-3264},
   MRCLASS = {15A15},
  MRNUMBER = {327788},
MRREVIEWER = {E. Seneta},
}

@article {CuKa09b,
    AUTHOR = {Cuckler, Bill and Kahn, Jeff},
     TITLE = {Hamiltonian cycles in {D}irac graphs},
   JOURNAL = {Combinatorica},
  FJOURNAL = {Combinatorica. An International Journal on Combinatorics and
              the Theory of Computing},
    VOLUME = {29},
      YEAR = {2009},
    NUMBER = {3},
     PAGES = {299--326},
      ISSN = {0209-9683},
   MRCLASS = {05C45 (05A15)},
  MRNUMBER = {2520274},
MRREVIEWER = {Shaohui Zhai},
       DOI = {10.1007/s00493-009-2360-2},
       URL = {https://doi.org/10.1007/s00493-009-2360-2},
}

@article {CuKa09a,
    AUTHOR = {Cuckler, Bill and Kahn, Jeff},
     TITLE = {Entropy bounds for perfect matchings and {H}amiltonian cycles},
   JOURNAL = {Combinatorica},
  FJOURNAL = {Combinatorica. An International Journal on Combinatorics and
              the Theory of Computing},
    VOLUME = {29},
      YEAR = {2009},
    NUMBER = {3},
     PAGES = {327--335},
      ISSN = {0209-9683},
   MRCLASS = {05C45 (05C70 94A17)},
  MRNUMBER = {2520275},
MRREVIEWER = {Shaohui Zhai},
       DOI = {10.1007/s00493-009-2366-9},
       URL = {https://doi.org/10.1007/s00493-009-2366-9},
}

@article {SSS03,
    AUTHOR = {S\'{a}rk\"{o}zy, G\'{a}bor N. and Selkow, Stanley M. and Szemer\'{e}di, Endre},
     TITLE = {On the number of {H}amiltonian cycles in {D}irac graphs},
   JOURNAL = {Discrete Math.},
  FJOURNAL = {Discrete Mathematics},
    VOLUME = {265},
      YEAR = {2003},
    NUMBER = {1-3},
     PAGES = {237--250},
      ISSN = {0012-365X},
   MRCLASS = {05-02 (05C45)},
  MRNUMBER = {1969376},
MRREVIEWER = {Martin Sonntag},
       DOI = {10.1016/S0012-365X(02)00582-4},
       URL = {https://doi.org/10.1016/S0012-365X(02)00582-4},
}

@article {BBP21,
    AUTHOR = {Bonamy, Marthe and Bousquet, Nicolas and Perarnau, Guillem},
     TITLE = {Frozen {$(\Delta + 1)$}-colourings of bounded degree graphs},
   JOURNAL = {Combin. Probab. Comput.},
  FJOURNAL = {Combinatorics, Probability and Computing},
    VOLUME = {30},
      YEAR = {2021},
    NUMBER = {3},
     PAGES = {330--343},
      comISSN = {0963-5483},
   MRCLASS = {05C15},
  MRNUMBER = {4247628},
MRREVIEWER = {Daqing Yang},
       DOI = {10.1017/s0963548320000139},
       comURL = {https://doi.org/10.1017/s0963548320000139},
}

@unpublished{BHK+,
author = {Buys, Pjotr and van den Heuvel, Jan and Kang, Ross J.},
title = {Condensation in {T}ur\'an's theorem},
note = {Manuscript in preparation},
}

@article {BKO25,
    AUTHOR = {Buys, Pjotr and Kang, Ross J. and Ozeki, Kenta},
     TITLE = {Reconfiguration of independent transversals},
   JOURNAL = {Random Structures Algorithms},
  FJOURNAL = {Random Structures \& Algorithms},
    VOLUME = {67},
      YEAR = {2025},
    NUMBER = {1},
     PAGES = {Paper No. e70025, 10},
      ISSN = {1042-9832},
   MRCLASS = {05D15},
  MRNUMBER = {4946484},
       DOI = {10.1002/rsa.70025},
       URL = {https://doi.org/10.1002/rsa.70025},
}

@ARTICLE{CCCHK25+,
       author = {{Cambie}, Stijn and {Cames van Batenburg}, Wouter and {Cranston}, Daniel W. and {van den Heuvel}, Jan and {Kang}, Ross J.},
        title = "{Reconfiguration of List Colourings}",
      journal = {arXiv e-prints},
         year = 2025,
        month = may,
          eid = {arXiv:2505.08020},
        pages = {arXiv:2505.08020},
          doi = {10.48550/arXiv.2505.08020},
archivePrefix = {arXiv},
       eprint = {2505.08020},
 primaryClass = {math.CO},
       adsurl = {https://ui.adsabs.harvard.edu/abs/2025arXiv250508020C}
}

@article {GaKa25,
    AUTHOR = {Galliano, J. and Kang, R. J.},
     TITLE = {Largest component in {B}oolean sublattices},
   JOURNAL = {Acta Math. Hungar.},
  FJOURNAL = {Acta Mathematica Hungarica},
    VOLUME = {176},
      YEAR = {2025},
    NUMBER = {1},
     PAGES = {183--214},
      ISSN = {0236-5294},
   MRCLASS = {05D05 (05C35 06A07)},
  MRNUMBER = {4936588},
       DOI = {10.1007/s10474-025-01536-0},
       URL = {https://doi.org/10.1007/s10474-025-01536-0},
}

@article {Csi17,
    AUTHOR = {Csikv\'{a}ri, P\'{e}ter},
     TITLE = {Lower matching conjecture, and a new proof of {S}chrijver's
              and {G}urvits's theorems},
   JOURNAL = {J. Eur. Math. Soc. (JEMS)},
  FJOURNAL = {Journal of the European Mathematical Society (JEMS)},
    VOLUME = {19},
      YEAR = {2017},
    NUMBER = {6},
     PAGES = {1811--1844},
      ISSN = {1435-9855},
   MRCLASS = {05C70},
  MRNUMBER = {3646876},
MRREVIEWER = {Qinglin Roger Yu},
       DOI = {10.4171/JEMS/706},
       URL = {https://doi.org/10.4171/JEMS/706},
}

@article {Gur08,
    AUTHOR = {Gurvits, Leonid},
     TITLE = {Van der {W}aerden/{S}chrijver-{V}aliant like conjectures and
              stable (aka hyperbolic) homogeneous polynomials: one theorem
              for all},
      NOTE = {With a corrigendum},
   JOURNAL = {Electron. J. Combin.},
  FJOURNAL = {Electronic Journal of Combinatorics},
    VOLUME = {15},
      YEAR = {2008},
    NUMBER = {1},
     PAGES = {Research Paper 66, 26},
   MRCLASS = {15A15 (05A15 05C70)},
  MRNUMBER = {2411443},
MRREVIEWER = {Peter M. Gibson},
       DOI = {10.37236/790},
       URL = {https://doi.org/10.37236/790},
}

@article {Sch98,
    AUTHOR = {Schrijver, Alexander},
     TITLE = {Counting {$1$}-factors in regular bipartite graphs},
   JOURNAL = {J. Combin. Theory Ser. B},
  FJOURNAL = {Journal of Combinatorial Theory. Series B},
    VOLUME = {72},
      YEAR = {1998},
    NUMBER = {1},
     PAGES = {122--135},
      ISSN = {0095-8956},
   MRCLASS = {05C70 (15A15)},
  MRNUMBER = {1604705},
MRREVIEWER = {Ko-Wei Lih},
       DOI = {10.1006/jctb.1997.1798},
       URL = {https://doi.org/10.1006/jctb.1997.1798},
}

@article {Voo79,
    AUTHOR = {Voorhoeve, M.},
     TITLE = {A lower bound for the permanents of certain
              {$(0,\,1)$}-matrices},
   JOURNAL = {Nederl. Akad. Wetensch. Indag. Math.},
  FJOURNAL = {Koninklijke Nederlandse Akademie van Wetenschappen.
              Indagationes Mathematicae},
    VOLUME = {41},
      YEAR = {1979},
    NUMBER = {1},
     PAGES = {83--86},
      ISSN = {0019-3577},
   MRCLASS = {15A15 (15A45)},
  MRNUMBER = {528221},
MRREVIEWER = {D. J. Hartfiel},
}

@article{kang2024Perfect,
 author = {Kang, Dong Yeap and Kelly, Tom and K{\"u}hn, Daniela and Osthus, Deryk and Pfenninger, Vincent},
 title = {Perfect matchings in random sparsifications of {Dirac} hypergraphs},
 fjournal = {Combinatorica},
 journal = {Combinatorica},
 issn = {0209-9683},
 volume = {44},
 number = {6},
 pages = {1233--1266},
 year = {2024},
 language = {English},
 doi = {10.1007/s00493-024-00116-0}
}

@article{achlioptas2011Solution,
 author = {Achlioptas, Dimitris and Coja-Oghlan, Amin and Ricci-Tersenghi, Federico},
 title = {On the solution-space geometry of random constraint satisfaction problems},
 fjournal = {Random Structures \& Algorithms},
 journal = {Random Struct. Algorithms},
 issn = {1042-9832},
 volume = {38},
 number = {3},
 pages = {251--268},
 year = {2011},
 language = {English},
 doi = {10.1002/rsa.20323}
}

@inproceedings{AcCo08,
  author       = {Dimitris Achlioptas and
                  Amin Coja{-}Oghlan},
  title        = {Algorithmic Barriers from Phase Transitions},
  booktitle    = {49th Annual {IEEE} Symposium on Foundations of Computer Science, {FOCS}
                  2008, October 25-28, 2008, Philadelphia, PA, {USA}},
  pages        = {793--802},
  publisher    = {{IEEE} Computer Society},
  year         = {2008},
  url          = {https://doi.org/10.1109/FOCS.2008.11},
  doi          = {10.1109/FOCS.2008.11},
  bibsource    = {dblp computer science bibliography, https://dblp.org}
}

@article{Krzakala2007Gibbs,
 author = {Krz{\k{a}}ka{\l}a, Florent and Montanari, Andrea and Ricci-Tersenghi, Federico and Semerjian, Guilhem and Zdeborov{\'a}, Lenka},
 title = {Gibbs states and the set of solutions of random constraint satisfaction problems},
 fjournal = {Proceedings of the National Academy of Sciences of the United States of America},
 journal = {Proc. Natl. Acad. Sci. USA},
 issn = {0027-8424},
 volume = {104},
 number = {25},
 pages = {10318--10323},
 year = {2007},
 language = {English},
 doi = {10.1073/pnas.0703685104}
}

@article{kelly2024Optimal,
 author = {Kelly, Tom and M{\"u}yesser, Alp and Pokrovskiy, Alexey},
 title = {Optimal spread for spanning subgraphs of {Dirac} hypergraphs},
 fjournal = {Journal of Combinatorial Theory. Series B},
 journal = {J. Comb. Theory, Ser. B},
 issn = {0095-8956},
 volume = {169},
 pages = {507--541},
 year = {2024},
 language = {English},
 doi = {10.1016/j.jctb.2024.08.006}
}

@article{krivelevich2014Robust,
 author = {Krivelevich, Michael and Lee, Choongbum and Sudakov, Benny},
 title = {Robust {Hamiltonicity} of {Dirac} graphs},
 fjournal = {Transactions of the American Mathematical Society},
 journal = {Trans. Am. Math. Soc.},
 issn = {0002-9947},
 volume = {366},
 number = {6},
 pages = {3095--3130},
 year = {2014},
 language = {English},
 doi = {10.1090/S0002-9947-2014-05963-1}
 }

@article{komlos1995Proof,
 author = {Koml{\'o}s, J{\'a}nos and S{\'a}rk{\"o}zy, G{\'a}bor N. and Szemer{\'e}di, Endre},
 title = {Proof of a packing conjecture of {Bollob{\'a}s}},
 fjournal = {Combinatorics, Probability and Computing},
 journal = {Comb. Probab. Comput.},
 issn = {0963-5483},
 volume = {4},
 number = {3},
 pages = {241--255},
 year = {1995},
 language = {English},
 doi = {10.1017/S0963548300001620}
}

@misc{pham2022Toolkit,
 author = {Pham, Huy-Tuan and Sah, Ashwin and Sawhney, Mehtaab and Simkin, Michael},
 title = {A {Toolkit} for {Robust} {Thresholds}},
 year = {2022},
 howpublished = {Preprint, {arXiv}:2210.03064 [math.{CO}] (2022)},
 url = {https://arxiv.org/abs/2210.03064},
 arXiv = {arXiv:2210.03064}
}

@misc{bastide2024Random,
 author = {Bastide, Paul and Legrand-Duchesne, Cl{\'e}ment and M{\"u}yesser, Alp},
 title = {Random embeddings of bounded degree trees with optimal spread},
 year = {2024},
 howpublished = {Preprint, {arXiv}:2409.06640 [math.{CO}] (2024)},
 url = {https://arxiv.org/abs/2409.06640},
 arXiv = {arXiv:2409.06640}
}

@incollection{sudakov2017Robustness,
 author = {Sudakov, Benny},
 title = {Robustness of graph properties},
 booktitle = {Surveys in Combinatorics 2017. Papers based on the 26th British Combinatorial Conference, University of Strathclyde, Glasgow, UK, July 2017},
 isbn = {978-1-108-41313-8; 978-1-108-33103-6},
 pages = {372--408},
 year = {2017},
 publisher = {Cambridge: Cambridge University Press},
 language = {English},
 doi = {10.1017/9781108332699.009}
 }

@article{balogh2012Corradi,
 author = {Balogh, J{\'o}zsef and Lee, Choongbum and Samotij, Wojciech},
 title = {Corr{\'a}di and {Hajnal}'s theorem for sparse random graphs},
 fjournal = {Combinatorics, Probability and Computing},
 journal = {Comb. Probab. Comput.},
 issn = {0963-5483},
 volume = {21},
 number = {1-2},
 pages = {23--55},
 year = {2012},
 language = {English},
 doi = {10.1017/S0963548311000642}
}

@article{allen2024Robust,
 author = {Allen, Peter and B{\"o}ttcher, Julia and Corsten, Jan and Davies, Ewan and Jenssen, Matthew and Morris, Patrick and Roberts, Barnaby and Skokan, Jozef},
 title = {A robust {Corr{\'a}di}-{Hajnal} theorem},
 fjournal = {Random Structures \& Algorithms},
 journal = {Random Struct. Algorithms},
 issn = {1042-9832},
 volume = {65},
 number = {1},
 pages = {61--130},
 year = {2024},
 language = {English},
 doi = {10.1002/rsa.21209}
 }

@article{valiant1979Complexity,
 author = {Valiant, L. G.},
 title = {The complexity of computing the permanent},
 fjournal = {Theoretical Computer Science},
 journal = {Theor. Comput. Sci.},
 issn = {0304-3975},
 volume = {8},
 pages = {189--201},
 year = {1979},
 language = {English},
 doi = {10.1016/0304-3975(79)90044-6}
}

@inproceedings{caccetta1978Minimal,
 author = {Caccetta, Louis and H{\"a}ggkvist, R.},
 title = {On minimal digraphs with given girth},
 year = {1978},
 booktitle = {Proc. 9th Southeast. {Conf}. on {Combinatorics}, Graph Theory, and Computing, {Boca} {Raton}.},
 pages = {181--187},
}

@article {BCW70,
    AUTHOR = {Behzad, Mehdi and Chartrand, Gary and Wall, Curtiss E.},
     TITLE = {On minimal regular digraphs with given girth},
   JOURNAL = {Fund. Math.},
  FJOURNAL = {Polska Akademia Nauk. Fundamenta Mathematicae},
    VOLUME = {69},
      YEAR = {1970},
     PAGES = {227--231},
      ISSN = {0016-2736,1730-6329},
   MRCLASS = {05.60},
  MRNUMBER = {285448},
MRREVIEWER = {Neil\ Robertson},
       DOI = {10.4064/fm-69-3-227-231},
       URL = {https://doi.org/10.4064/fm-69-3-227-231},
}

@article{hladky2017Counting,
 author = {Hladk{\'y}, Jan and Kr{\'a}l, Daniel and Norin, Sergey},
 title = {Counting flags in triangle-free digraphs},
 fjournal = {Combinatorica},
 journal = {Combinatorica},
 issn = {0209-9683},
 volume = {37},
 number = {1},
 pages = {49--76},
 year = {2017},
 language = {English},
 doi = {10.1007/s00493-015-2662-5}
}

@book{levin2017Markov,
 author = {Levin, David A. and Peres, Yuval and Wilmer, Elizabeth L.},
 title = {Markov chains and mixing times. {With} a chapter on ``{Coupling} from the past'' by {James} {G}. {Propp} and {David} {B}. {Wilson}.},
 edition = {2nd edition},
 isbn = {978-1-4704-2962-1; 978-1-4704-4232-3},
 year = {2017},
 publisher = {Providence, RI: American Mathematical Society (AMS)},
 language = {English}
}

@article {MWW04,
    AUTHOR = {McKay, Brendan D. and Wormald, Nicholas C. and Wysocka, Beata},
     TITLE = {Short cycles in random regular graphs},
   JOURNAL = {Electron. J. Combin.},
  FJOURNAL = {Electronic Journal of Combinatorics},
    VOLUME = {11},
      YEAR = {2004},
    NUMBER = {1},
     PAGES = {Research Paper 66, 12},
      ISSN = {1077-8926},
   MRCLASS = {05C80 (05C38)},
  MRNUMBER = {2097332},
MRREVIEWER = {Bert\ Fristedt},
       DOI = {10.37236/1819},
       URL = {https://doi.org/10.37236/1819},
}

@book{CKLOT16,
 author = {Csaba, B{\'e}la and K{\"u}hn, Daniela and Lo, Allan and Osthus, Deryk and Treglown, Andrew},
 title = {Proof of the 1-factorization and {Hamilton} decomposition conjectures},
 fseries = {Memoirs of the American Mathematical Society},
 series = {Mem. Am. Math. Soc.},
 issn = {0065-9266},
 volume = {1154},
 isbn = {978-1-4704-2025-3; 978-1-4704-3508-0},
 year = {2016},
 publisher = {Providence, RI: American Mathematical Society (AMS)},
 language = {English},
 doi = {10.1090/memo/1154},
 zbMATH = {6751812},
 Zbl = {1367.05165}
}

\end{document}